\documentclass{amsart}
\usepackage{amsmath,amssymb,mathrsfs,bbm}
\usepackage{amsthm,fancyvrb,rotating,color,epsfig}

\usepackage{graphicx}

\date{}

\def\interior{\qopname\relax o{int}}
\def\tb{\qopname\relax o{tb}}
\def\lk{\qopname\relax o{lk}}
\def\Arg{\qopname\relax o{Arg}}

\theoremstyle{theorem}
\newtheorem{theo}{Theorem}
\newtheorem{lemm}{Lemma}
\newtheorem{prop}{Proposition}
\newtheorem{conj}{Conjecture}
\newtheorem*{appr}{Approximation Principle}

\theoremstyle{remark}
\newtheorem{rema}{Remark}

\theoremstyle{definition}
\newtheorem{defi}{Definition}

\author {Ivan Dynnikov and Maxim Prasolov}
\thanks{The work is supported by the Russian Science Foundation under grant 14-50-00005 and performed in Steklov Mathematical Institute of Russian Academy of Sciences.}
\address{\noindent Steklov Mathematical Institute of Russian Academy of Sciences, 8 Gubkina Str., Moscow 119991, Russia}
\email{dynnikov@mech.math.msu.su}
\email{0x00002a@gmail.com}

\title{Rectangular diagrams of surfaces: representability}

\begin{document}
\maketitle

\begin{abstract}
We introduce a simple combinatorial way, which we call
a rectangular diagram of a surface, to represent a surface in
the three-sphere. It has a particularly nice relation to the standard
contact structure on $\mathbb S^3$ and to rectangular diagrams
of links. By using rectangular diagrams of surfaces we are going,
in particular, to develop a method to distinguish Legendrian knots.
This requires a lot of technical work of which the present
paper addresses only the first basic question: which isotopy classes
of surfaces can be represented by a rectangular
diagram. Vaguely speaking the answer is this:
there is no restriction on the isotopy class of the surface, but
there is a restriction on the rectangular diagram of the boundary
link that can arise from the presentation of the surface. The result
extends to Giroux's convex surfaces for which this restriction on
the boundary has a natural meaning.
In a subsequent paper we are going to consider transformations
of rectangular diagrams of surfaces and to study their properties.
By using the formalism of rectangular diagrams of surfaces
we also produce here an annulus in $\mathbb S^3$ that
we expect to be a counterexample to the following conjecture:
if two Legendrian knots cobound an annulus and have zero
Thurston--Bennequin numbers relative to this annulus, then they
are Legendrian isotopic.
\end{abstract}

\tableofcontents

\section{Introduction}
Rectangular diagrams of links, also known as arc-presentations and grid diagrams,
have proved to be useful tool in knot theory. It is shown in~\cite{Dyn} that
any rectangular diagram of the unknot admits a monotonic simplification to
a square by elementary moves. This was used by M.\,Lackenby to prove
a polynomial bound on the number of Reidemeister moves needed
to untangle a planar diagram of the unknot \cite{Lac}.

It is tempting to extend the monotonic simplification approach to general knots and
links, but certainly it cannot be done straightforwardly. Typically a non-trivial
link type admits more than one rectangular diagram that cannot be simplified
further by elementary moves without using stabilizations. As shown in \cite{DyPr}
classifying such diagrams is closely related to classifying Legendrian links not admitting a destabilization.

In particular, the main result of~\cite{DyPr} implies that the finiteness of
non-destabilizable Legendrian types in each topological link type (which is the matter of Question~61 in~\cite{cgh}) is equivalent
to the finiteness of the number of rectangular diagrams that represent
each given link type and cannot be monotonically simplified. Neither of these
has been established so far.

Surfaces embedded in $\mathbb S^3$, either closed or bounded by a link, have always been
one of the key instruments of the knot theory. For studying contact structures
and Legendrian links it is useful to consider surfaces that are in special
position with respect to the contact structure, so called convex surfaces.
They were introduced by E.\,Giroux in~\cite{Gi} and studied also in~\cite{col,etho,honda,kan,mas}.

Sometimes one can prove
the existence of a convex surface having certain combinatorial and
algebraic properties for one contact structure and
non-existence of such a surface for the other,
thus showing there is no contactomorphism between the structures.
This method can be viewed as a generalization
of the classical work of
Bennequin~\cite{ben} where he distinguishes a contact structure in $\mathbb R^3$
from the standard one by showing that an overtwisted disc exists for the
former and does not exist for the latter.
Indeed, an overtwisted disc can be viewed as a disc with Legendrian boundary and a closed dividing curve.

This method can also be used to distinguish Legendrian links by applying
the argument to the link complement. For instance, some Legendrian types
of iterated torus knots are distinguished in~\cite{etho05} by showing that
certain slopes on an incompressible torus in the knot complement are realized
by dividing curves for one knot and not realized for the other.

Note that whereas the problem of algorithmic comparing
topological types of two links has been solved (see \cite{mat})
no algorithm is known to compare Legendrian types of two Legendrian links
having the same topological type. Examples of pairs of Legendrian knots for which
the equivalence remains an open question start from six (!)\ crossings \cite{cho}.

In this paper we propose a simple combinatorial way to represent compact surfaces
in $\mathbb S^3$ so that the boundary is represented in the rectangular way.
The basic idea of this presentation is not quite new and is implicitly
present in the literature. In particular, it can be considered as
an instance of Kneser--Haken's normal surfaces for triangulations
of $\mathbb S^3\cong\mathbb S^1*\mathbb S^1$ obtained by the join construction from
triangulations of two circles $\mathbb S^1$.

We observe in this paper that
the discussed combinatorial approach appears
to be well adapted to Giroux's convex surfaces with Legendrian boundary.
In particular, we show that any convex surface with Legendrian boundary
is equivalent (in a certain natural sense) to a one that can be presented in the proposed way.

By using this approach we are going to distinguish some
pairs of Legendrian knots that are not distinguishable by known
methods.
Since a rectangular diagram of a surface is a simple combinatorial object, sometimes
the non-existence of a diagram with required properties can be established
by combinatorial methods. Combined with the representability result
of the present paper this may be used to show the non-existence of a convex surface with
certain properties in the complement of a Legendrian knot. If a convex surfaces
with the same properties is known to exist for another Legendrian knot the two
Legendrian types must be different.
This method is out
of the scope of the present paper and will be presented
in a subsequent paper.

We also consider here in more detail embedded convex (in Giroux's sense) annuli in $\mathbb S^3$ tangent to
the contact structure along the boundary. Such annuli are the main building blocks
of closed convex surfaces in $\mathbb S^3$. We discuss the following specific question:
are the two boundary components of such an annulus always equivalent as Legendrian knots?

This question appears to be highly non-trivial. An attempt to answer it in the positive was made
in manuscript \cite{Gos}
but a complete proof was not given. We believe that the answer is actually
negative, but our attempt to find a \emph{simple} counterexample has failed. In a number of examples that
we tried the two boundary components of the annulus appeared to be Legendrian equivalent,
and a Legendrian isotopy was easily found. Namely, the boundary components
were transformed to each other by elementary moves preserving
the complexity of the diagram (exchange moves).

We propose here an example of an annulus of the above mentioned type
whose boundary components
cannot be transformed to each other without using stabilizations.
We conjecture that
the two Legendrian knot types in our example are different.

\begin{rema}
Another use of rectangular diagrams of links comes from their relation
to Floer homologies. As C.\,Manolescu, P.\,Ozsv\'ath, and S.\,Sarkar show in \cite{mos}
the definition of the knot Floer homology for links presented in the grid form can be given in purely combinatorial
terms. Up to this writing, we are not aware of any useful interaction between
rectangular diagrams of surfaces and algebraic theories. It would be interesting to see some.
\end{rema}

The paper is organized as follows. In Section~\ref{rds-sec} we introduce the main subject
of the paper, rectangular diagrams of surfaces, and other basic related objects.
We show that every isotopy class of a surface can be presented
by a rectangular diagram, but rectangular presentations of the boundary that will arise in this way
satisfy certain restrictions. In Section~\ref{legendrian} we discuss Legendrian links and graphs and
their presentations by rectangular diagrams. They are used in the formulation and
the proof of our main result---on the representability of convex surfaces by rectangular
diagrams---which occupy Section~\ref{giroux}. At the end of Section~\ref{giroux}
we discuss the above mentioned conjecture about annuli with Legendrian boundary.

\subsection*{Acknowledgement} We are indebted to our anonymous referee for very careful reading of our paper, which has
led to many clarifications in the text.

\section{Rectangular diagrams of a surface}\label{rds-sec}
\subsection{Definitions}
For any two distinct points $x$, $y$, say, of the oriented circle $\mathbb S^1$ we denote by $[x,y]$ the
closed arc of $\mathbb S^1$ with endpoints $x$, $y$ such
that $\partial[x,y]=y-x$ if it is endowed with the orientation inherited
from $\mathbb S^1$, see Fig.~\ref{interval}.
\begin{figure}[ht]
\centerline{\includegraphics{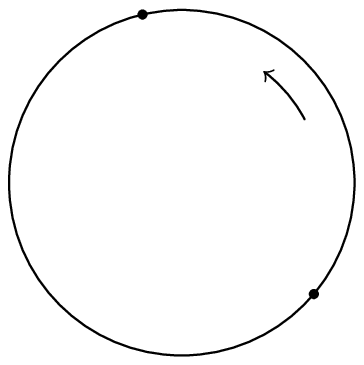}\put(-11,11){$x$}\put(-68,104){$y$}\put(-6,80){$[x,y]$}\put(-118,20){$[y,x]$}}
\caption{Intervals $[x,y]$ and $[y,x]$ on the oriented circle $\mathbb S^1$}\label{interval}
\end{figure}
By $(x,y)$ we denote the corresponding open arc: $(x,y)=[x,y]\setminus\{x,y\}$.

\begin{defi}
\emph{A rectangle} in the $2$-torus $\mathbb T^2=\mathbb S^1\times\mathbb S^1$ is
a subset of the form $[\theta_1,\theta_2]\times[\varphi_1,\varphi_2]$,
where $\theta_1\ne\theta_2$, $\varphi_1\ne\varphi_2$,
$\theta_1,\theta_2,\varphi_1,\varphi_2\in\mathbb S^1$.

Two rectangles $r_1$, $r_2$ are said to be \emph{compatible}
if their intersection satisfies one of the following:
\begin{enumerate}
\item $r_1\cap r_2$ is empty;
\item $r_1\cap r_2$ is a subset of vertices of $r_1$;
\item $r_1\cap r_2$ is a rectangle disjoint from the vertices of both rectangles $r_1$ and $r_2$.
\end{enumerate}
\emph{A rectangular diagram of a surface} is a collection $\Pi=\{r_1,\ldots,r_k\}$
of pairwise compatible rectangles in~$\mathbb T^2$ such
that every meridian $\{\theta\}\times\mathbb S^1$ and every longitude $\mathbb S^1\times\{\varphi\}$
of the torus contains at most two free vertices, where by \emph{a free vertex}
we mean a point that is a vertex of exactly one rectangle from $\Pi$.
\end{defi}

To represent such a diagram graphically, we draw $\mathbb T^2$ as a square with
identified opposite sides, so, some rectangles are cut into two or four pieces,
\begin{figure}[ht]
\centerline{\includegraphics[scale=0.9]{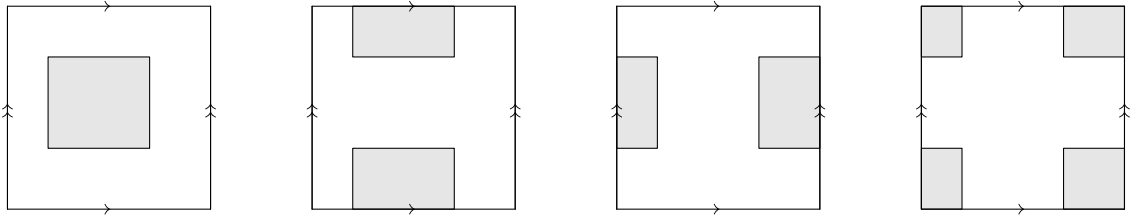}}
\caption{A rectangle in $\mathbb T^2$}\label{rectangles}
\end{figure}
see Fig.~\ref{rectangles}.

One can see that for any pair of compatible rectangles
one of the following three mutually exclusive cases occurs:
\begin{enumerate}
\item the rectangles are disjoint;
\item the rectangles share $1$, $2$, or $4$ vertices and are otherwise disjoint;
\item the rectangles have the form (possibly after exchanging them) $r_1=[\theta_1,\theta_2]\times[\varphi_1,\varphi_2]$,
$r_2=[\theta_3,\theta_4]\times[\varphi_3,\varphi_4]$ with
$$[\theta_1,\theta_2]\subset(\theta_3,\theta_4),\qquad
[\varphi_3,\varphi_4]\subset(\varphi_1,\varphi_2).$$
\end{enumerate}
\begin{figure}[ht]
\centerline{\includegraphics{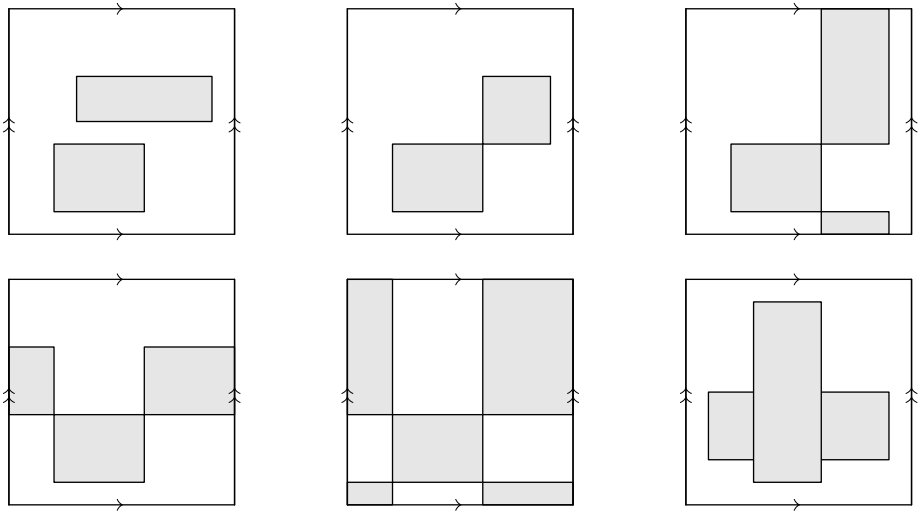}}
\caption{Compatible rectangles}\label{compatible}
\end{figure}
In the latter case we draw $r_1$ as if it passes over $r_2$, see Fig.~\ref{compatible}.

An example of a rectangular diagram of a surface is shown in Fig.~\ref{trefoil} on the left. The free vertices
are marked by small circles.
\begin{figure}[ht]
$$\begin{array}{ccc}
\includegraphics{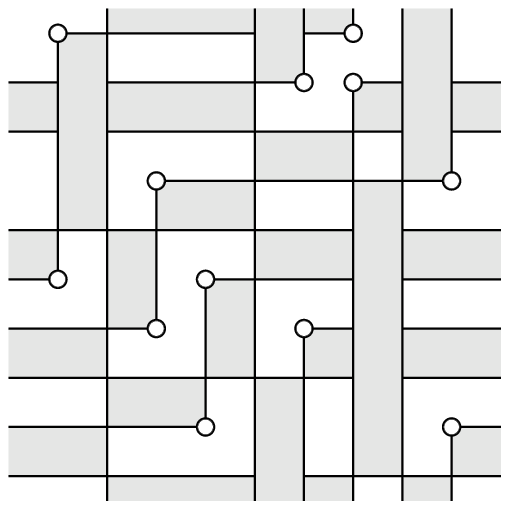}&\hskip0.5cm&\includegraphics{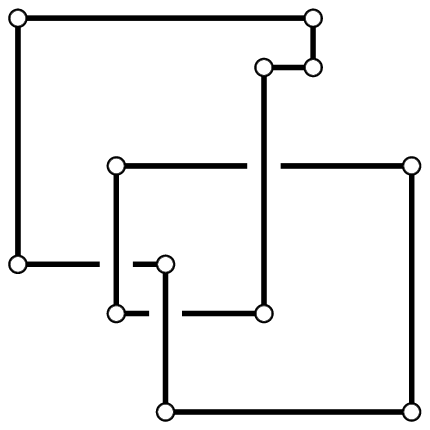}\\
\Pi&&\partial\Pi
\end{array}
$$
\caption{A rectangular diagram of a surface and its boundary}\label{trefoil}
\end{figure}

\begin{defi}
By \emph{a rectangular diagram of a link} we mean a finite set $R$ of points in $\mathbb T^2$
such that every meridian and every longitude of $\mathbb T^2$ contains no or exactly two points from $R$.
The points in $R$ are referred to as \emph{vertices of $R$}.

When $R$ is presented graphically in a square, the vertical and horizontal straight
line segments connecting two vertices of $R$ will be called \emph{the edges of $R$}.
Formally, \emph{an edge} of $R$ is a pair of vertices lying on the same longitude or meridian of $\mathbb T^2$.
Such vertices are said \emph{to be connected by an edge}, and this will have literal meaning in
the pictures.

A rectangular diagram of a link $R$ is said to be \emph{connected} or to be \emph{a rectangular
diagram of a knot} if the vertices of $R$ can be ordered so that any two consecutive vertices
are connected by an edge.

\emph{A connected component} of a rectangular diagram of a link is a non-empty
subset that is a rectangular diagram of a knot.
\end{defi}

\begin{defi}
Let $\Pi$ be a rectangular diagram of a surface.
The set of free vertices of $\Pi$  will be called \emph{the boundary of $\Pi$} and
denoted by $\partial\Pi$.
\end{defi}

It is readily seen that the boundary of a rectangular diagram of a surface is always a rectangular
diagram of a link. In particular, for any rectangle $r\subset\mathbb T^2$ the boundary $\partial\{r\}$
of the diagram $\{r\}$ is the set of vertices of $r$. It should not be confused with the boundary $\partial r$
of the rectangle itself, which is understood in the conventional sense.

Fig.~\ref{trefoil} shows a rectangular diagram of a surface (left) and its boundary (right) with
edges added. The reason for drawing some edges passing over the others will be explained below.

\subsection{Cusps and pizza slices}
Here we introduce the class of topological objects (curves and surfaces) we will mostly deal with.

\begin{defi}
Let $K$ be a piecewise smooth simple arc or simple closed curve in $\mathbb S^3$ and $p\in K$
a point distinct from the endpoints. We
say that $K$ has \emph{a cusp} at the point $p$ if $K$ admits a local parametrization
$\gamma:(-1,1)\rightarrow K$ such that
$$\lim_{t\rightarrow{-0}}\dot\gamma=-
\lim_{t\rightarrow{+0}}\dot\gamma\ne0,\quad\gamma(0)=p.$$
The curve $K$ is called \emph{cusp-free} if it has no cusps, and \emph{cusped} if
all singularities of $K$ are cusps.
\end{defi}

The links in $\mathbb S^3$ that we consider will always be cusp-free.

\begin{defi}\label{surf-with-corners}
We say that $F\subset\mathbb S^3$ is \emph{a surface with corners} if $F$ is a subset of a
$2$-dimensional submanifold $M\subset\mathbb S^3$ with (possibly empty) boundary such
that:
\begin{enumerate}
\item
the embedding $M\hookrightarrow\mathbb S^3$ is regular and of smoothness class $C^1$;
\item
 $F$ is bounded in $M$ by a (possibly empty) collection of mutually disjoint
piecewise smooth cusp-free simple closed curves.
\end{enumerate}
\end{defi}

This definition simply means that surfaces we want to consider may have corners at the boundary,
but in a broader sense than one usually means by saying `a manifold with corners'.
Namely, the angle at a corner can be arbitrary between $0$ and $2\pi$ (exclusive), not
necessarily between $0$ and $\pi$
see Fig.~\ref{corners}.
\begin{figure}[ht]
\centerline{\includegraphics{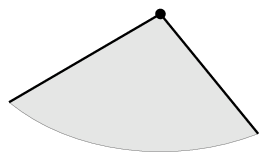}\hskip2cm\includegraphics{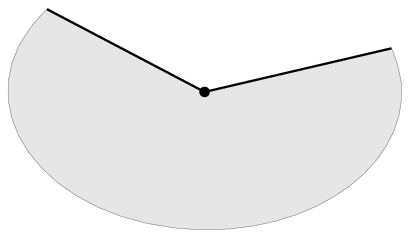}}
\caption{Arbitrary angles from $(0,2\pi)$ are allowed at corners}\label{corners}
\end{figure}
However, surfaces with corners are not allowed to spiral around a point at the boundary as shown in~Fig.~\ref{spiral}.
\begin{figure}[ht]
\centerline{\includegraphics{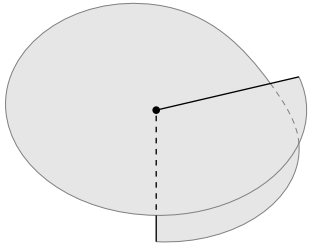}}
\caption{Spiralling around a boundary point is forbidden for a surface with corners}\label{spiral}
\end{figure}

For future use we also need the following definition.

\begin{defi}
By \emph{a pizza slice centered at $p\in\mathbb S^3$} we mean a disc that has
the form $F\cap\mathbb B_\varepsilon(p)$, where~$F$ is a surface with corners such that
$p\in\partial F$, and $\mathbb B_\varepsilon(p)$ is a closed ball centered
at $p$ with small enough radius $\varepsilon>0$ such that $\partial F$
has no singularities in $\mathbb B_\varepsilon(p)\setminus\{p\}$ and $\partial F\cap\mathbb B_\varepsilon(p)$
consists of a single arc. (The point $p$
may or may not be a singularity of $\partial F$.)

For such a pizza slice we also say that it is \emph{attached to the arc
$\partial F\cap\mathbb B_\varepsilon(p)$}.

Two pizza slices $\nabla_1$, $\nabla_2$ centered at the same point $p\in\mathbb S^3$
are said to be \emph{equivalent} if for small enough $\varepsilon>0$
the pizza slices $\nabla_i'=\nabla_i\cap\mathbb B_\varepsilon(p)$, $i=1,2$, are attached to the same
arc and for $x\in\nabla_1$ we have $d(x,\nabla_2)=o\bigl(d(x,p)\bigr)$ $(x\rightarrow p)$,
where $d(\,\cdot\,,\,\cdot\,)$ denotes the distance function.
\end{defi}

\subsection{3D realizations of rectangular diagrams}
With every rectangular diagram of a link or a surface we associate an
actual link or a surface, respectively, in $\mathbb S^3$ as we now describe.

We represent $\mathbb S^3$ as the join of two circles:
\begin{equation}\label{join}
\mathbb S^3=\mathbb S^1*\mathbb S^1=\bigl(\mathbb S^1\times\mathbb S^1\times[0,1]\bigr)/\bigl((\theta,\varphi,0)\sim(\theta',\varphi,0),\
(\theta,\varphi,1)\sim(\theta,\varphi',1)\bigr).
\end{equation}
We also identify it with the unit sphere in $\mathbb R^4$ as follows:
$$(\theta,\varphi,\tau)\mapsto\bigl(\cos(\pi\tau/2)\cos\varphi,\cos(\pi\tau/2)\sin\varphi,
\sin(\pi\tau/2)\cos\theta,\sin(\pi\tau/2)\sin\theta\bigr).$$
The triple $(\theta,\varphi,\tau)$ will be used as a coordinate system in $\mathbb S^3$.
The two circles defined by $\tau=0$ and $\tau=1$
will be denoted by $\mathbb S^1_{\tau=0}$ and $\mathbb S^1_{\tau=1}$, respectively.

\begin{defi}
By \emph{the torus projection} of a subset $X\subset\mathbb S^3$ we mean the image
of $X\setminus(\mathbb S^1_{\tau=0}\cup\mathbb S^1_{\tau=1})$ under the map
$\mathbb S^3\setminus(\mathbb S^1_{\tau=0}\cup\mathbb S^1_{\tau=1})\rightarrow\mathbb T^2$
defined by $(\theta,\varphi,\tau)\mapsto(\theta,\varphi)$.
\end{defi}

For a point $v=(\theta,\varphi)$ in $\mathbb T^2$ we denote by $\widehat v$
the image of the arc $v\times[0,1]\subset\mathbb S^1\times\mathbb S^1\times[0,1]$ in $\mathbb S^3=\mathbb S^1*\mathbb S^1$.

\begin{defi}
Let $R$ be a rectangular diagram of a link. By \emph{the link associated with $R$} we mean
the following union of arcs:
$\widehat R=\bigcup_{v\in R}\widehat v$.
\end{defi}

One can see that this union is indeed a collection of pairwise disjoint simple
closed piecewise smooth curves. Moreover, $\widehat R$ is a union of
piecewise geodesic closed curves, which are cusp-free.
One can also see that if $R'\subset R$ is a connected component of $R$, then $\widehat{R'}$
is a connected component of~$\widehat R$ and vice versa.

To get a conventional, planar picture of a link isotopic to $\widehat R$ one cuts $\mathbb T^2$ into
a square and then join the vertices of $R$ by edges, letting
the vertical edges pass over the horizontal ones, see the right picture in Fig.~\ref{trefoil}.
Finally, we remark that $R$ is the torus projection of $\widehat R$, and $\widehat R$
is the only link satisfying this property.

Cooking a surface out of a rectangular diagram of a surface is less visual, but the main principle
is similar: the surface associated with a rectangular diagram of a surface
should have the torus projection prescribed by the diagram.
The surface associated with a rectangular diagram will be composed of discs
associated with individual rectangles.

For a rectangle $r=[\theta_1,\theta_2]\times[\varphi_1,\varphi_2]\subset\mathbb T^2$,
the image of $r\times[0,1]$ in $\mathbb S^3$ under identifications~\eqref{join}
is the tetrahedron $[\theta_1,\theta_2]*[\varphi_1,\varphi_2]\subset
\mathbb S^1*\mathbb S^1$, which we denote by $\Delta_r$. We want to define $\widehat r$ as a disc
in $\Delta_r$ with boundary $\{\theta_1,\theta_2\}*\{\varphi_1,\varphi_2\}$ so
that the union of such discs over all rectangles of a rectangular
diagram of a surface yield a smoothly embedded surface.
This can be done in numerous ways among which we
choose one particularly convenient as it allows us to write down
an explicit parametrization of~$\widehat r$ and behaves
nicely in the respects addressed in Section~\ref{giroux}.

We denote by $h_r$ a bounded harmonic function on the interior~$\interior(r)$ of~$r$ that
tends to $0$ as $\varphi$ tends to $\varphi_1$ or~$\varphi_2$ while
$\theta\in(\theta_1,\theta_2)$ stays fixed, and
tends to~$1$ as $\theta$ tends to $\theta_1$ or $\theta_2$
while $\varphi\in(\varphi_1,\varphi_2)$ stays fixed.
Such a function exists and is unique, which follows from the Poisson integral formula and
the uniformization theorem.

\begin{rema}\label{wp}
The function $h_r$ admits an explicit presentation in terms of the Weierstrass elliptic function
$\wp(z)=\wp\bigl(z\,|\,\theta_2-\theta_1,\mathbbm i(\varphi_2-\varphi_1)\bigr)$
with half-periods $(\theta_2-\theta_1)$ and $\mathbbm i(\varphi_2-\varphi_1)$:
$$h_r(\theta,\varphi)=\frac1\pi\arg\frac{\wp(z_{\theta,\varphi})-\wp(z_{\theta_2,\varphi_2})}{
\bigl(\wp(z_{\theta,\varphi})-\wp(z_{\theta_1,\varphi_2})\bigr)
\bigl(\wp(z_{\theta,\varphi})-\wp(z_{\theta_2,\varphi_1})\bigr)},$$
where $z_{\theta,\varphi}=\theta-\theta_1+\mathbbm i(\varphi-\varphi_1)$,
$\theta_{1,2}$ and $\varphi_{1,2}$ are assumed to be reals satisfying
$\theta_1<\theta_2<\theta_1+2\pi$, $\varphi_1<\varphi_2<\varphi_1+2\pi$.
\end{rema}

\begin{defi}\label{tile-def}
We call the image in $\mathbb S^3$ under identifications~\eqref{join}
of the closure of the following open disc in $\mathbb T^2\times[0,1]$:
\begin{equation}\label{tileformula}
\left\{\bigl(v, \widetilde h_r(v)\bigr)\;;\;v\in\interior(r)\right\},\quad\text{where }
\widetilde h_r(v)=(2/\pi)\cdot\arctan\sqrt{\tan(\pi h_r(v)/2)},
\end{equation}
\emph{the tile associated with $r$} and denote it by $\widehat r$.

Let $\Pi$ be a rectangular diagram of a surface. We define  \emph{the surface associated with $\Pi$}
to be the union~$\widehat\Pi$ of the tiles associated with
rectangles from $\Pi$:
$\widehat\Pi=\bigcup_{r\in\Pi}\widehat r$.
\end{defi}

\begin{prop}
Let  $\Pi$ be a rectangular diagram of a surface. Then $\widehat\Pi$ is a surface
with corners in $\mathbb S^3$, and we have
$\partial\widehat\Pi=\widehat{\partial\Pi}$.
\end{prop}

\begin{proof}
First consider a single rectangle $r=[\theta_1,\theta_2]\times[\varphi_1,\varphi_2]$. The
function $h_r$ can be extended continuously and smoothly to $\partial r$ except at the
vertices of $r$, where it jumps by $\pm1$. The closure $\Gamma$ of the graph of $h_r$
is a smooth image of an octagon. Its boundary $\partial\Gamma$ consists of
four straight line segments parallel to the sides of $r$ and another four straight line segments
hanging over the vertices of $r$, see Fig.~\ref{h(theta,phi)}.

\begin{figure}[ht]
\includegraphics[scale=0.75]{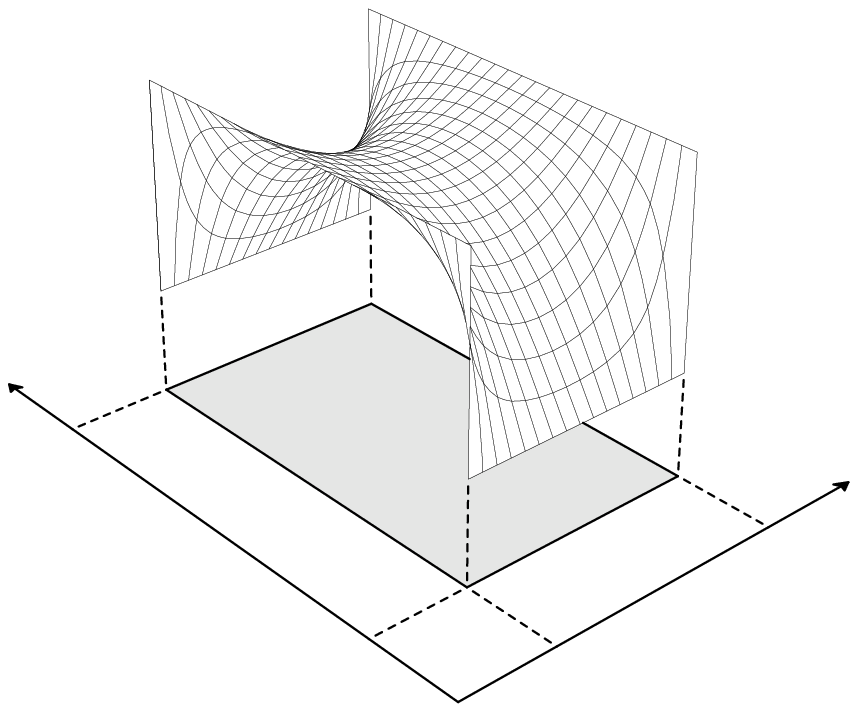}\put(0,40){$\theta$}%
\put(-62,5){$\theta_1$}\put(-17,30){$\theta_2$}\put(-112,8){$\varphi_1$}\put(-175,53){$\varphi_2$}%
\put(-189,62){$\varphi$}\put(-95,55){$r$}\put(-30,100){$\Gamma$}
\hskip1cm
\includegraphics[scale=0.75]{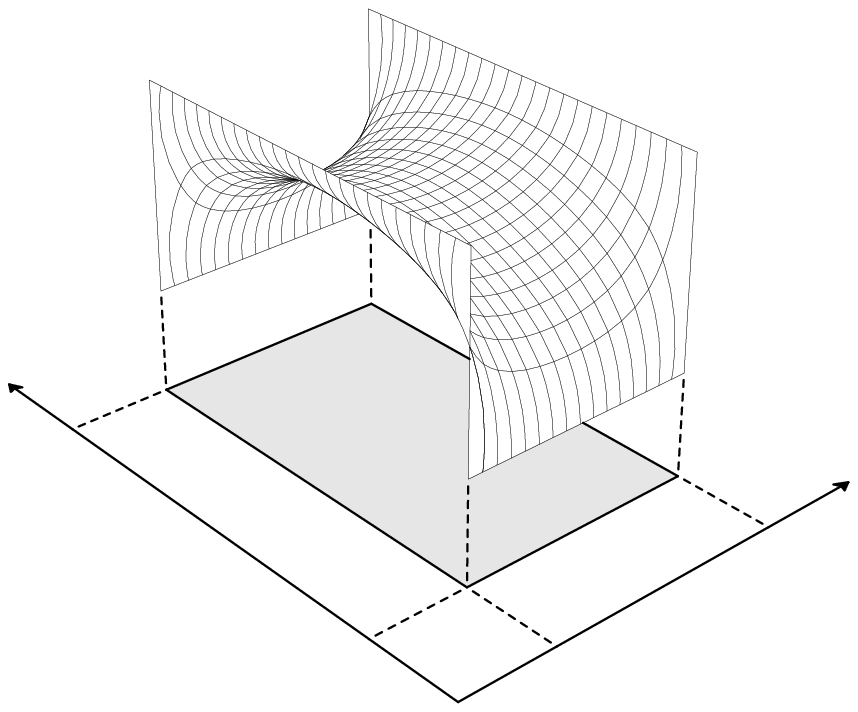}\put(0,40){$\theta$}%
\put(-62,5){$\theta_1$}\put(-17,30){$\theta_2$}\put(-112,8){$\varphi_1$}\put(-175,53){$\varphi_2$}%
\put(-189,62){$\varphi$}\put(-95,55){$r$}\put(-30,100){$\widetilde\Gamma$}
\caption{The graphs $\Gamma$ and $\widetilde\Gamma$ of the functions $h_r(\theta,\varphi)$ and
$\widetilde h_r(\theta,\varphi)$, respectively}\label{h(theta,phi)}
\end{figure}

Along each of the latter four straight line segments, $\Gamma$ is tangent to a helicoid that is
independent of the size of $r$ (i.e.\ of $\theta_2-\theta_1$ and $\varphi_2-\varphi_1$):
the tangent plane to $\Gamma$ at the point $(\theta_i,\varphi_j,h)$, $i,j\in\{1,2\}$, $h\in[0,1]$,
is $\ker\bigl(\cos(\pi h/2)\,d\varphi-(-1)^{i-j}\sin(\pi h/2)\,d\theta\bigr)$.
Tangency with a helicoid is a general property of the graphs of bounded harmonic functions
on a polygon that are constant on each side of the polygon and have jumps at corners. Near
each corner $w_0\in\mathbb C$ such a function $h$ has a form $h(w)=a+b\Arg(w-w_0)+o(|w|)$,
where $a,b\in\mathbb R$ and $w$ is a complex coordinate in the plane. The tangency
with a helicoid can even be shown to be of the second order.

By composing $h_r$ with the map $\zeta: x\mapsto(2/\pi)\cdot\arctan\sqrt{\tan(\pi x/2)}$,
which is a monotonic function $[0,1]\rightarrow[0,1]$
such that $\zeta(1-x)=1-\zeta(x)$ and $\zeta'(0)=\zeta'(1)=\infty$, we get the function $\widetilde h_r$,
whose graph~$\widetilde\Gamma$ is also a curved octagon with the same boundary as $\Gamma$, but now it
has an additional property that the tangent plane to~$\widetilde\Gamma$ is orthogonal to the
$(\theta,\varphi)$-plane along the whole of the boundary $\partial\widetilde\Gamma$.
The specific choice of $\zeta$ will play an important role in the sequel.

Identification~\eqref{join} takes $\widetilde\Gamma$ to a disc with corners in $\mathbb S^3$, which is the tile $\widehat r$.
Fig.~\ref{tile-curved} shows a tile projected stereographically into $\mathbb R^3$.
(The reader is alerted that the `coordinate grids' in Fig.~\ref{h(theta,phi)},
which are added to visualize the surfaces, are not related to that in  Fig.~\ref{tile-curved}.)
The boundary $\partial\widehat r$ has exactly two points at each of the circles $\mathbb S^1_{\tau=0}$ and $\mathbb S^1_{\tau=1}$,
and at these points the tangent plane to $\widehat r$ is orthogonal to the corresponding circle.
\begin{figure}[ht]
\includegraphics[width=150pt]{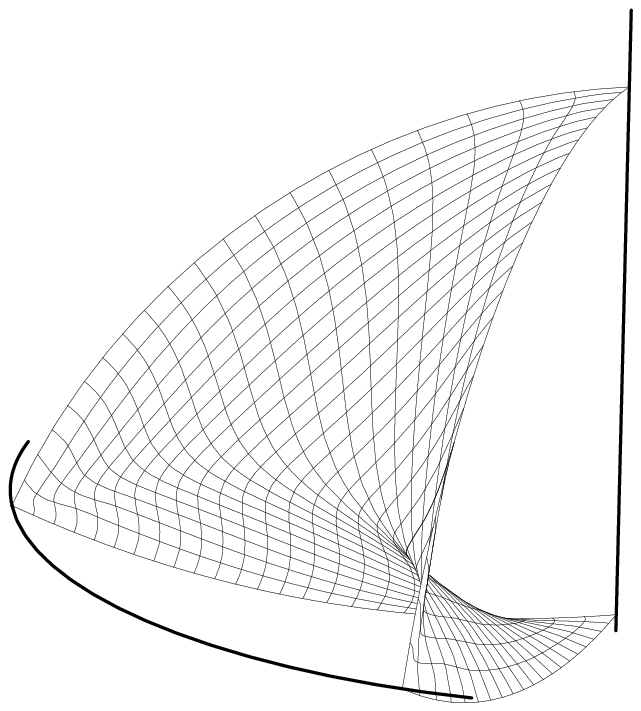}\put(-2,20){$\varphi_1$}\put(1,146){$\varphi_2$}%
\put(-158,42){$\theta_1$}\put(-60,-5){$\theta_2$}\put(-1,73){$S^1_{\tau=0}$}\put(-130,10){$S^1_{\tau=1}$}
\caption{A tile}\label{tile-curved}
\end{figure}

These points will be referred to as \emph{the vertices of $\widehat r$}, and the four
arcs of the form $\widehat v$, where $v\in\{\theta_1,\theta_2\}\times\{\varphi_1,\varphi_2\}$,
which form the boundary of $\widehat r$, \emph{the sides of $\widehat r$}.

By construction we also have the following.

\begin{lemm}\label{tileboundary}
The tile $\widehat r$ is tangent to the plane field
\begin{equation}\label{xi-}
\xi_-=\ker\bigl(\cos^2(\pi\tau/2)\,d\varphi-\sin^2(\pi\tau/2)\,d\theta\bigr)
\end{equation}
along the sides $\widehat{(\theta_1,\varphi_1)}$ and $\widehat{(\theta_2,\varphi_2)}$,
and to the plane field
\begin{equation}\label{xi+}\xi_+=\ker\bigl(\cos^2(\pi\tau/2)\,d\varphi+\sin^2(\pi\tau/2)\,d\theta\bigr)
\end{equation}
along the sides $\widehat{(\theta_1,\varphi_2)}$ and $\widehat{(\theta_2,\varphi_1)}$.
\end{lemm}

Let $r_1$ and $r_2$ be two compatible rectangles. If they are disjoint, then either the tiles $\widehat{r_1}$
and $\widehat{r_2}$ are disjoint or they share one or two vertices. In the latter case they
have the same tangent plane at the common vertices.

If $r_1$ and $r_2$ share a vertex $v$, say, then $\widehat{r_1}$ and $\widehat{r_2}$ share a side,
which is $\widehat v$, and all other tiles are disjoint from $\interior(\widehat v)$.
The tiles $\widehat{r_1}$ and $\widehat{r_2}$ approach $\widehat v$ from opposite sides and have the same tangent
plane at each point of $\widehat v$. Moreover, the intersection $\widehat{r_1}\cap\widehat{r_2}$
has the form $\cup_{v\in r_1\cap r_2}\widehat v$, where $r_1\cap r_2$ consists of common vertices of $r_1$
and $r_2$ (one, two, or four of them). Thus, the direction of the tangent plane to $\widehat{r_1}\cup\widehat{r_2}$
has no discontinuity at $\widehat{r_1}\cap\widehat{r_2}$.

Figure~\ref{two-tiles1} shows the preimages in $\mathbb T^2\times[0,1]$ of two tiles sharing a single side,
and the tiles themselves.
\begin{figure}[ht]
\centerline{\includegraphics[width=200pt]{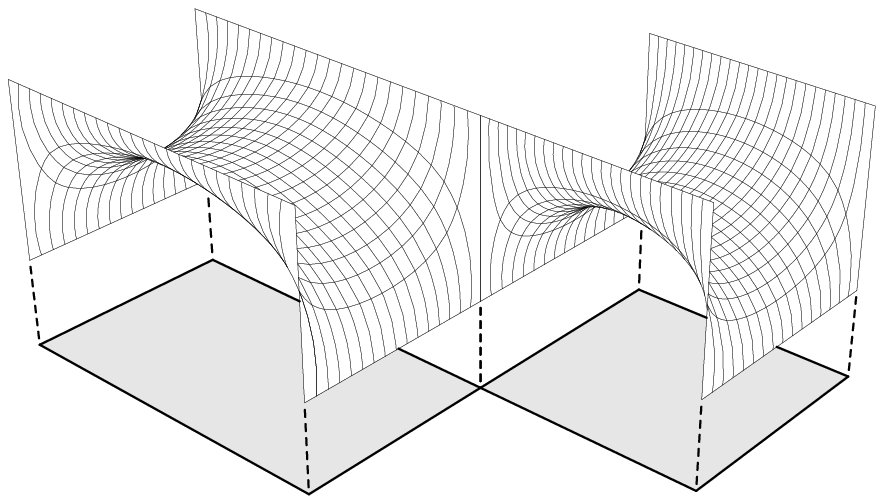}\put(-148,28){$r_1$}\put(-92,16){$v$}\put(-52,25){$r_2$}
\hskip1cm
\includegraphics[width=150pt]{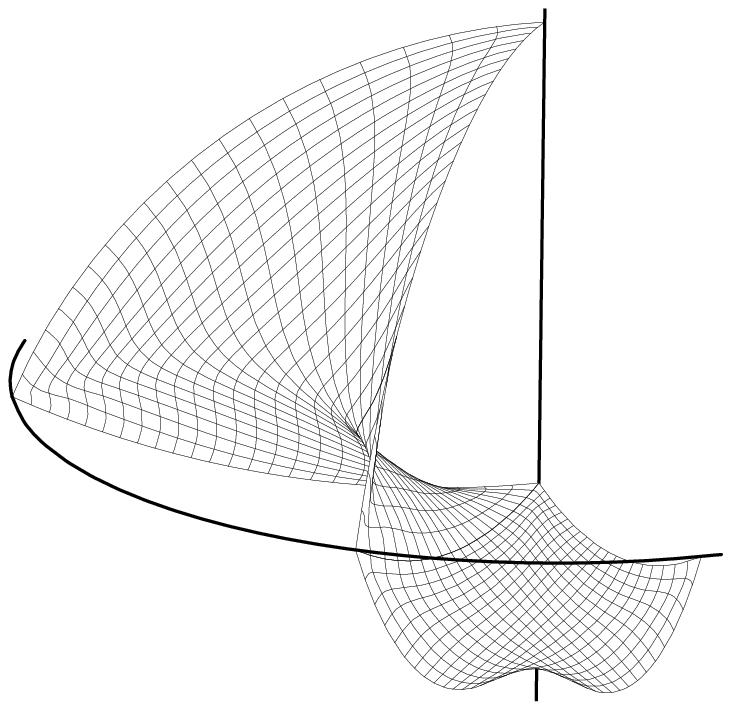}}
\caption{Two tiles sharing a side}\label{two-tiles1}
\end{figure}

Finally, if $r=r_1\cap r_2$ is a rectangle, then the tiles $\widehat{r_1}$ and $\widehat{r_2}$ are disjoint.
Indeed, by construction, at the horizontal sides of $r$ one of the functions $h_{r_1}$, $h_{r_2}$ vanishes whereas the other
is strictly positive, and at the vertical sides of $r$ one of these functions equals $1$ and
the other is strictly less than~$1$. Thus, we have $h_{r_1}(x)\ne h_{r_2}(x)$ for all $x\in\partial r$, and,
therefore, the inequality $h_{r_1}(x)\ne h_{r_2}(x)$ also holds for all $x\in r$ by the maximum principle
for harmonic functions.
Fig.~\ref{two-tiles2} shows the relative position of the graphs of the functions
$\widetilde h_{r_1}$ and $\widetilde h_{r_2}$
and the associated tiles in this case.
\begin{figure}[ht]
\centerline{\includegraphics[width=150pt]{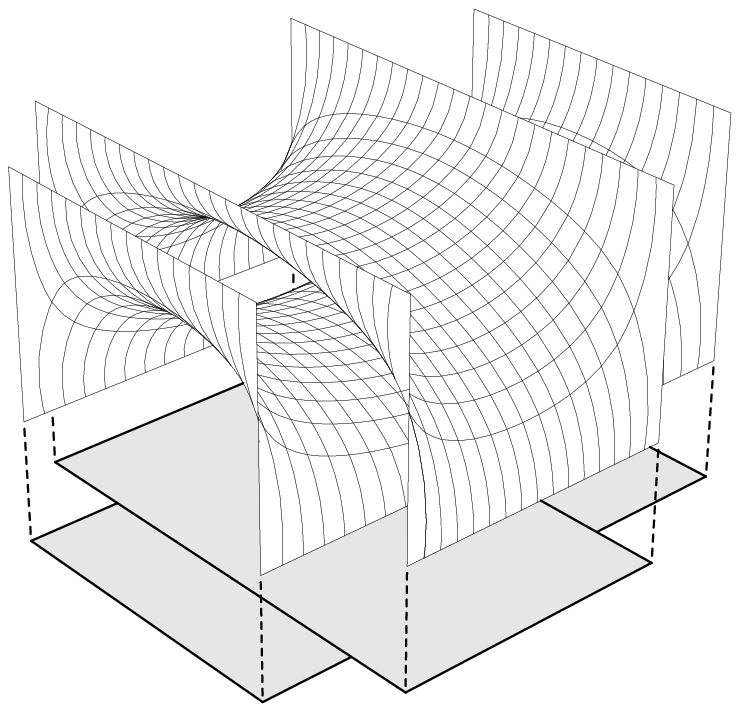}\put(-115,20){$r_1$}\put(-55,22){$r_2$}
\hskip1cm
\includegraphics[width=200pt]{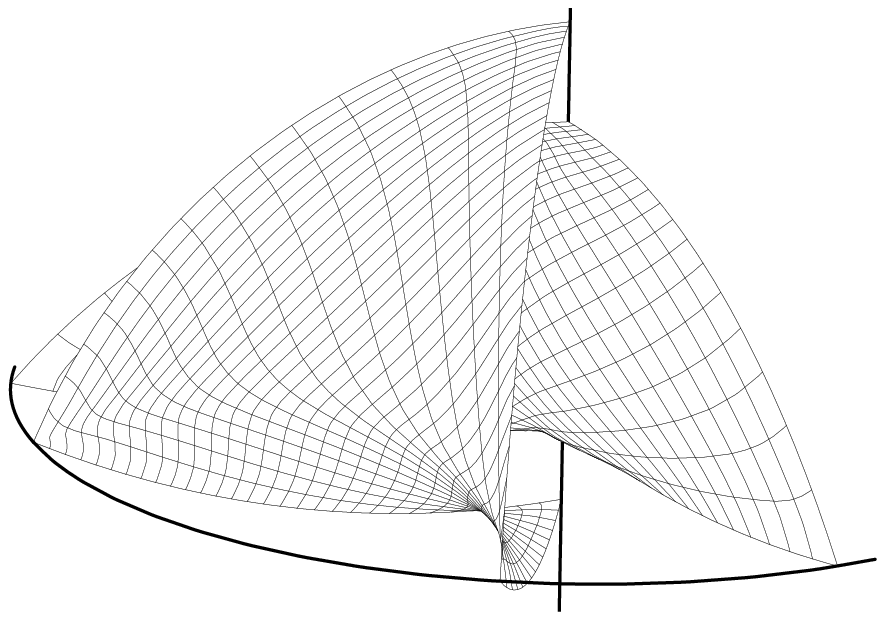}}
\caption{Two tiles corresponding to overlapping rectangles}\label{two-tiles2}
\end{figure}

Thus, in general, intersection of any two tiles $\widehat{r_1}$, $\widehat{r_2}$ of $\widehat\Pi$ bounds
to the intersection of their boundaries, and the direction of the tangent plane to $\widehat\Pi$ is continuous
on the whole of $\widehat\Pi$. It remains to examine the boundary of $\widehat\Pi$ and to verify that no
singularities other than corners occur at $\partial\widehat\Pi$.

The boundary of each tile $\widehat r$ consists of four arcs of the form $\widehat v$, where $v$ is a vertex of $r$.
If $v$ is a vertex of two rectangles from $\Pi$, then two tiles of $\widehat\Pi$ are attached to $\widehat v$,
so, the interior of $\widehat v$ is disjoint from $\partial\widehat\Pi$. Thus, $\partial\widehat\Pi$
consists of all the arcs $\widehat v$ such that $v$ is a free vertex of $\Pi$.
This implies $\partial\widehat\Pi=\widehat{\partial\Pi}$.

Let $p$ be a point in $\widehat\Pi\cap\mathbb S^1_{\tau=0}$. It corresponds
to a longitude of the torus $\mathbb T^2$. Let this longitude
be $\varphi=\varphi_0$ and denote it by $\ell$. The tiles having $p$ as a vertex are associated with rectangles
of the form $[\theta_1,\theta_2]\times[\varphi_0,\varphi_1]$
or $[\theta_1,\theta_2]\times[\varphi_1,\varphi_0]$. Let $r_1,\ldots,r_m$ be all such
rectangles.

For small enough $\varepsilon>0$, the intersections $\nabla_i=\widehat r_i\cap\mathbb B_\varepsilon(p)$ are all pizza
slices, and the only our concern is what their union looks like.

Since the rectangles $r_1,\ldots,r_m$ are pairwise compatible, the arcs $r_i\cap\ell$ have pairwise
disjoint interiors. The last condition in the definition of a rectangular diagram of
a surface guarantees that $\Bigl(\bigcup_{i=1}^mr_i\Bigr)\cap\ell$ is either
the whole of $\ell$ or a single arc of the form~$[\theta',\theta'']\times\{\varphi_0\}$.

In the former case, the union $\cup_{i=1}^m\nabla_i$ is `the whole pizza', i.e. a disc contained
in the interior of $\widehat\Pi$. In the latter case,
this union is a pizza slice with the angle $\alpha\in(0,2\pi)$, $\alpha\equiv\theta''-\theta'(\mathrm{mod}\,2\pi)$,
at $p$.

The intersection of $\widehat\Pi$ with the circle $\mathbb S^1_{\tau=1}$ is considered similarly.
\end{proof}

\begin{rema}
Recovering a surface from a rectangular diagram of a surface as described
above may look somewhat counterintuitive. Indeed, if $r$ is a rectangle, the vertices of the
tile $\widehat r$ correspond to the sides of $r$ and the sides
of~$\widehat r$ to the vertices of $r$.

However, as we will see, presentation of surfaces by rectangular
diagram is quite practical. In particular, the class of all surfaces
of the form $\widehat\Pi$ is such that each surface in it is uniquely
recovered from its torus projection, whose closure is the union $\bigcup_{r\in\Pi}r$.
\end{rema}

\subsection{Orientations} Quite often one needs to endow surfaces and links with an orientation.
Here is how to do this in the language of rectangular diagrams.

\begin{defi}
By \emph{an oriented rectangular diagram of a link} we mean a pair $(R,\epsilon)$
in which $R$ is a rectangular diagram of a link and $\epsilon$ is
an assignment  of `$+$' or `$-$' to every vertex so that the endpoints of every edge are
assigned different signs. The vertices with `$+$' assigned are referred to as \emph{positive}
and those with `$-$' assigned \emph{negative}. We also say that $\epsilon$
is \emph{an orientation} of $R$.
\end{defi}

Every orientation $\epsilon$ of a rectangular diagram of a link $R$ defines an
orientation of the link $\widehat R$ by demanding
that $\tau$ increases on the oriented arc $\widehat v\subset\widehat R$
whenever $v$ is a positive vertex of $R$ and
decreases otherwise. One can readily see that this gives a one-to-one correspondence
between orientations of a rectangular diagram of a link and those of the link associated with it.

\begin{rema}
From the combinatorial point of view oriented rectangular diagrams of links is the same thing
as grid diagrams, with X's in the latter corresponding to positive vertices and O's to negative ones.
\end{rema}

Now we introduce orientations for rectangular diagrams of surfaces.

The pair of functions $(\theta,\varphi)$ is a local coordinate system in the
interior of each tile. If two tiles $t_1$, $t_2$, say, share a side, then
the orientations defined by this system in $t_1$ and $t_2$
disagree at $t_1\cap t_2$, see Fig.~\ref{two-tiles1}. So, it is natural
to specify an orientation of a rectangular diagram of a surface as follows.

\begin{defi}
\emph{An oriented rectangular diagram of a surface} is a pair $(\Pi,\epsilon)$ in which $\Pi$ is a rectangular diagram
of a surface and $\epsilon$ is an assignment `$+$' or `$-$' to
every rectangle in $\Pi$ so that the signs assigned to any two rectangles sharing
a vertex are different. The rectangles with `$+$' assigned are
then called \emph{positive} and the others \emph{negative}. The assignment
$\epsilon$ is referred to as \emph{an orientation} of $\Pi$.
\end{defi}

Like in the case of links, orientations of any rectangular diagram of a surface $\Pi$ are put in one-to-one correspondence with
orientations of the surface $\widehat\Pi$ by demanding $(\theta,\varphi)$ being a positively oriented
coordinate pair in the tiles corresponding to positive rectangles, and negatively oriented coordinate pair
in the tiles corresponding to negative rectangles.
In particular, $\Pi$ does not admit an orientation if and only if $\widehat\Pi$ is a non-orientable surface.

In what follows, to simplify the notation, we omit an explicit reference to orientations.

\subsection{Framings} We will need a slightly more general notion of a framing of a link than
the one typically uses in knot theory.

\begin{defi}
Let $L$ be a cusp-free piecewise smooth link in $\mathbb S^3$. By \emph{a framing of $L$} we
mean an isotopy class (relative to $L$) of surfaces $F$ with corners such that
\begin{enumerate}
\item
$F$ is a union of pairwise disjoint annuli $A_1,\ldots,A_k$;
\item
for each $i=1,\ldots,k$, one of the connected components of $\partial A_i$
is a component of $L$ and the other is smooth and disjoint from $L$.
\end{enumerate}
\end{defi}

If $f$ is a framing of a link $L$ and $F$ is a collection of annuli representing $f$,
then we denote by $L^f$ the link $\partial F\setminus L$. Roughly speaking,
$L^f$ is obtained from $L$ by shifting along the framing $f$.
If $L$ is oriented, then $L^f$ is assumed to be oriented coherently with $L$.

We also denote by $\langle f\rangle$ the linking number
$\lk(L,L^f)$, which is clearly an invariant of $f$.
If $L$ is a smooth knot then framings $f$ of $L$ are classified by
$\langle f\rangle\in\mathbb Z$ (which is independent on the orientation of
the knot).

But if $L$ has singularities, a framing contains more information.
Indeed, at a singularity of $L$, the tangent plane to any surface $F$
with corners such that $L\subset\partial F$ is prescribed by $L$. So,
it cannot be changed by an isotopy of $F$ within the class of surfaces
with corners.

The surface $F$ can approach such a singularity from two
sides. Formally, this means that if $p$ is a singularity of $\partial F$
and $\gamma$ is an arc of the form $\partial F\cap\mathbb B_\varepsilon(p)$
with small enough $\varepsilon>0$, then there are two equivalence
classes of pizza slices attached to $\gamma$, and the one that
the surface $F$ realizes is an invariant of the framing.

Let a framing $f_0$ of $L$ be fixed.
If $f$ is another framing of $L$, then the twist of $f$
relative to~$f_0$ along any arc connecting two singularities of $\partial F$ is a multiple of $\pi$
and is also a topological invariant of $f$.

Clearly, these invariants determine the framing, and there are some obvious restrictions on them, which
we need not to discuss.

If $L'\subset L$ is a sublink, then the restriction $f|_{L'}$ of
a framing $f$ is defined in the obvious way. Clearly, any framing
of $L$ is defined uniquely by its restriction to every connected component of $L$.

\begin{defi}
If $F$ is a surface with corners we call the framing of $\partial F$ defined by a collar
neighborhood of $\partial F$ \emph{the boundary framing induced by $F$}.
\end{defi}

If $L$ is a sublink of $\partial F$, where $F$ is a surface
then we denote by $L^F$
any link contained in $F$ such that $L\cup L^F$
bounds a collar neighborhood of $L$ in $F$.

\begin{defi}\label{framing-def}
Let $R$ be a rectangular diagram of a link. By \emph{a framing of $R$} we mean
an ordering in each pair of vertices of $R$ forming an edge of $R$.
\end{defi}

For specifying a framing of a rectangular diagram of a link in a picture we proceed as follows.
For every edge $\{v_1,v_2\}$, we draw the arc $[v_1,v_2]$ if $v_2>v_1$ in the given framing
and the arc $[v_2,v_1]$ if $v_1>v_2$. Here by $[v_1,v_2]$ we mean
$[\theta_1,\theta_2]\times\{\varphi_0\}\subset\mathbb T^2$
if $\{v_1,v_2\}=\{\theta_1,\theta_2\}\times\{\varphi_0\}$
is a horizontal edge and $\{\theta_0\}\times[\varphi_1,\varphi_2]$ if $\{v_1,v_2\}=
\{\theta_0\}\times\{\varphi_1,\varphi_2\}$ is a vertical one.
The arcs are assumed to be drawn on $\mathbb T^2$, so in the actual planar
picture some of these arcs get cut into two pieces, see Fig.~\ref{framing-edge}.
\begin{figure}[ht]
$$\begin{array}{ccccccc}\includegraphics{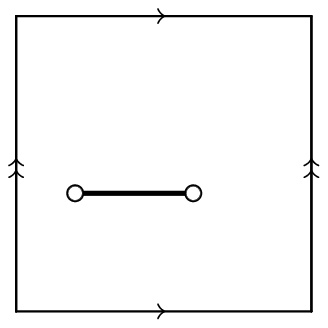}\put(-75,45){$v_1$}\put(-40,45){$v_2$}&&
\includegraphics{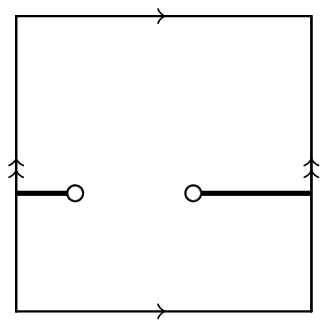}\put(-75,45){$v_1$}\put(-40,45){$v_2$}&&
\includegraphics{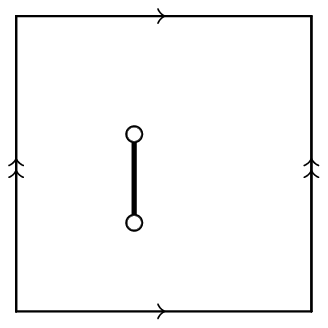}\put(-67,26){$v_1$}\put(-68,53){$v_2$}&&
\includegraphics{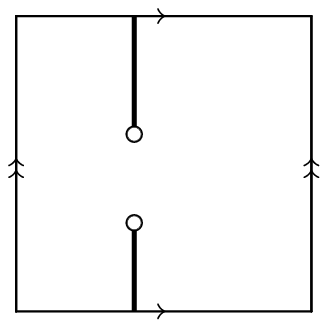}\put(-67,26){$v_1$}\put(-68,53){$v_2$}\\
v_2>v_1&&v_1>v_2&&v_2>v_1&&v_1>v_2
\end{array}$$
\caption{Specifying a framing in the picture of a rectangular diagram of a link}\label{framing-edge}
\end{figure}
As always, we draw vertical arcs over horizontal ones. The right
picture in Fig.~\ref{boundary-framing} shows
an example of a framed rectangular diagram of a link.
\begin{figure}[ht]
\centerline{\includegraphics[width=120pt]{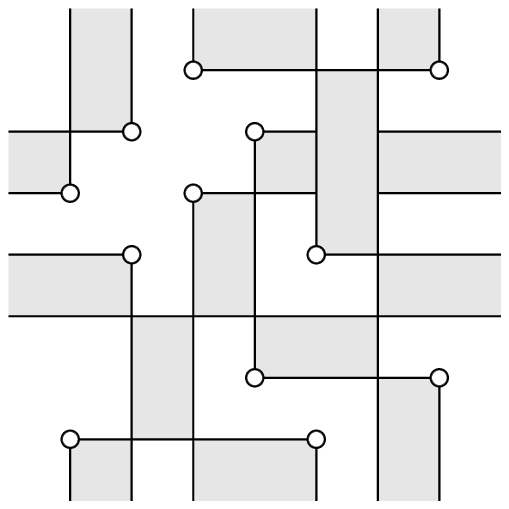}\hskip2cm
\includegraphics[width=120pt]{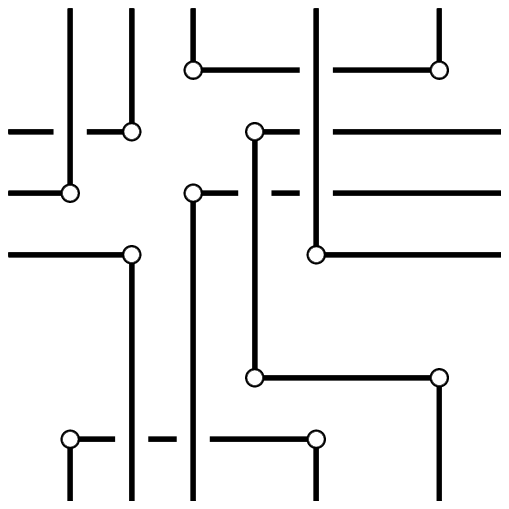}}
\caption{A rectangular diagram of a surface and its boundary with
boundary framing}\label{boundary-framing}
\end{figure}

By definition, any rectangular diagram of a link supports only finitely many
framings. They encode framings of the corresponding
link in $\mathbb S^3$ that are not `twisted too much'. We now define them formally.

\begin{defi}
Let $R$ be a rectangular diagram of a link. A framing of $\widehat R$
is called \emph{admissible} if it can be presented by a surface $F$ tangent
to one of the plane fields~$\xi_\pm$ defined by \eqref{xi-} and~\eqref{xi+} at every point $p\in\widehat R$.
\end{defi}

With every admissible framing of $\widehat R$ we associate (in general, non-uniquely) a framing of $R$ as follows.

Let $F$ be a union of annuli that represents an admissible framing $f$ of $\widehat R$
such that at every point $q\in\widehat R$ the surface $F$ is tangent to $\xi_+$ or $\xi_-$.
Let $p$ be a point from $\widehat R\cap(\mathbb S^1_{\tau=0}\cup\mathbb S^1_{\tau=1})$.
The tangent plane to $F$ at~$p$
is orthogonal to the corresponding $\mathbb S^1$ since so are $\xi_\pm$.

The intersection of $F$ with a small ball $\mathbb B_\varepsilon(p)$ is a pizza slice,
which is equivalent to one of the form
$[v_1,v_2]\times[0,\varepsilon]/{\sim}$ or $[v_1,v_2]\times[1-\varepsilon,1]/{\sim}$
(depending on whether $p\in\mathbb S^1_{\tau=0}$ or $p\in\mathbb S^1_{\tau=1}$),
where $\{v_1,v_2\}$ is the edge of $R$ corresponding to $p$,
and the equivalence~$\sim$ is given by~\eqref{join}. We denote
the pizza slice~$F\cap\mathbb B_\varepsilon(p)$ by $\nabla_p$.
The framing of $R$ associated with $f$ is defined
on the edge $\{v_1,v_2\}$ as $v_2>v_1$.

\begin{prop}
For a generic $R$,
the construction above defines a one-to-one correspondence between
admissible framings of $\widehat R$ and framings of $R$.
In general, this correspondence is one-to-many.
\end{prop}

\begin{proof}
Clearly, we can recover the equivalence classes of
the pizza slices $\nabla_p$ from the corresponding framing of $R$, since the only information carried
by the pizza slice $\nabla_p$ is from which side it approaches the link $\widehat R$.
Namely, let $p\in\mathbb S^1_{\tau=0}$ and let
$\{v_1,v_2\}$ be the corresponding horizontal edge.
Then~$\nabla_p$ is equivalent either to
$[v_1,v_2]\times[0,\varepsilon]/{\sim}$
or to $[v_2,v_1]\times[0,\varepsilon]/{\sim}$, and a choice of
one of this options is precisely the information recorded in the framing of $R$.
Similarly for vertical edges.

The point now is that, for an admissible framing $f$, the equivalence
classes of $\nabla_p$'s define $f$ completely.
Indeed, let $v$ be a vertex of $R$ and $p_0\in\mathbb S^1_{\tau=0}$, $p_1\in\mathbb S^1_{\tau=1}$
the endpoints of $\widehat v$.
At $p_i$, $i=0,1$, the tangent plane to $F$ is orthogonal to $\mathbb S^1_{\tau=i}$, and
the framing of $R$ associated with $f$
prescribes from which side the surface $F$ approaches $\widehat R$.

When $q$ traverses $\widehat v$ the planes $\xi_+(q)$
and $\xi_-(q)$ rotate around $\widehat v$ by $-\pi/2$ and $\pi/2$, respectively.
Which way the tangent plane to $F$ must rotate is determined by $\nabla_{p_0}$
and $\nabla_{p_1}$, see Fig.~\ref{twist}. So, the framing~$f$ is recovered uniquely.
\begin{figure}[ht]
\includegraphics[scale=0.85]{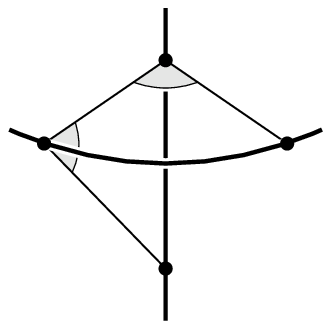}\put(-80,38){$p_0$}\put(-38,71){$p_1$}
\hskip0.3cm\raisebox{48pt}{$\longrightarrow$}\hskip0.3cm
\includegraphics[scale=0.85]{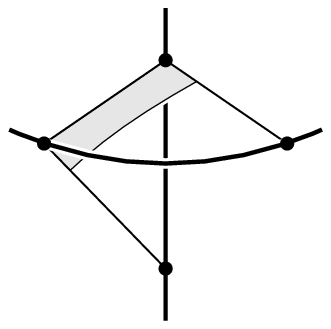}\hskip1cm
\includegraphics[scale=0.85]{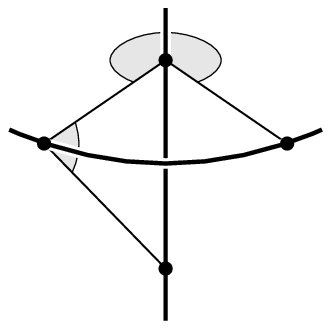}\put(-80,38){$p_0$}\put(-38,68){$p_1$}
\hskip0.3cm\raisebox{48pt}{$\longrightarrow$}\hskip0.3cm
\includegraphics[scale=0.85]{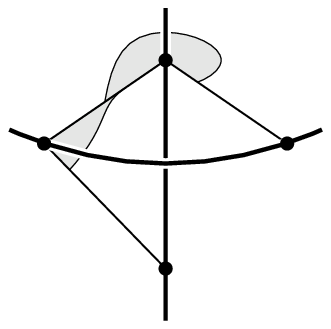}
\caption{Recovering an admissible framing from the pizza slices $\nabla_p$}\label{twist}
\end{figure}

The only possible reason for the correspondence between framings of $R$ and
admissible framings of $\widehat R$ to not be a bijection is a situation
when not all points in $\widehat R\cap(\mathbb S^1_{\tau=0}\cup\mathbb S^1_{\tau=1})$ are
singularities of $\widehat R$. This occurs when some edges of $R$
have `length' $\pi$, which generically does not happen.

If this does happen, then the tangent plane to a surface representing $f$
can rotate freely around $\widehat R$ at any point $p\in\widehat R\cap(\mathbb S^1_{\tau=0}\cup\mathbb S^1_{\tau=1})$
that is not a singularity of $\widehat R$, so when such a surface is required
to be tangent to $\xi_\pm$ this does not necessarily prescribe from
which side it should approach $\widehat R$ near $p$.
\end{proof}

In the sequel we will assume that all rectangular diagrams of links that we consider
are generic, and make no distinction between a framing
of a diagram and an admissible framing of the corresponding link.

An obvious but important thing to note here is the following

\begin{appr}
If $(R,f)$ is a framed rectangular diagram of a link and $D$ is
the picture of~$(R,f)$ obtained as described right
after Definition~\ref{framing-def}, then by smoothing $D$
near the vertices of $R$ we obtain the torus projection
of a link of the form ${\widehat R}^f$, in which
the indication of passing over and under at the crossings
corresponds to the relative position of the arcs
in the $\tau$-direction: the one with greater $\tau$
is overcrossing, see Fig.~\ref{smoothing}.
\begin{figure}[ht]
\includegraphics[width=100pt]{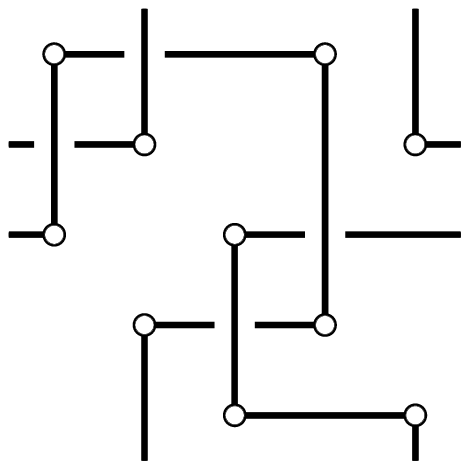}\hskip2cm
\includegraphics[width=100pt]{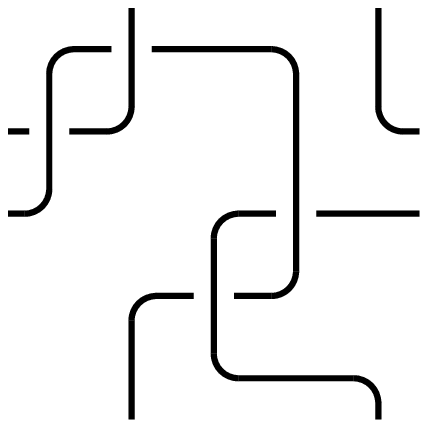}
\caption{A framed rectangular diagram $(R,f)$ and the torus projection of ${\widehat R}^f$}\label{smoothing}
\end{figure}

The principle also works in the other way. If $L\subset\mathbb S^3$ is
a link disjoint from the circles $\mathbb S^1_{\tau=0,1}$,
then a rectangular diagram of a link $R$ such that $\widehat R$
is isotopic to $L$ can be obtained by approximating the
torus projection of $L$ by the picture of a framed rectangular diagram
of a link.
\end{appr}

If $\Pi$ is a rectangular diagram of a surface, then $\widehat\Pi$
is tangent to $\xi_+$ or $\xi_-$ at every point of the boundary.
This together with Approximation Principle implies the following.

\begin{prop}\label{boundaryframing}
Let $\Pi$ be a rectangular diagram of a surface and $R=\partial\Pi$.
The boundary framing on $\widehat R$ induced by $\widehat\Pi$
is admissible.

The picture of the corresponding framing of $R$ is obtained
by connecting each pair of vertices of $R$ forming an edge by an arc
in $\mathbb T^2$ covered by the boundaries
of the rectangles from $\Pi$, see Fig.~\ref{boundary-framing}.
\end{prop}

\subsection{Thurston--Bennequin numbers}
Thurston--Bennequin number is a classical invariant of
Legendrian links~\cite{ben}. There are two Legendrian
links associated with every rectangular diagram $R$ (see \cite{ngth,DyPr} and
also Section~\ref{legendrian} below\footnote{The fact that arc-presentations `are Legendrian' was mentioned by W.\,Menasco
to the first present author already in 2003 and has been popularized since then}),
and their Thurston--Bennequin numbers can be computed
from $R$ without any reference to the associated Legendrian
links as we now describe.

Let $R$ be a(n oriented) rectangular diagram of a link. By $R^{^\nearrow}$
we denote a diagram obtained from $R$ by a small shift in the $(1,1)$-direction,
i.e.\ a diagram of the form
$$\{(\theta+\varepsilon,\varphi+\varepsilon)\;;\;(\theta,\varphi)\in R\},$$
where $\varepsilon>0$ is so small that, whenever $(\theta_0,\varphi_0)$
is a vertex of $R$, there are no vertices of $R$ in the following four regions:
$$\theta_0-\varepsilon\leqslant\theta<\theta_0,\quad
\theta_0<\theta\leqslant\theta_0+\varepsilon,\quad
\varphi_0-\varepsilon\leqslant\varphi<\varphi_0,\quad
\varphi_0<\varphi\leqslant\varphi_0+\varepsilon.$$
If $R$ is oriented, than $R^{^\nearrow}$ inherits the orientation
from~$R$.

Similarly we define $R^{^\nwarrow}$ by using a shift in the $(-1,1)$-direction.

\begin{defi}\label{tb(R)-def}
Let $R$ be an oriented diagram of a link. By \emph{the Thurston--Bennequin numbers of $R$}
we mean the following two linking numbers:
$$\tb_+(R)=\lk(\widehat R,\widehat{R^{^\nearrow}}),\quad
\tb_-(R)=-\lk(\widehat R,\widehat{R^{^\nwarrow}}).$$
\end{defi}

These numbers do not change if the orientation of $R$ is reversed.
So, if $\widehat R$ is a knot, then the numbers $\tb_\pm(R)$ do not depend on the orientation
of $R$, and we can speak about Thurston--Bennequin numbers of $R$ even if $R$
is not oriented.

\begin{prop}\label{tbproperties}
{\rm(i)} Let $R$ be an oriented rectangular diagram of a link and $F$ a surface with corners
in~$\mathbb S^3$ such that $\partial F\supset\widehat R$. If $F$ is tangent to the
plane field $\xi_+$ defined by~\eqref{xi+} {\rm(}respectively, $\xi_-$ defined by~\eqref{xi-}{\rm)}
at all points of $\widehat R$, then $\lk(\widehat R,{\widehat R}^F)$ is equal to $\tb_+(R)$ {\rm(}respectively,
to $-\tb_-(R)${\rm)}.

{\rm(ii)} The following equality holds
\begin{equation}\tb_+(R)+\tb_-(R)=-|R|/2,
\end{equation}
where $|R|$ denotes
the number of vertices in $R$.

{\rm(iii)} The set
$$\bigl\{\langle f\rangle\;;\;f\text{ is a framing of }R\bigr\}$$
coincides with $\bigl[\tb_+(R),-\tb_-(R)\bigr]\cap\mathbb Z$.
\end{prop}

\begin{proof}
Claim~(i) follows from the fact that $\widehat{R^\nearrow}$ (respectively, $\widehat{R^\nwarrow}$)
is obtained from $\widehat R$ by a small shift in a direction transverse to $\xi_+$ (respectively,
to $\xi_-$).

When $p$ traverses an arc of the form $\widehat v$, where $v\in R$, the plane~$\xi_+(p)$
rotates relative to the plane~$\xi_-(p)$ by the angle $-\pi$. This implies Claim~(ii).

We also see, that $\langle f\rangle$ is maximized (respectively, minimized) over all admissible
framings of $\widehat R$ if it can be presented by a surface with corners $F$ such that
$\widehat R\subset\partial F$ and $F$ is tangent to $\xi_-$ (respectively, to $\xi_+$)
along $\widehat R$. Thus $\langle f\rangle\in\bigl[\tb_+(R),-\tb_-(R)\bigr]$ for any
admissible framing $f$.

Let $n=|R|/2$ and $k\in[0,n]\cap\mathbb Z$. Define a framing $f_k$
as follows. For each horizontal edge $\{v_-,v_+\}$ of~$R$
with $v_-$ the negative vertex and $v_+$ the positive one,
we put $v_+>v_-$. Do the same for $k$ arbitrarily chosen vertical edges.
For the remaining $(n-k)$ vertical edges put the opposite: $v_->v_+$.

Let $F$ be a surface representing the corresponding admissible framing of $\widehat R$.
There will be $2k$ arcs of the form $\widehat v$, $v\in R$, along
which $F$ is tangent to $\xi_-$ and $2(n-k)$ such arcs along which $F$ is tangent
to~$\xi_+$. This implies $\langle f_k\rangle=\tb_++k$.
So, any integer in the interval $\bigl[\tb_+(R),-\tb_-(R)\bigr]$ can be realized
by $\langle f\rangle$ for a framing $f$ of $R$, which completes
the proof of Claim~(iii).
\end{proof}

\subsection{Which isotopy classes of surfaces can be presented by rectangular diagrams?}

The following is a combinatorial definition of the relative Thurston--Bennequin number.
The latter appears, e.g.\ in~\cite{ho}.

\begin{defi}\label{reltb(R)-def}
Let $R$ be an oriented rectangular diagram of a link and $F\subset\mathbb S^3$
a surface (which can be more general than a surface with corners)
whose boundary contains $\widehat R$.
By \emph{the Thurston--Bennequin numbers of $R$ relative to $F$} we call
$$\tb_+(R;F)=\tb_+(R)-\lk(\widehat R,{\widehat R}^F),\quad
\tb_-(R;F)=\tb_-(R)+\lk(\widehat R,{\widehat R}^F).$$
\end{defi}

Again, if $R$ has a single connected component, then the numbers $\tb_\pm(R;F)$ are independent
of the orientation of $R$.

\begin{theo}\label{main1}
Let $R$ be a rectangular diagram of a link and let $F$ be a compact surface in $\mathbb S^3$ such
that each connected component of $\widehat R$ is either disjoint from $F$ or
contained in $\partial F$. Then a rectangular diagram
of a surface~$\Pi$ such that $\widehat\Pi$ is isotopic to $F$
relative to $\widehat R$ exists if and only any connected component
$K\subset R$ such that $\widehat{K}\subset\partial F$, has
non-positive Thurston--Bennequin numbers relative to $F$:
\begin{equation}\label{nonpositive}
\tb_+(K,F)\leqslant0,\quad\tb_-(K,F)\leqslant0.
\end{equation}
\end{theo}

In this theorem $F$ is not necessarily assumed to be a surface with corners in the
sense of Definition~\ref{surf-with-corners}, and is allowed to have more
general singularities at the boundary, both initially and during the isotopy.
This allows to rotate it freely around $\partial F$, even near the singularities of $\partial F$.

\begin{proof}
First, we show that condition~\eqref{nonpositive} is necessary. Let $\Pi$ be a rectangular diagram of a surface
and~$K$ a connected component of $\partial\Pi$.
By Proposition~\ref{boundaryframing}
the boundary framing of~$K$ induced by~$\widehat\Pi$ is admissible.

Therefore, by Proposition~\ref{tbproperties} for any connected component $K\subset R$
we have
$$\lk\bigl(\widehat{K},{\widehat{K}}^F\bigr)\in\bigl[\tb_+(K),-\tb_-(K)\bigr],$$
which is equivalent to~\eqref{nonpositive}.

Now assume that~\eqref{nonpositive} holds for any component of $\widehat R\cap\partial F$.
We will use Lemma~\ref{more-components} and Proposition~\ref{isotopy-to-rectangular} proven below.

The link $\partial F$ can have connected components disjoint from $\widehat R$.
By Lemma~\ref{more-components} we can isotop $F$ keeping~$\widehat R$
fixed so that $\partial F\setminus\widehat R$ will have the form
$\widehat{R'}$ for some rectangular diagram of a link $R'$.
Moreover,
we can ensure that the restriction of the boundary framing induced by $F$ to any connected
component $\widehat K\subset\widehat{R'}$
is admissible by making $\tb_+(K)$ and $\tb_-(K)$ smaller than $\lk(\widehat K,{\widehat K}^F)$
and $-\lk(\widehat K,{\widehat K}^F)$, respectively.

Thus, we may assume without loss of generality that $\widehat R$ contains $\partial F$
(we achieve this by replacing $R$ with $R\cup R'$).
We can now isotop $F$ by altering it only in a small neighborhood of $\partial F$ so
that $F$ will become tangent either to $\xi_+$ or to $\xi_-$ at every point of $\partial F$.

Application of Proposition~\ref{isotopy-to-rectangular} with $\widehat{X_1}=\partial F$ and
$X_2=R\setminus X_1$ completes the proof.
\end{proof}

\begin{lemm}\label{more-components}
Let $R$ be a rectangular diagram of a link, $L$ a link disjoint from $\widehat R$, and $k$ an integer.
Then there exists a rectangular diagram of a link $R'$ such that
\begin{enumerate}
\item
the link $\widehat{R'}$ is isotopic to $L$ relative to $\widehat R$;
\item
for any connected component $K$ of $R'$ we have $\tb_+(K)<k$ and $\tb_-(K)<k$.
\end{enumerate}
\end{lemm}

\begin{proof}
To show the existence of $R'$ without the restrictions on the Thurston--Bennequin numbers
we follow the Approximation Principle with a slight modification that now we
should avoid intersections with the given link $\widehat R$ already given in the `rectangular'
form.

First, we perturb $L$ slightly to make it disjoint from both circles $\mathbb S^1_{\tau=0,1}$.
We may also ensure that the torus projection $P$ of $L$ has only double transverse self-intersections.

The torus projection of the link $\widehat R$ is $R$.
The torus projection $P$ is disjoint from
this set, since $L$ is disjoint from $\widehat R$. Let $\varepsilon$
be the distance between $R$ and $P$.

We can now approximate $P$ by a framed rectangular diagram
of a link $R'$ in the $\varepsilon$-neighborhood of~$P$.
By a small perturbation if necessary we can achieve that the edges
of $R'$ do not appear on the same longitudes or meridians as the
edges of $R$.
The link $\widehat{R'}$ will be isotopic to $L$ relative to $\widehat R$, see Fig.~\ref{approximation}.
\begin{figure}[ht]
\centerline{\includegraphics[width=120pt]{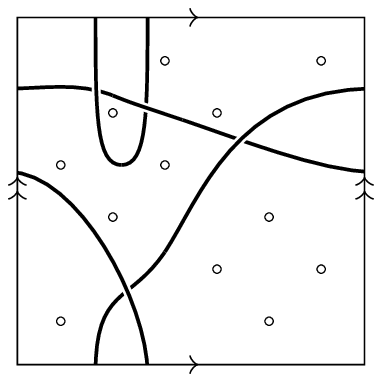}\quad\raisebox{55pt}{$\longrightarrow$}\quad
\includegraphics[width=120pt]{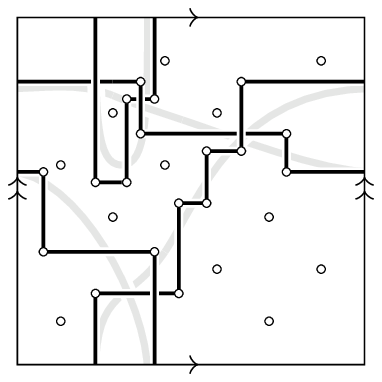}\quad\raisebox{55pt}{$\longrightarrow$}\quad
\includegraphics[width=120pt]{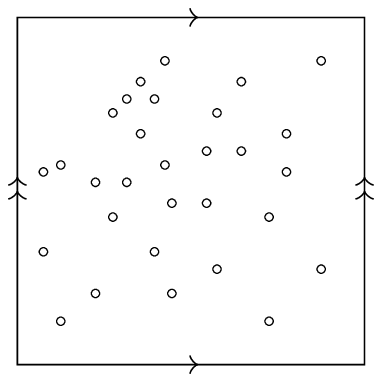}}
\caption{Approximating a torus projection of a link in the complement
of $\widehat R$
by a framed rectangular diagram of a link}\label{approximation}
\end{figure}
(The framing that $R'$ comes with plays no role in the sequel and should be discarded.)

Now if some of the Thurston--Bennequin numbers of connected components of $R'$
are too large, we apply stabilizations to $R'$. Recall (see~\cite{Dyn,DyPr})
that \emph{a stabilization} of a rectangular diagram is a replacement of a vertex $v_0$,
say, by three vertices $v_1,v_2,v_3$ such that
\begin{enumerate}
\item
$v_0,v_1,v_2,v_3$ are
vertices of a square $s$ of the form $[\theta_1,\theta_2]\times[\varphi_1,\varphi_2]$;
\item
the square $s$ is so small that
there are no other vertices of the diagram inside the regions $(\theta_1,\theta_2)\times\mathbb S^1$
and $\mathbb S^1\times(\varphi_1,\varphi_2)$ and there are exactly two
vertices at the boundary of each.
\end{enumerate}

We distinguish between \emph{type I} and \emph{type II} stabilizations as shown in Fig.~\ref{stabilizationtypes}.
\begin{figure}[ht]
\center{\includegraphics[scale=0.85]{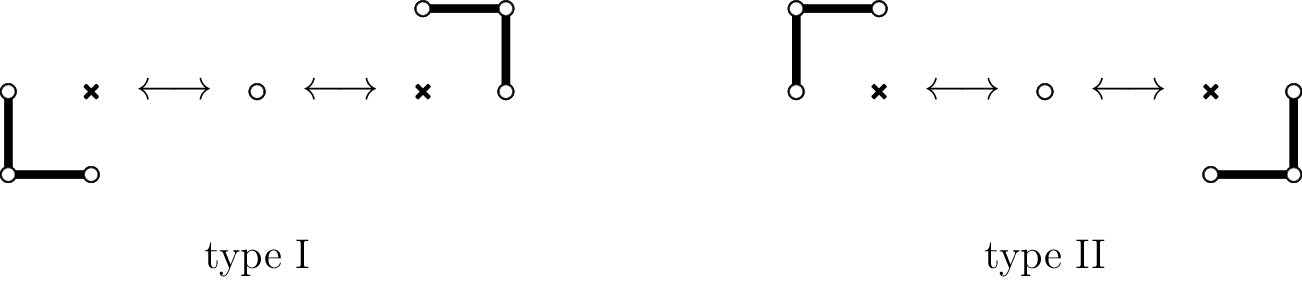}}
\caption{Types of stabilizations and destabilizations}\label{stabilizationtypes}
\end{figure}
Type~I stabilizations preserve $\tb_+$ and drop $\tb_-$ by $1$. Type~II stabilizations
preserve $\tb_-$ and drop $\tb_+$ by $1$. So, by applying sufficiently many
stabilizations of appropriate types we make the Thurston--Bennequin numbers of all
the components of $R'$ smaller than the given $k$.
\end{proof}

For any finite subset $X\subset\mathbb T^2$, we will use notation $\widehat X$
in the same sense as for rectangular diagrams:
$$\widehat X=\bigcup_{v\in X}\widehat v.$$
If $X$ is not a rectangular diagram of a link, then $\widehat X$ is not a link,
but $\widehat X$ is always a graph, i.e.\ a $1$-dimensional CW-complex.

The following statement will be given in greater generality
than necessary to establish Theorem~\ref{main1}. It will be used in full
generality later in the proof of Theorem~\ref{main2}.

\begin{prop}\label{isotopy-to-rectangular}
Let $F$ be a surface with corners,
and $X_1,X_2$ two disjoint finite subsets of~$\mathbb T^2$. Suppose that they satisfy
the following conditions:
\begin{enumerate}
\item
the graph $\widehat{X_1}$ contains $\partial F$;
\item
$F\cap\widehat{X_2}\subset\widehat{X_1}$;
\item
$F$ contains $\widehat{X_1}$;
\item $F$ is tangent either to $\xi_+$ or to $\xi_-$ at
all points of $\widehat{X_1}$.
\end{enumerate}
Then there exists a rectangular diagram of a surface $\Pi$
and an isotopy relative to $\widehat{X_1\cup X_2}$ taking $F$ to $\widehat\Pi$
such that the tangent plane to the surface remains constant
at any point from $\widehat{X_1}$ during the isotopy.
\end{prop}

The proof of this statement occupies the rest of this section.

\begin{proof}
We will be using \emph{the open book decomposition} of $\mathbb S^3$
with \emph{the binding circle} $\mathbb S^1_{\tau=0}$ and \emph{pages}
$\mathscr P_\theta=\{\theta\}*\mathbb S^1\subset\mathbb S^1*\mathbb S^1$ , $\theta\in\mathbb S^1$.

The sought-for diagram $\Pi$ and an isotopy from $F$ to $\widehat\Pi$
will be constructed in four steps. At each step,
we apply an isotopy to~$F$, and the result will still be denoted by $F$.
At the end we will have $F=\widehat\Pi$. Every isotopy is silently
assumed to be fixed at
$\widehat{X_1\cup X_2}$  and to have identity differential at $\widehat{X_1}$
at every moment. One can see that all modification of $F$ described below
can be made so as to comply with this restriction. Roughly speaking,
the surface near
$\widehat{X_1}$ is already `good enough' and
doesn't need to be altered there.

\smallskip
\noindent\emph{Step 1.}
We perturb $F$ slightly to make it transverse to the binding circle,
and so that the following conditions hold:
\begin{enumerate}
\item
the restriction of the (multi-valued) function $\theta$
to $F\setminus\mathbb S^1_{\tau=0}$ has only finitely many
critical points, and each page $\mathscr P_\theta$
contains at most one of them;
\item each critical point of $\theta|_{F\setminus\mathbb S^1_{\tau=0}}$ is
a local maximum, a local minimum, or a multiple saddle (including those
at the boundary, see explanations below);
\item
strict local extrema of $\theta|_{F\setminus\mathbb S^1_{\tau=0}}$ do
not occur at $\mathbb S^1_{\tau=1}$
(non-strict local extrema can occur at $\partial F$, they are not critical points);
\item
any arc of the form $\widehat v$, where $v\in\mathbb T^2$,
containing a local extremum of
$\theta|_{F\setminus\mathbb S^1_{\tau=0}}$ has no other tangency point
with $F$;
\item
if $p_1$ and $p_2$ are two points of $F$ at which
$\theta|_{F\setminus\mathbb S^1_{\tau=0}}$ has a local extremum,
then the values of the coordinate $\varphi$ at $p_1$ and $p_2$ are different.
\end{enumerate}

All this is achieved by a generic arbitrarily small perturbation of $F$.

Some explanations are in order. Recall that
throughout the paper we deal with surfaces with corners, whose
smoothness class is only supposed to be $C^1$. So, when we
speak about (multiple) saddle critical points we mean
by that only the topological type of the singularity of the level line foliation.

Namely, let $\theta_0$ be the value of $\theta$ at such a singularity $p$, say.
Assume that $p$ is an interior point of $F$.
In a small neighborhood of $p$ the level line $\theta|_F=\theta_0$ must be
a star graph and the sign of $(\theta|_F(q)-\theta_0)$ must change whenever $q$
crosses any of its `legs'. The number of these legs is even and
equals $2(k+1)$, where $k$ is \emph{the multiplicity} of the saddle.
We put no restriction on the angle between the legs at $p$, and
allow $k$ to be zero.

By saying that $p$ is a boundary (multiple) saddle we mean that it becomes
a (multiple) saddle after a small extension of $F$ beyond $\partial F$,
which must exist by the definition of a surface with corners. Figs.~\ref{multiple-saddles}
and~\ref{multiple-saddles-boundary}
show a few examples of how the level line foliation may look like
near an interior or boundary multiple saddle.

\begin{figure}[ht]
$$\begin{array}{ccccc}
\includegraphics[width=80pt]{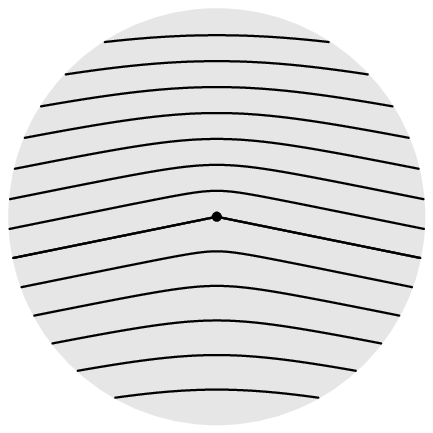}&\hskip1cm&
\includegraphics[width=80pt]{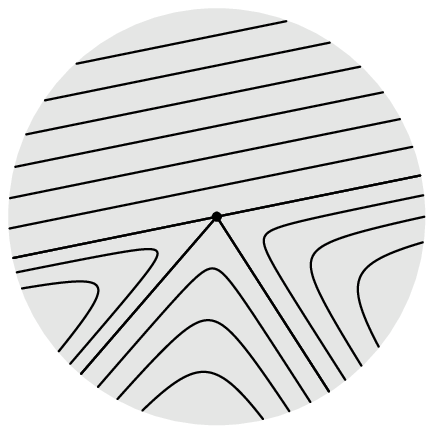}&\hskip1cm&
\includegraphics[width=80pt]{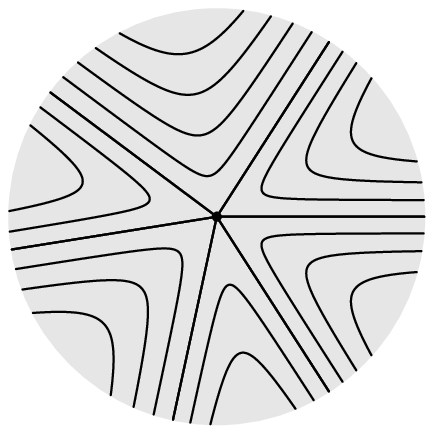}\\
k=0&&k=1&&k=2
\end{array}$$
\caption{Examples of multiple saddles}\label{multiple-saddles}
\end{figure}
\begin{figure}[ht]
$$\begin{array}{ccccc}
\includegraphics[width=80pt]{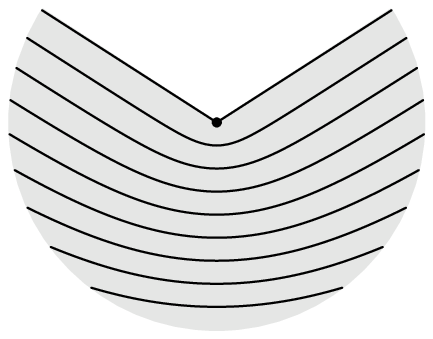}\put(-60,55){$\partial F$}&\hskip1cm&
\includegraphics[width=80pt]{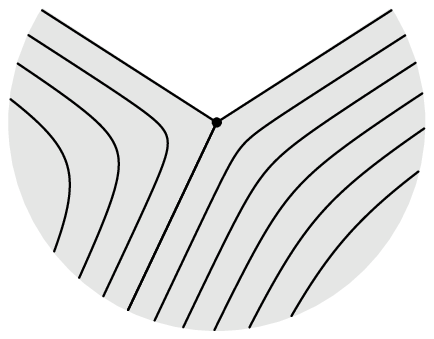}\put(-60,55){$\partial F$}&\hskip1cm&
\includegraphics[width=80pt]{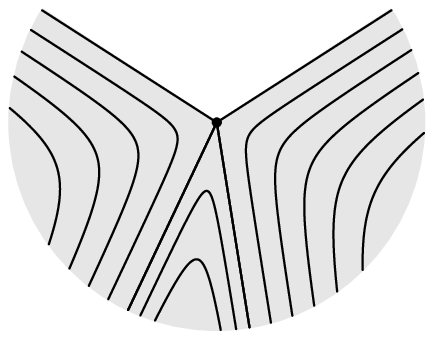}\put(-60,55){$\partial F$}
\end{array}$$
\caption{Examples of multiple saddles at the boundary}\label{multiple-saddles-boundary}
\end{figure}

The reader may wonder why we allow multiple saddles for
$\theta|_{F\setminus\mathbb S^1_{\tau=0}}$ with
other multiplicities than $1$. This is because $F$
and the tangent plane to $F$
are supposed to be fixed at $\widehat{X_1}$, which makes
some of such singularities occurring at $\mathbb S^1_{\tau=1}$ stable
under allowed isotopies.

\smallskip
\noindent\emph{Step 2.} Now we are going to get rid of all local extrema of
$\theta|_{\interior(F)\setminus\mathbb S^1_{\tau=0}}$
and to make all connected components of $\mathscr P_\theta\cap F$ simply connected for all $\theta\in\mathbb S^1$.
To this end, we use
the finger move trick introduced in a similar situation by T.\,Ito and K.\,Kawamuro in~\cite{IK}.

\begin{defi}
By \emph{a finger move disc} we call a pizza slice $\nabla=\nabla(\theta_1,\theta_2,\varphi_0,a)$ of the form
$[\theta_1,\theta_2]\times\{\varphi_0\}\times[0,a]/{\sim}$,
where $\theta_1,\theta_2,\varphi_0\in\mathbb S^1$, $a\in(0,1)$,
and the equivalence $\sim$ is given by~\eqref{join}, such that:
\begin{enumerate}
\item
$\nabla$ is transverse to $F$;
\item
each connected component of $\nabla\cap F$ is an arc intersecting
each radial arc $\{\theta\}\times\{\varphi_0\}\times[0,a]$, $\theta\in[\theta_1,\theta_2]$,
exactly once;
\item
$\nabla$ is disjoint from (multiple) saddle critical points of $\theta|_{F\setminus\mathbb S^1_{\tau=0}}$;
\item
$\nabla$ is disjoint from the graph $\widehat{X_1\cup X_2}$,
\end{enumerate}
see examples in the left pictures of Fig.~\ref{finger1} and~\ref{finger2}.
\end{defi}

\begin{lemm}\label{finger-discs}
There exists a finite collection of pairwise disjoint finger move discs $Y=\{\nabla_1,\ldots,\nabla_m\}$
such that, for any $\theta\in\mathbb S^1$, any simple closed curve in $F\cap\mathscr P_\theta$
intersects at least one of them.
\end{lemm}

\begin{proof}
We start by constructing, for any $\theta_0\in\mathbb S^1$,
a family $Y_{\theta_0}$ of pairwise disjoint
finger move discs such that their union intersects any simple closed curve in $F\cap\mathscr P_\theta$
if $\theta$ is sufficiently close to $\theta_0$.

If $p=(\theta_0,\varphi_0,a)$ is a local minimum of the restriction $\theta|_{\interior(F)\setminus\mathbb S^1_{\tau=0}}$, then
for a small enough $\varepsilon>0$, the disc $\nabla(\theta_0,\theta_0+\varepsilon,\varphi_0,a)$
will be a finger move disc that intersects all the
circles in $F\cap\mathscr P_\theta$, $\theta\in(\theta_0,\theta_0+\varepsilon)$,
located in a small neighborhood of $p$. We include it into $Y_{\theta_0}$.

Similarly, if $p=(\theta_0,\varphi_0,a)$ is a local maximum of the restriction $\theta|_{F\setminus\mathbb S^1_{\tau=0}}$, then
we include into $Y_{\theta_0}$ the finger move disc $\nabla(\theta_0-\varepsilon,\theta_0,\varphi_0,a)$
with small enough $\varepsilon>0$.

The finger moves discs we have added to $Y_{\theta_0}$ so far are called \emph{unmovable} (because the
parameter $\varphi_0$ is prescribed by the position of the critical point). Now we add
to $Y_{\theta_0}$
some finite number of \emph{movable} finger move discs so that every simple closed curve in $F\cap\mathscr P_{\theta_0}$
is intersected by at least one of them. The union of all discs in $Y_{\theta_0}$ will also intersect
all simple closed curves in $F\cap\mathscr P_\theta$ for any $\theta$ sufficiently close to $\theta_0$.

Now by the compactness of $\mathbb S^1$ we can choose finitely many $\theta_1,\ldots,\theta_m\in\mathbb S^1$
so that the union of all finger move discs from $Y=Y_{\theta_1}\cup\ldots\cup Y_{\theta_m}$
intersect any simply closed curve contained in $F\cap \mathscr P_\theta$ for any~$\theta$.
However, some of these discs may intersect each other.

Due to the restrictions we put on $F$ at Step~1 the unmovable discs in $Y$
are already pairwise disjoint (they occur at different $\varphi$'s). We can alter the parameters of the movable discs
so as to make all discs from~$Y$ pairwise disjoint and still have the property
that every simple closed curve in $F\cap \mathscr P_\theta$ is intersected
by a disc from $Y$ for any $\theta\in\mathbb S^1$.
\end{proof}

Having chosen a collection $Y$ of finger move discs as in Lemma~\ref{finger-discs}
we now apply \emph{a finger move} to $F$ for each disc $\nabla\in Y$.
By this we mean an isotopy in a small neighborhood of $\nabla$
that pushes $F$ across~$\nabla$ toward $\mathbb S^1_{\tau=0}$ as shown in Fig.~\ref{finger1} and~\ref{finger2}.
\begin{figure}[ht]
\includegraphics[scale=0.3]{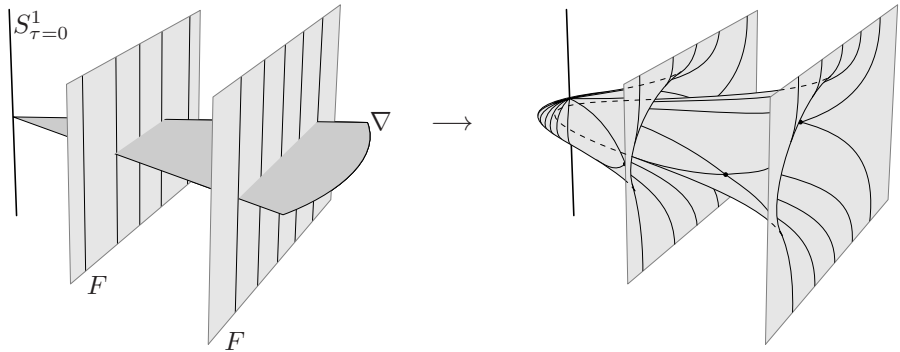}\put(-185,85){$\longrightarrow$}\put(-343,122){$S^1_{\tau=0}$}\put(-208,85){$\nabla$}\put(-315,23){$F$}\put(-263,2){$F$}
\caption{Finger move for a movable finger move disc}\label{finger1}
\end{figure}

\begin{figure}[ht]
\includegraphics[scale=0.3]{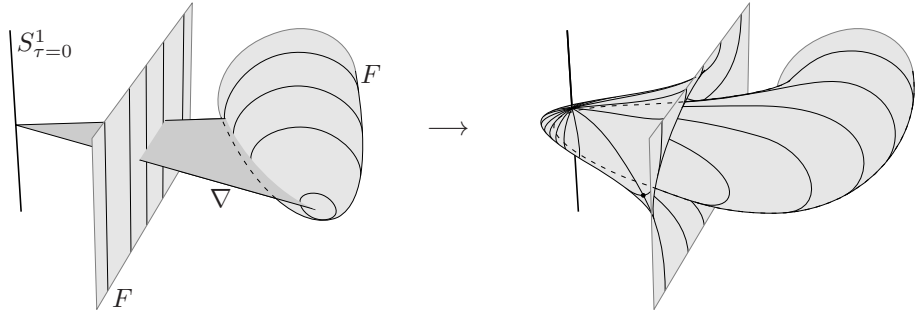}\put(-188,70){$\longrightarrow$}\put(-343,102){$S^1_{\tau=0}$}\put(-270,43){$\nabla$}\put(-307,5){$F$}\put(-213,90){$F$}
\caption{Finger move for an unmovable finger move disc}\label{finger2}
\end{figure}

The intersection of $F$ with pages $\mathscr P_\theta$ is altered as indicated in Fig.~\ref{sections}.
\begin{figure}[ht]
\includegraphics[scale=0.7]{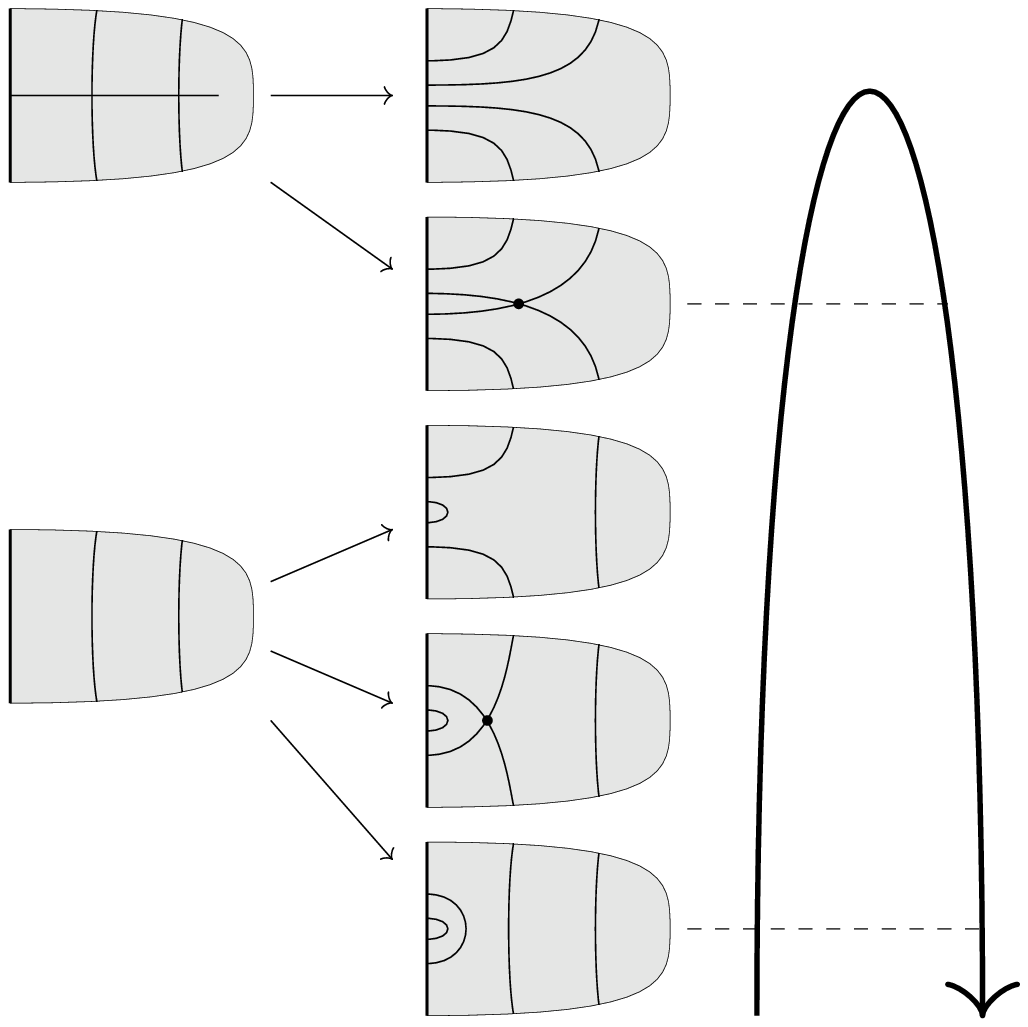}\put(-45,150){\small$\theta_1$}\put(-17,150){\small$\theta_2$}%
\put(-53,24){\small$\theta_1-\delta$}\put(-8,24){\small$\theta_2+\delta$}\put(-5,0){$\theta$}\put(-203,192){\small$\nabla$}
\hskip1cm\includegraphics[scale=0.7]{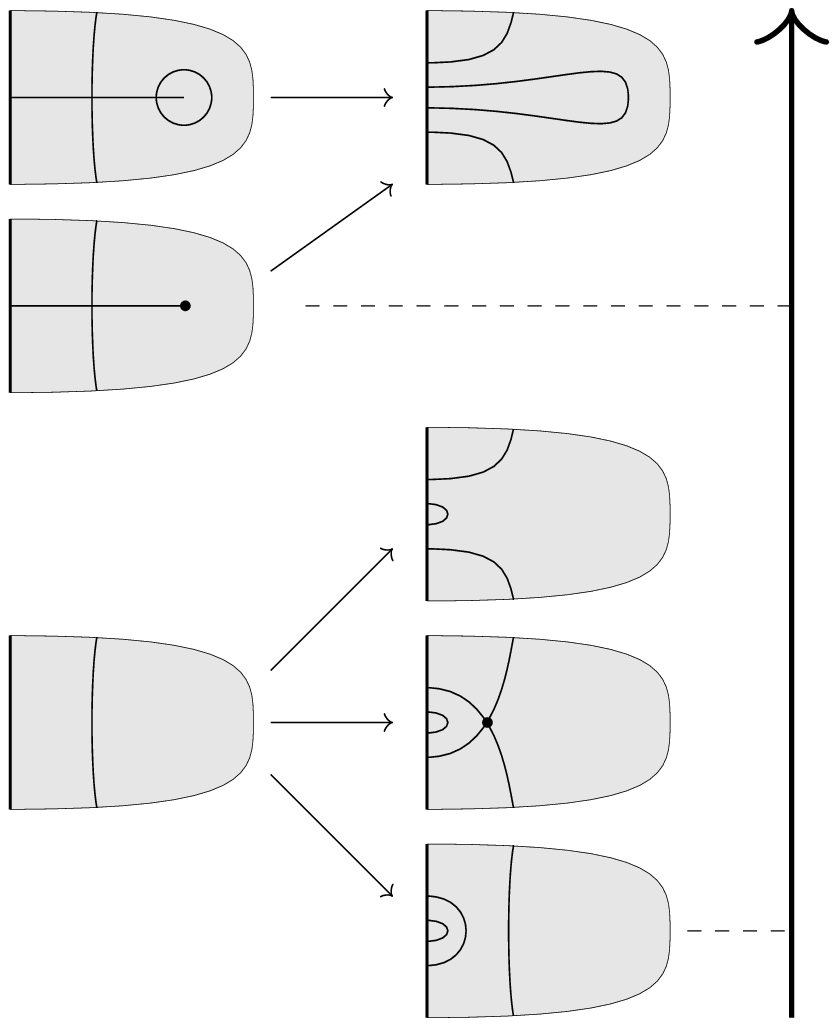}\put(-5,203){$\theta$}\put(-8,18){\small$\theta_1-\delta$}%
\put(-8,144){\small$\theta_1$}\put(-165,192){\small$\nabla$}
\caption{Changes in the sections $F\cap\mathscr P_\theta$ occurring as a result of a finger move.
Pictures on the left illustrate the movable disc case. The pictures on the right
correspond to the situation when $\theta|_{F\setminus\mathbb S^1_{\tau=0}}$ has a local minimum
in $\mathscr P_{\theta_1}$. The small arcs created as a result of the move are preserved in $\mathscr
P_\theta$ for
$\theta\in[\theta_2+\delta,\theta_1-\delta]$.}\label{sections}
\end{figure}
Since the discs in $Y$ are disjoint, the corresponding finger moves can
be done so that they do not
interfere with each other.

As a result of this step we have that the intersection of $F$ with every page
$\mathscr P_\theta$ is a union of pairwise disjoint arcs with endpoints on $\mathbb S^1_{\tau=0}$
and possibly a star graph having a $n$-valent vertex, $n>2$, in the interior of $\mathscr P_\theta$
and $1$-valent vertices at $\mathbb S^1_{\tau=0}$. There may be only finitely many pages
containing such a star graph. By a small deformation of $F$, if necessary,
we can ensure that the pages containing the star graphs do not contain
edges of the graph $\widehat{X_2}$.

\smallskip
\noindent\emph{Step 3.} At this and the next step we apply an isotopy to $F$
such that the combinatorics of the intersection of $F$ with each page~$\mathscr P_\theta$
remains unchanged. By such an isotopy
we can achieve the following:
\begin{enumerate}
\item
$F$ has only finitely many intersections with $\mathbb S^1_{\tau=1}$ and is orthogonal to
$\mathbb S^1_{\tau=1}$ at all those points;
\item
whenever $F\cap\mathscr P_\theta$ contains a star-like graph this graph
must have the form $\widehat Q$, where $Q$ is a finite subset
of the meridian $\{\theta\}\times\mathbb S^1$.
This means, in particular, that all critical points of $\theta|_{F\setminus\mathbb S^1_{\tau=0}}$
occur at $\mathbb S^1_{\tau=1}$;
\item
whenever a connected component $\beta$ of $F\cap\mathscr P_\theta$
contains the center $\mathscr P_\theta\cap\mathbb S^1_{\tau=1}$ of the page $\mathscr P_\theta$,
we have that $\beta$ is either a star graph or an arc of the form $\widehat{v_1}\cup\widehat{v_2}$
with $v_1,v_2\in\{\theta\}\times\mathbb S^1$;
\item
for any $\theta\in\mathbb S^1$, any connected component $\beta$ of $F\cap\mathscr P_\theta$
not containing the center of the page $\mathscr P_\theta$ intersects
any arc of the form $\widehat v$, $v\in\{\theta\}\times\mathbb S^1$, at most once;
\item
every connected component of $F$ intersects $\mathbb S^1_{\tau=1}$.
\end{enumerate}

Fig.~\ref{isotopy3} demonstrates the idea of such an isotopy. First, we pull all
\begin{figure}[ht]
\includegraphics[scale=0.8]{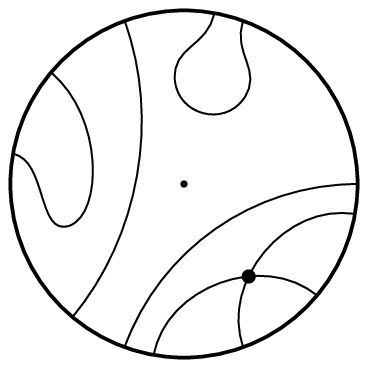}\hskip0.5cm\raisebox{47pt}{$\longrightarrow$}\hskip0.5cm
\includegraphics[scale=0.8]{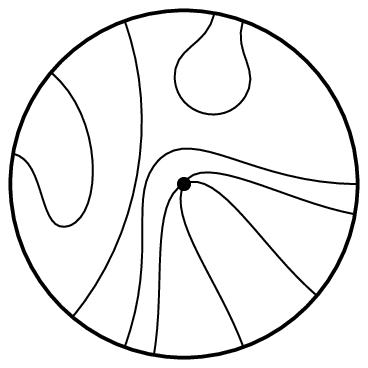}\hskip0.5cm\raisebox{47pt}{$\longrightarrow$}\hskip0.5cm
\includegraphics[scale=0.8]{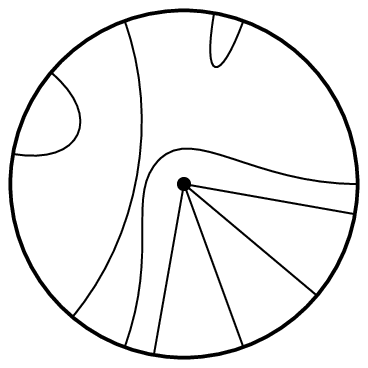}
\caption{Isotopy at Step~3}\label{isotopy3}
\end{figure}
critical points of $\theta|_{F\setminus\mathbb S^1_{\tau=0}}$ to
the circle $\mathbb S^1_{\tau=1}$. This may require to introduce
more intersection points of $F$ and $\mathbb S^1_{\tau=1}$.
If there is a spherical connected component of $F$ that does not
intersect $\mathbb S^1_{\tau=1}$
(non-spherical ones need not to be taken care of since they
must already intersect $\mathbb S^1_{\tau=1}$ by this moment)
we also pull it to this
circle to create an intersection. We keep fixed $F\cap\mathbb S^1_{\tau=1}$ from now on.

To obtain the desired isotopy we proceed as follows. First, we make $F$ orthogonal to $\mathbb S^1_{\tau=1}$ at all the intersections
and straighten the arcs joining the points from $F\cap\mathbb S^1_{\tau=1}$
with $\mathbb S^1_{\tau=0}$. This may include a small deformation in the $\theta$-direction, in particular, to
ensure that the arcs have a proper direction at $F\cap\mathbb S^1_{\tau=1}$. All the other deformations
at this step can be made `pagewise'.

Then we pull tight all the arcs in each intersection $F\cap\mathscr P_\theta$ that do not cross $S^1_{\tau=1}$
so as to ensure Condition~(4) above. The fact that there is no topological obstruction to do this in
all pages simultaniously follows from Smale's theorem on contractibility of the group of diffeomorphisms
of a $2$-disc fixed at boundary.

We are ready to produce the sought-for diagram $\Pi$.
For every connected component $\beta$ of $F\cap\mathscr P_\theta$
that has the form of an arc not passing through the center of the page $\mathscr P_\theta$,
there are $\varphi_1,\varphi_2\in\mathbb S^1$ such that $\beta$ intersects $\widehat v$, $v\in\theta\times\mathbb S^1$,
if and only if $v\in\{\theta\}\times[\varphi_1,\varphi_2]$. We will express this
by saying that $\beta$ \emph{covers} the interval $[\varphi_1,\varphi_2]$.

If a page $\mathscr P_\theta$ contains an arc covering $[\varphi_1,\varphi_2]$,
then so does the page $\mathscr P_{\theta'}$ for any $\theta'$ close enough
to~$\theta$. Moreover, if the page $\mathscr P_{\theta_1}$ contains
an arc covering the interval $[\varphi_1,\varphi_2]$ and
$\mathscr P_{\theta_2}$ does not do so, then there is
an intersection point $F\cap\mathbb S^1_{\tau=1}$
in each of the regions $\theta\in(\theta_1,\theta_2)$
and~$\theta\in(\theta_2,\theta_1)$.

Since there are only finitely many points in $F\cap\mathbb S^1_{\tau=1}$,
for any fixed $\varphi_1,\varphi_2\in\mathbb S^1\cap F$, $\varphi_1\ne\varphi_2$,
the set of all $\theta$ such that the page $\mathscr P_\theta$
contains an arc covering the interval $[\varphi_1,\varphi_2]$ is a union
of finitely many pairwise disjoint intervals $(\theta_1',\theta_1''),\ldots,(\theta_k',\theta_k'')$,
$k\geqslant0$. Indeed, it is an open subset of $\mathbb S^1$, but it cannot be the whole of $\mathbb S^1$
as this would mean that $F$ has a spherical component that does not intersect $\mathbb S^1_{\tau=1}$.

We then include in $\Pi$ the rectangles $[\theta_i',\theta_i'']\times[\varphi_1,\varphi_2]$, $i=1,\ldots,k$,
and do so for all possible choices of $\varphi_1,\varphi_2$ (there are only finitely many of them
since $\mathbb S^1_{\tau=0}\cap F$ is finite).

\begin{lemm}
The set  $\Pi$ of rectangles obtained in this way is a rectangular
digram of a surface.
\end{lemm}

\begin{proof}
To see this, we first examine a neighborhood of a meridian $\{\theta_0\}\times\mathbb S^1$
such that $F$ passes through the center of the page $\mathscr P_{\theta_0}$. Such
meridians are the only ones that can contain vertical sides of the rectangles from $\Pi$.
Assume for the moment that the center $p$ of this page is an interior point of $F$.

The intersection $F\cap\mathscr P_{\theta_0}$ contains a singular component $\Gamma$
that has the form $\widehat{v_1}\cup\widehat{v_2}\cup\ldots\cup\widehat{v_{2k}}$,
where $k\in\mathbb N$, and $v_i=(\theta_0,\varphi_i)$, $i=1,\ldots,2k$,
are $2k$ pairwise distinct points of the meridian $\{\theta_0\}\times\mathbb S^1$.
We number them in the cyclic order they appear on the meridian
so that the intervals $(\varphi_1,\varphi_2),\ (\varphi_2,\varphi_3),\ldots,(\varphi_{2k},\varphi_1)$
are pairwise disjoint.

Any other component of $F\cap\mathscr P_{\theta_0}$ is an arc covering a subinterval
of $(\varphi_i,\varphi_{i+1})$, $i\in\{1,\ldots,2k\}$, where we put $\varphi_{2k+1}=\varphi_1$.
When $\theta$ deviates from $\theta=\theta_0$ the singular
component $\Gamma$ gets resolved into a family of arcs. When it deviates
one way (say, $\theta<\theta_0$) $\Gamma$ resolves to
arcs covering all the intervals $[\varphi_i,\varphi_{i+1}]$ with odd $i$'s,
and when $\theta$ deviates the other way $\Gamma$ resolves to
such arcs with even $i$'s.

Thus, the rectangles from $\Pi$ that have a vertical side on $\{\theta_0\}\times\mathbb S^1$
will be arranged near this meridian in a checkerboard manner, see the left picture
in Fig.~\ref{checkerboard}.
\begin{figure}[ht]
\includegraphics{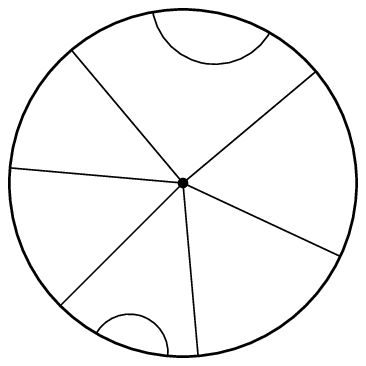}\put(-13,90){$\varphi_1$}\put(-98,100){$\varphi_2$}%
\put(-119,60){$\varphi_3$}\put(-103,13){$\varphi_4$}\put(-54,-4){$\varphi_5$}\put(-7,28){$\varphi_6$}%
\put(-43,20){$\mathscr P_{\theta_0}$}\hskip1cm
\includegraphics{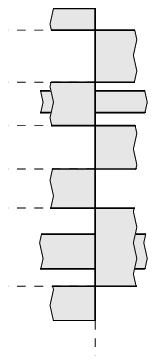}\put(-22,-3){$\theta_0$}\put(-55,23){$\varphi_1$}\put(-55,46){$\varphi_2$}\put(-55,57){$\varphi_3$}%
\put(-55,69){$\varphi_4$}\put(-55,82){$\varphi_5$}\put(-55,98){$\varphi_6$}
\hskip2cm\includegraphics{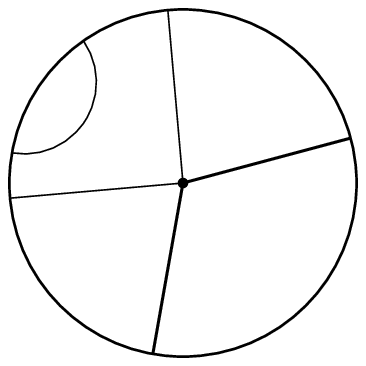}\put(-43,20){$\mathscr P_{\theta_0}$}\put(-4,66){$\varphi_1$}\put(-67,110){$\varphi_2$}%
\put(-119,49){$\varphi_3$}\put(-70,-3){$\varphi_4$}\put(-36,51){$\partial F$}\put(-58,25){$\partial F$}\hskip1cm
\includegraphics{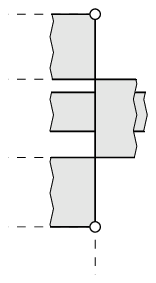}\put(-22,-3){$\theta_0$}\put(-55,17){$\varphi_1$}\put(-55,37){$\varphi_2$}\put(-55,59){$\varphi_3$}%
\put(-55,78){$\varphi_4$}
\caption{The rectangles from $\Pi$ near a `critical' meridian}\label{checkerboard}
\end{figure}
Note that all other rectangles intersecting the meridian $\{\theta_0\}\times\mathbb S^1$
will be disjoint from the longitudes $\mathbb S^1\times\{\varphi_i\}$, $i=1,\ldots,2k$.
In other words, vertices of rectangles cannot lie in the interior of other rectangles.

If $p$ lies at the boundary $\partial F$, we will have a similar picture with
the only difference that the vertical sides of the rectangles now cover
not the whole meridian $\{\theta_0\}\times\mathbb S^1$ but a proper subarc,
see the right picture in Fig.~\ref{checkerboard}.

One can now see that any two rectangles in $\Pi$ are compatible.
Indeed, let $r_1=[\theta_1',\theta_1'']\times[\varphi_1',\varphi_1'']$ and
$r_2=[\theta_2',\theta_2'']\times[\varphi_2',\varphi_2'']$ be two rectangles from $\Pi$.
Suppose that their interiors have a non-empty intersection.
This means that $(\theta_1',\theta_1'')\cap(\theta_2',\theta_2'')\ne\varnothing$ and
$(\varphi_1',\varphi_1'')\cap(\varphi_2',\varphi_2'')\ne\varnothing$.
For any $\theta\in(\theta_1',\theta_1'')\cap(\theta_2',\theta_2'')$ there are two disjoint arcs
in $F\cap\mathscr P_\theta$ one of which covers $[\varphi_1',\varphi_1'']$
and the other $[\varphi_2',\varphi_2'']$. This is possible only if these intervals are
disjoint or one is contained in the other. The former case does not occur as
the intervals have non-empty intersection.

Thus, without loss of generality, we may assume $[\varphi_1',\varphi_1'']\subset(\varphi_2',\varphi_2'')$.
The meridians $\{\theta_1'\}\times\mathbb S^1$ and $\{\theta_1''\}\times\mathbb S^1$ contain
the vertical sides of $r_1$. As the reasoning above shows,
there can be no arc covering $[\varphi_2',\varphi_2'']$ in
$F\cap\mathscr P_\theta$ if $\theta$ is close enough to $\theta_1'$ or $\theta_1''$,
since otherwise a vertex of $r_1$ would lie in the interior of $r_2$.
This means $\theta_1',\theta_1''\notin[\theta_2',\theta_2'']$,
which implies $[\theta_2',\theta_2'']\subset(\theta_1',\theta_2'')$.
Therefore, $r_1$ and $r_2$ are compatible.

It also follows from the above arguments that two rectangles from $\Pi$
cannot share a nontrivial subarc of their vertical sides. Neither they can
share a nontrivial subarc of horizontal sides as this would imply
a singularity of $F$ at $\mathbb S^1_{\tau=0}$ that is not allowed for
surfaces with corners.

Thus, if the interiors of two rectangles from $\Pi$ are disjoint,
then the intersection of these rectangles will be a subset of their vertices. So,
any two rectangles from $\Pi$ are compatible.

The free vertices of $\Pi$ are by construction such $v\in\mathbb T^2$
that $\widehat v\subset\partial F$. There can be at most two
on each meridian or longitude as the converse would imply that $\partial F$
is a graph having vertices of valence greater than two.
\end{proof}

\smallskip
\noindent\emph{Step 4.}
We now have that the intersection of surfaces $F$ and $\widehat\Pi$ with
every page have the same combinatorial structure. This means that, for
every arc not passing through the center of the page, in each of this intersections, there is an arc in the other
one covering the same interval. Whenever one of the intersections contains
a singular component (i.e.\ a star graph or an arc passing through the
center of the page) the other intersection contains exactly the same one.

An isotopy from $F$ to $\widehat\Pi$ is now obvious.
For every $v\in\mathbb T^2$ either $\widehat v$ is contained in both
$F$ and $\widehat\Pi$ or the sets $\widehat v\cap F$ and
$\widehat v\cap\widehat\Pi$ are finite and have the same number of points.
In the latter case we simply move the points from $F\cap\widehat v$ to the corresponding
points of $\widehat\Pi\cap\widehat v$ along $\widehat v$.

This concludes the proof of Proposition~\ref{isotopy-to-rectangular}.
\end{proof}

\section{Legendrian links and Legendrian graphs}\label{legendrian}

\subsection{Definitions. Equivalence of Legendrian graphs} First, recall some basic things
from contact topology.

The plane field $\xi_+$ defined by~\eqref{xi+} is one of many possible
presentations of \emph{the standard contact structure on $\mathbb S^3$}. For any $p\in\mathbb S^3$
the plane $\xi_+(p)$ is called \emph{the contact plane at $p$}. Diffeomorphisms
$\mathbb S^3\rightarrow\mathbb S^3$ preserving this plane field are called
\emph{contactomorphisms}.

The contact structure $\xi_+$ comes with \emph{a coorientation}, which is defined
by demanding that the $1$-form $\alpha_+=\cos^2(\pi\tau/2)\,d\varphi+\sin^2(\pi\tau/2)\,d\theta$,
which defines $\xi_+$,
evaluates positively on a positively oriented transverse vectors.
The standard (cooriented) contact structure is defined also by any positive multiple
$\alpha=f\alpha_+$ of $\alpha_+$, $f:\mathbb S^3\rightarrow(0,\infty)$, $f\in C^\infty$.
Any such form $\alpha$ will be referred to as \emph{a standard contact form}.
It has the property that $\alpha\wedge d\alpha$ is a positive multiple of the standard
volume element of $\mathbb S^3$ at every point.
By Darboux's theorem
all contact structures
are locally equivalent.

\begin{defi}
A cusp-free piecewise smooth curve $\gamma$ in $\mathbb S^3$ is called \emph{Legendrian}
if it is composed of $C^1$-smooth arcs tangent to the contact plane $\xi_+(p)$ at every point $p\in\gamma$.

\emph{A Legendrian link} is a link in $\mathbb S^3$ each component of which
is a Legendrian curve.

\emph{A Legendrian graph} is a compact $1$-dimensional CW-complex $\Gamma$
embedded in $\mathbb S^3$
such that every embedded curve in $\Gamma$ is Legendrian (in particular,
any such curve must be cusp-free).
\end{defi}

In this paper we understand graphs as topological subspaces in $\mathbb S^3$, so, we do not distinguish between two graphs
if one is obtained from the other by subdividing an edge by new vertices.

For an oriented Legendrian link $L$ one defines \emph{the Thurston--Bennequin number $\tb_+(L)$ of $L$}
as $\lk(L,L')$, where the link $L'$ is obtained from $L$ by a small shift in a direction transverse to $\xi_+$.
This is known to be a Legendrian link invariant. For a link of the form $L=\widehat R$, where
$R$ is a rectangular diagram of a link, $\tb_+(L)$ coincides with $\tb_+(R)$, see Definition~\ref{tb(R)-def}
and Proposition~\ref{tbproperties}.
If $L$ is a Legendrian sublink of $\partial F$, where $F$ is a surface,
one also defines \emph{the Thurston--Bennequin number of $L$ relative to $F$}:
$$\tb_+(L;F)=\tb_+(L)-
\lk(L,L^F),$$
which agrees with Definition~\ref{reltb(R)-def} given in the `rectangular' terms.

\begin{defi}\label{adjustment-def}
Let $\Gamma_1$, $\Gamma_2$ be two Legendrian graphs and $p\in\Gamma_1\cap\Gamma_2$ a point. We say that
$\Gamma_1$ and $\Gamma_2$ are obtained from each other by \emph{an angle adjustment at $p$} if
there is a local coordinate system $x,y,z$ near $p$
and an $\varepsilon>0$ such that
\begin{enumerate}
\item
$p$ is the origin of the system $x,y,z$;
\item
$dz+x\,dy-y\,dx$
is a standard contact form near $p$;
\item
$\Gamma_1$ and $\Gamma_2$ coincide outside of the neighborhood
$U$ of $p$ that is defined by $x^2+y^2<\varepsilon$, $|z|<\varepsilon$,
and in a small tubular neighborhood of $\partial U$;
\item
for each $i=1,2$, the closure of the intersection $\Gamma_i\cap U$
is a simple arc containing $p$ and having endpoints at the cylinder $C=\{(x,y,z)\;;\;x^2+y^2=\varepsilon$\},
or a star graph whose internal node is $p$
and whose one-valent vertices lie in $C$;
\item
for each $i=1,2$, the projection of $\Gamma_i\cap U$
to the $xy$-plane (which is the contact plane at $p$)
along the $z$-direction is injective;
\item
for each $i=1,2$ the closure of each connected component of $(\Gamma_i\cap U)\setminus\{p\}$
is a smooth arc.
\end{enumerate}

If $\Gamma_2'$ and $\Gamma_2''$ are two Legendrian graphs each of which is obtained
from $\Gamma_1$ by an angle adjustment at a vertex $p\in\Gamma_1$ of valence $>1$ and the edges of $\Gamma_2'$
emanating from $p$ have the same tangent rays as the respective edges of $\Gamma_2''$
we say that \emph{the angle adjustments $\Gamma_1\mapsto\Gamma_2'$ and $\Gamma_1\mapsto\Gamma_2''$
are equivalent}.
\end{defi}

\begin{prop}\label{adjustment}
{\rm(i)} Let $\Gamma$ be a Legendrian graph and $p\in\Gamma$ its vertex of valence $k>1$.
Let $r_1,\ldots,r_k$ be the tangent rays to the edges of $\Gamma$ at $p$.
Let $r_1',\ldots,r_k'$ be another family of pairwise distinct rays emanating from $p$ and lying in the contact
plane $\xi_+(p)$. Assume that $r_1',\ldots,r_k'$ follow in this plane in the same
circular order as $r_1,\ldots,r_k$ do.

Then for any neighborhood $U$ of $p$ there exists an angle adjustment $\Gamma\mapsto\Gamma'$
such that $\Gamma\setminus U=\Gamma'\setminus U$ and if an edge
of $\Gamma$ is tangent to $r_i$ at $p$, then the respective edge of $\Gamma'$
is tangent to $r_i'$.

{\rm(ii)} Let Legendrian graphs $\Gamma'$ and $\Gamma''$ be obtained from
a Legendrian graph $\Gamma$ by equivalent angle adjustments.
Then there exists a contactomorphism $\psi:\mathbb S^3\rightarrow\mathbb S^3$
such that $\psi(\Gamma')=\Gamma''$.
\end{prop}

\begin{proof}
(i) Let $x,y,z$ be a local coordinate system as in Definition~\ref{adjustment-def}.
A Legendrian arc is uniquely recovered from its projection to the $xy$-plane
plus the knowledge of the $z$-coordinate of a single point, since
along the arc we have $dz=-x\,dy+y\,dx$.

Let $V\subset U$ be a small cylindrical neighborhood of $p$ such that
$\Gamma\cap V$ is projected to the $xy$-plane injectively and
the closure of the projection has the form $\gamma_1\cup\ldots\cup\gamma_k$,
where $\gamma_i$ is a smooth arc connecting~$p$ with a point at $\partial V\cap\{z=0\}$ and tangent
to $r_i$ at $p$, $i=1,\ldots,k$.

We can obviously find another family of arcs $\gamma_1',\ldots,\gamma_k'$ that
also form a star graph in $V\cap\{z=0\}$ such that for any $i=1,\ldots,k$ the following hold:
\begin{enumerate}
\item
$\partial\gamma_i=\partial\gamma_i'$;
\item
$\gamma_i'$ is tangent to $r_i'$ at $p$ and to $\gamma_i$ at the other endpoint;
\item
$\int_{\gamma_i}(x\,dy-y\,dx)=\int_{\gamma_i'}(x\,dy-y\,dx)$ if $\gamma_i$ and $\gamma_i'$ are oriented
coherently.
\end{enumerate}
Now $\Gamma'$ is uniquely defined by demanding that $\Gamma'\setminus U=\Gamma\setminus U$ and
that the projection to the $xy$-plane of~$\Gamma'\cap V$ is $\gamma_1'\cup\ldots\cup\gamma_k'$.

(ii) Let $\gamma_1',\ldots,\gamma_k'$ be as above and $\gamma_1'',\ldots,\gamma_k''$
be another family of arcs with the same properties. One can see that there is an isotopy from
one family to the other such that the tangent rays at the endpoints to each arc as well as
the integral of the form $x\,dy-y\,dx$ along each arc stay fixed, and the interiors of arcs
remain disjoint. This means that
there is a smooth isotopy from $\Gamma'$ to $\Gamma''$ (the latter obtained
similarly to $\Gamma'$ by using arcs $\gamma_i''$ in place of $\gamma_i'$)
in the class of Legendrian graphs such that the tangent rays to the edges
of the graph at $p$ and the intersection of the graph with the exterior of $V$
remain fixed during the isotopy. By smoothness of this isotopy we mean
that it changes only a part of a graph consisting of smooth arcs, and its
restriction to these arcs is smooth.

It is then classical to see that such an isotopy can be realized by an ambient
isotopy in the class of contactomorphisms, see~\cite[Theorem 2.6.2]{Ge}.
\end{proof}

\begin{rema}
The integral of the form $x\,dy-y\,dx$ over a closed curve
in the $xy$-plane is twice the oriented area bounded by
the curve. So, if we want to make an angle adjustment
that rotates a portion of an edge around an endpoint $p$, then some farther
portion must rotate in the opposite direction, see Fig.~\ref{adjustment-pic}.
\begin{figure}[ht]
\includegraphics[width=80pt]{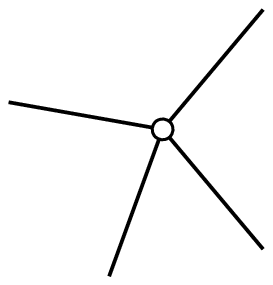}\hskip1cm\raisebox{38pt}{$\longrightarrow$}\hskip1cm\includegraphics[width=80pt]{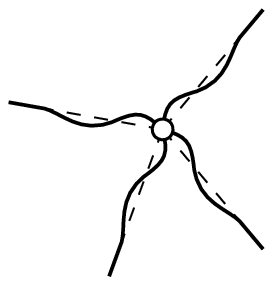}
\caption{Angle adjustment}\label{adjustment-pic}
\end{figure}
\end{rema}

\begin{prop}\label{equivalent-equivalence}
Let $\Gamma_0$ and $\Gamma_1$ be two Legendrian graphs.
The following
two  conditions are equivalent:
\begin{enumerate}
\item
$\Gamma_1$ can be obtained from $\Gamma_0$ by an angle adjustment
followed by a contactomorphism of $\mathbb S^3$;
\item\label{legisotopy3}
there exists an isotopy $\eta:[0,1]\times\Gamma_0\rightarrow\mathbb S^3$ and
a finite collection of regularly parametrized closed arcs $\gamma_1,\ldots,\gamma_k\subset\Gamma_0$
such that
\begin{enumerate}
\item
the arcs $\gamma_i$, $i=1,\ldots,k$, cover the whole of $\Gamma_0$;
\item
$\eta_t(\Gamma_0)$ is a Legendrian graph for any $t\in[0,1]$;
\item
$\eta_0=\mathrm{id}_{\Gamma_0}$ and $\eta_1(\Gamma_0)=\Gamma_1$;
\item
$\{\eta_t(\gamma_i)\;;\;t\in[0,1]\}$ is a $C^1$-continuous
family of regularly parametrized arcs
for each $i=1,\ldots,k$.
\end{enumerate}
\end{enumerate}
\end{prop}

We skip the elementary but boring proof.

\begin{defi}
If the
two equivalent conditions from Proposition~\ref{equivalent-equivalence}
hold we say that \emph{the Legendrian graphs $\Gamma_0$ and~$\Gamma_1$ are equivalent}.
If, in addition, an isotopy $\eta$ from $\Gamma_0$ to $\Gamma_1$ satisfying Condition~(\ref{legisotopy3})
from Proposition~\ref{equivalent-equivalence} exists such that $\eta$ is fixed
on a subset $X\subset\mathbb S^3$ we say that \emph{$\Gamma_0$ and $\Gamma_1$ are equivalent
relative to $X$}.
\end{defi}

\begin{rema}
Condition~(\ref{legisotopy3}) from Proposition~\ref{equivalent-equivalence} has the same
meaning as the definition of the Legendrian graph equivalence given in~\cite{OP}.
Our angle adjustment operation is a generalization of the notion of a canonical smoothing defined in~\cite{EF1,EF2},
which can be considered as an angle adjustment applied to a Legendrian link so that
the result is a smooth Legendrian link. As noted in~\cite{OP} the results of two
different canonical smoothings are related by a contactomorphism, so extending
the class of Legendrian links from smooth to cusp-free piecewise smooth
neither creates any new class of Legendrian links nor merges
any two distinct classes together.
\end{rema}

\subsection{Rectangular diagrams of graphs}
\begin{defi}
\emph{A rectangular diagram of a graph} is a triple $G=(\Theta,\Phi,E)$, where $\Theta$ and $\Phi$
are finite subsets of $\mathbb S^1$ and $E$ is a subset of $\Theta\times\Phi\subset\mathbb T^2$.

We interprete $\Theta$ as a subset of $\mathbb S^1_{\tau=1}\subset\mathbb S^3$ and $\Phi$
as a subset of $\mathbb S^1_{\tau=0}$.
\emph{The graph associated with $G$} is the following union:
$$\widehat G=\Theta\cup\Phi\cup\Bigl(\bigcup_{v\in E}\widehat v\Bigr)\subset\mathbb S^3.$$
\end{defi}

Unless we want to deal with graphs that have isolated vertices we
can forget about $\Theta$ and $\Phi$ in this definition and understand
a rectangular diagram of a graph simply as a finite subset
of $\mathbb T^2$. Indeed, if $\widehat G$ does not have isolated vertices,
then $\Theta$ and $\Phi$ can be recovered from $E$,
since they are just the projections of $E\subset\mathbb S^1\times\mathbb S^1$ to
the $\mathbb S^1$'s. In this paper graphs with isolated vertices
are not needed, but we prefer to give the general definition for future reference.

In the sequel, for simplicity, we assume all considered graphs not to have isolated
vertices, so `a rectangular diagram of a graph' will mean just a finite subset of the
torus $\mathbb T^2$. The following result obviously remains true if isolated vertices are allowed.

Note that all graphs of the form $\widehat G$, where $G$ is a rectangular diagram
of a graph, are Legendrian graphs. They also remain Legendrian
if $\xi_+$ is replaced by $\xi_-$ in the definition.

\begin{prop}\label{leg-appr-pcinciple}
Let $G$ be a rectangular diagram of a graph and $\Gamma$
a Legendrian graph such that $\Gamma\cap\widehat G\subset\mathbb S^1_{\tau=0}\cup\mathbb S^1_{\tau=1}$
and $\Gamma\cup\widehat G$ is also a Legendrian graph. Then there exists
a rectangular diagram of a graph~$G'$ such that the Legendrian graphs
$\widehat{G'}$ and $\Gamma$ are equivalent relative to $\widehat G$.
\end{prop}

\begin{proof}
The only thing we need is to extend the Approximation Principle to Legendrian graphs.
To this end, we look more closely at the torus projections of Legendrian links,
which are similar in nature to widely explored front projections.

Indeed, let $\gamma$ be a Legendrian curve disjoint from $\mathbb S^1_{\tau=0}$ and $\mathbb S^1_{\tau=1}$.
Its torus projection locally is a cusped curve, and $\gamma$ can be
recovered from the torus projection since the $\tau$ coordinate is determined
by the slope $d\varphi/d\theta$:
$$\tau=\frac2\pi\arctan\sqrt{-d\varphi/d\theta}.$$
The slope of the torus projection is everywhere negative and continuous,
and the larger the absolute value of the slope the `higher' is the
corresponding point of the curve. The slope tends to $\infty$
when $\gamma$ approaches~$\mathbb S^1_{\tau=1}$ and to $0$ when
$\gamma$ approaches $\mathbb S^1_{\tau=0}$.

Thus, a Legendrian curve disjoint from $\mathbb S^1_{\tau=0}$ and $\mathbb S^1_{\tau=1}$
can be recovered from its torus projection, and an isotopy in the class
of Legendrian curves can be described in terms of the projection.

Now we show how to produce $G'$. First, we apply an isotopy that will move all
vertices of $\Gamma$ to~$\mathbb S^1_{\tau=0}\cup\mathbb S^1_{\tau=1}$
(this can be realized
by a contactomorphism of $\mathbb S^3$ due to~\cite[Theorem~2.6.2]{Ge})
and make the rest of~$\Gamma$ have only finitely many points of tangency with
arcs of the form $\widehat v$, $v\in\mathbb T^2$. Such tangencies outside of
the circles $\mathbb S^1_{\tau=0,1}$ are responsible for cusps of the torus projection.

Now the torus projection of every edge of $\Gamma$ is a cusped open
arc (in general, with self-intersections)
that approaches two distinct points at the ends, and at each
of these points have either vertical or horizontal slope. Cusps
divide this arc into finitely many subarcs of negative slope.

We approximate each of these subarcs by a staircase arc so that (see Fig.~\ref{staircase}):
\begin{figure}[ht]
\raisebox{20pt}{\includegraphics{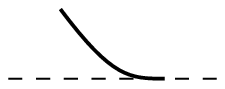}}\hskip0.5cm\raisebox{32pt}{$\longrightarrow$}\hskip0.5cm
\raisebox{20pt}{\includegraphics{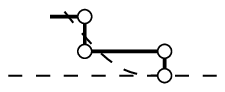}}\hskip1cm
\includegraphics{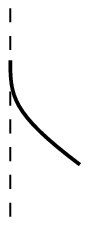}\hskip0.5cm\raisebox{32pt}{$\longrightarrow$}\hskip0.5cm\includegraphics{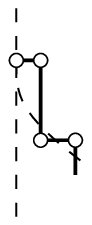}\\
\includegraphics{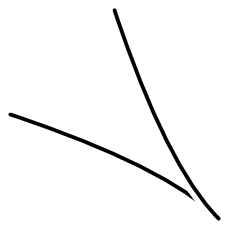}\hskip0.5cm\raisebox{32pt}{$\longrightarrow$}\hskip0.5cm\includegraphics{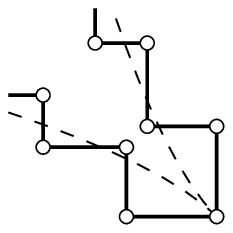}\hskip1cm
\includegraphics{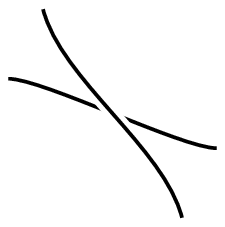}\hskip0.5cm\raisebox{32pt}{$\longrightarrow$}\hskip0.5cm\includegraphics{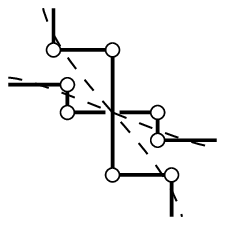}
\caption{Approximating the torus projection of a Legendrian graph by staircases}\label{staircase}
\end{figure}
\begin{enumerate}
\item
there is an isotopy from the subarc to the corresponding staircase
in $\mathbb T^2\setminus G$;
\item
intersections of the staircases occur exactly at intersections of the corresponding subarcs;
\item
the endpoints of the staircase arcs are the same as those of the approximated
subarcs;
\item
if the subarc has vertical (respectively, horizontal) slope at an endpoint,
then the staircase approximation has horizontal (respectively, vertical) one;
\item
at every cusp of the torus projection the corresponding staircases arrive vertically,
and the other horizontally;
\item
no two edges in the staircases occur at the same meridian or longitude, and
no edge occurs on a meridian or longitude already occupied by a vertex of $G$.
\end{enumerate}
If the approximation is fine enough, then
the set of all endpoints of the edges of the obtained staircases works for $G'$. We leave
the easy details to the reader.
\end{proof}

\section{Giroux's convex surfaces}\label{giroux}
\subsection{Definition}

Giroux's original definition is given for closed $C^\infty$-smooth orientable surfaces.
For surfaces with boundary it is naturally generalized for two closely related cases: surfaces
with Legendrian boundary~\cite{honda} and surfaces with
transverse boundary \cite{EVH-M}. We will be dealing only with surfaces with Legendrian boundary.

In order to adapt the language of rectangular diagrams to convex surfaces we need to extend
the class of considered surfaces to surfaces with corners. In particular, we downgrade their smoothness class
to~$C^1$, which will be crucial at some point. Many structural results about
convex surfaces remain true and require no or very little modification of the proof.

Another feature here is that we generalize Giroux's convexity to non-orientable surfaces.
This generalization is quite obvious and straightforward, and many known facts about convex
surfaces can be generalized accordingly at no effort.

However, we will exploit the fact that the standard contact structure of $\mathbb S^3$
is coorientable.

We define convex surfaces only with respect to the standard contact structure on $\mathbb S^3$.
The definition can be obviously extended to an arbitrary contact $3$-manifold.

\begin{defi}
A vector field in an open subset $U$ of $\mathbb S^3$
is called \emph{contact} if it generates a flow in $U$ that
preserves the standard contact structure.

A line element field (i.e.\ a smooth assignment of an unordered pair of opposite
tangent vectors to each point) in $U\subset\mathbb S^3$ is called \emph{contact}
if its restriction to every simply connected open subset $V$ of $U$
has the form $\{u,-u\}$, where $u$ is a contact vector field on $V$.
\end{defi}

\begin{defi}
A surface with corners $F$ is called \emph{convex (with respect to the standard
contact structure)} if $\partial F$
is a Legendrian link and there exists an open neighborhood $U$ of $F$
and a contact line element field in $U$ transverse to $F$.
\end{defi}

Note that to be convex in this sense is a global property. Locally every
surface is convex.

\begin{defi}\label{convex-equiv}
Let $F_0$ and $F_1$ be two convex surfaces with corners. We say that they \emph{are
equivalent} if there exists a $C^0$-isotopy $\eta:[0,1]\times F_0\rightarrow\mathbb S^3$
and finitely many points $p_1,\ldots,p_k\in\partial F_0$ such that
\begin{enumerate}
\item
for every $t\in[0,1]$ the image $\eta_t(F_0)$ is a convex surface with corners;
\item
$\eta_0=\mathrm{id}|_{F_0}$, $\eta_1(F_0)=F_1$;
\item
the restriction of $\eta$ to $[0,1]\times(F_0\setminus\{p_1,\ldots,p_k\})$
is of smoothness class $C^1$;
\item
the tangent plane to $\eta_t(F_0)$ at $\eta_t(p)$ depends continuously on $(t,p)\in[0,1]\times F_0$.
\end{enumerate}
\end{defi}

The following two statements are given for completeness of the exposition and are not used in the
sequel. We skip their proofs, which are easy.

\begin{prop}\label{smoothing-boundary}
Let $F$ be a convex surface with corners, $p_1,\ldots,p_k$ be
all the singularities of $\partial F$. Then for any $\varepsilon>0$
there exists a convex surface with corners $F'$ equivalent to $F$
such that $F'$ coincides with~$F$ outside of the union of the
$\varepsilon$-neighborhoods of $p_1,\ldots,p_k$,
and $\partial F'$ is obtained from $\partial F$ by
a canonical smoothing.
\end{prop}

\begin{prop}
Let $F_0$ and $F_1$ be convex surfaces with smooth boundaries.
Then they are equivalent if and only if
there exists a smooth isotopy between them within the class
of convex surfaces.
\end{prop}

\subsection{`Rectangular' surfaces are convex}

Here is the main result of this paper.

\begin{theo}\label{main2}
{\rm(i)} For every rectangular diagram of a surface $\Pi$ the associated surface
$\widehat\Pi$ is convex with respect to $\xi_+$ (and, by symmetry,
with respect to $\xi_-$). For any connected component $S$
of $\partial\Pi$ we have
\begin{equation}\label{lengthbound}
|S|\geqslant-2\tb_+(\widehat S;\widehat\Pi).
\end{equation}

{\rm(ii)} Let $R$ be a rectangular diagram of a link,
$L$ a Legendrian link in $\mathbb S^3$ equivalent to
$\widehat R$, and
$F$ a convex surface with corners such that $L\cap F$ is a sublink of $\partial F$.
Suppose that for any connected component $K$ of $L\cap F$ and the respective
component $S$ of~$R$ we have $|S|\geqslant-2\tb_+(K;F)$.

Then there exists a rectangular diagram of a surface $\Pi$ and an isotopy taking the pair $(L,F)$ to $(\widehat R,\widehat\Pi)$ and realizing an equivalence
of the convex surfaces $F$ and $\widehat\Pi$ in the sense of Definition~\ref{convex-equiv}.
\end{theo}

\begin{rema}
The assertion of the theorem can be strengthen by requiring that the isotopy
from $(L,F)$ to $(\widehat R,\widehat\Pi)$ extend a prescribed isotopy from
$L$ to $\widehat R$ satisfying Condigion~\eqref{legisotopy3} from Proposition~\ref{equivalent-equivalence}. The proof of this requires
essentially no new idea but some more technical details, which we
prefer not to overload the paper by.
\end{rema}

\begin{proof}[Proof of part (i) of Theorem~\ref{main2}]
A contact line element field transverse to $\widehat\Pi$ can be constructed almost explicitly.
Let $r=[\theta_1,\theta_2]\times[\varphi_1,\varphi_2]$ be a rectangle from $\Pi$.
Pick a small $\varepsilon>0$ and take arbitrary two $C^\infty$-functions
$a:[\theta_1-\varepsilon,\theta_2+\varepsilon]\rightarrow[-1,1]$,
$b:[\varphi_1-\varepsilon,\varphi_2+\varepsilon]\rightarrow[-1,1]$
such that:
\begin{enumerate}
\item
$a(\theta)=1$ if $\theta\in[\theta_1-\varepsilon,\theta_1+\varepsilon]$,
$b(\varphi)=-1$ if $\varphi\in[\varphi_1-\varepsilon,\varphi_1+\varepsilon]$;
\item
$a(\theta)=-a(\theta_1+\theta_2-\theta)$ for all $\theta$,
$b(\varphi)=-b(\varphi_1+\varphi_2-\varphi)$ for all $\varphi$;
\item
$a'(\theta)<0$ if $\theta\in(\theta_1+\varepsilon,\theta_2-\varepsilon)$;
$b'(\varphi)>0$ if $\varphi\in(\varphi_1+\varepsilon,\varphi_2-\varepsilon)$.
\end{enumerate}
Such functions obviously exist. The following line element field
defined in the domain
$$W_r=[\theta_1-\varepsilon,\theta_2+\varepsilon]*[\varphi_1-\varepsilon,\varphi_2+\varepsilon]\subset{\mathbb S^1}*\mathbb S^1=\mathbb S^3$$
is then contact:
\begin{equation}\label{lineelem}
l_r=\pm\Bigl(a(\theta)\,\partial_\theta+b(\varphi)\,\partial_\varphi+\frac{\sin{\pi\tau}}{2\pi}\bigl(b'(\varphi)-a'(\theta)\bigr)\,\partial_\tau\Bigr),
\end{equation}
which is verified by a direct check. One can also see that it is transverse to the tile $\widehat r$. Indeed, all three
summands in the right hand side define vector fields that point at the same side of $\widehat r$ unless they vanish.
(This is because the function $h_r$ used in the definition of $\widehat r$ is concave in $\theta$,
convex in $\varphi$, and symmetric in the center of the rectangle.)
The first summand vanishes only at the line $\theta=(\theta_1+\theta_2)/2$, the second at the line $\varphi=(\varphi_1+\varphi_2)/2$,
so, their sum vanishes only at the center of $r$, where the third summand doesn't.

If $r$ and $r'$ are two rectangles in $\Pi$ sharing a vertex $v$, and $p$ is an interior point of $\widehat v$,
then in a small neighborhood of $p$ the line element fields $l_r$ and $l_{r'}$
coincide and have the form either $\pm(\partial_\theta+\partial_\varphi)$ or
$\pm(\partial_\theta-\partial_\varphi)$. One can also see that for any $r\in\Pi$
the restriction of $l_r$ on $\mathbb S^1_{\tau=0}$ coincides with $\pm\partial_\varphi$
(where defined) and on $\mathbb S^1_{\tau=1}$ with $\pm\partial_\theta$,
and that the line element field $l_r$ is continuous where defined.

For each $r\in\Pi$, denote by $V_r$ the intersection of $W_r$ with the
$\varepsilon$-neighborhood of $\widehat r$.
If $\varepsilon$ is chosen small enough then the interiors of $V_r$ and $V_{r'}$ have
a non-empty intersection only if $r$ and $r'$ share a vertex (or two vertices, or four vertices),
and then $l_r=l_r'$ in $V_r\cap V_{r'}$.

Denote by $l$ the line element field in $V=\bigcup_{r\in\Pi}V_r$ whose restriction on $V_r$
coincides with $l_r$. We are almost done, but there are two minor problems with $l$.
First, $V$ does not contain an open neighborhood of $\widehat\Pi$ if $\partial\widehat\Pi\ne\varnothing$. Namely,
the intersection $\partial\widehat\Pi\cap(\mathbb S^1_{\tau=0}\cup\mathbb S^1_{\tau=1})$
does not lie in the interior of $V$. This can be easily resolved by adding more rectangles to $\Pi$ so
as to obtain a new diagram $\Pi'$ such that $\partial\widehat\Pi$
is contained in the interior of $\widehat{\Pi'}$. If we now construct $V$ using $\Pi'$
instead of $\Pi$, then the whole of $\widehat\Pi$ will be contained in the interior of $V$.

The other problem is that $l$ is only $C^0$ at $V\cap(\mathbb S^1_{\tau=0}\cup\mathbb S^1_{\tau=1})$.
(Everywhere else in $V$ it is $C^\infty$.) This issue can be resolved as follows.

In general, if $\alpha$ is a standard contact form,
then any contact field $X$ is recovered uniquely from the function $f=\alpha(X)$,
which can be an arbitrary smooth function.
This is done by solving at every point the following non-degenerate linear system
in $(X,c)$, where $c$ is an additional variable function:
\begin{equation}\label{linear}
\alpha(X)=f,\quad d\alpha(X,\cdot)+df+c\alpha=0.
\end{equation}

If $f$ is varied slightly in the $C^1$-topology, then the corresponding $X$ is varied slightly
in $C^0$.

Now the point is that the evaluation of the $1$-form in the right
hand side of~\eqref{xi+} on the line element field~\eqref{lineelem}
gives a $C^1$-smooth (two-valued) function. We can approximate it by a $C^\infty$
function making a $C^1$-small perturbation. This will give a $C^\infty$ line element
field $C^0$-close to $l$, which is enough to ensure that the perturbed field
is still transverse to $\widehat\Pi$.

Inequality~\eqref{lengthbound} follows from Proposition~\ref{tbproperties}.
\end{proof}

To prove part~(ii) of Theorem~\ref{main2} we need some more preparatory work.

\subsection{Characteristic foliation}
As shown in \cite{Gi} in order to see that a surface is convex in Giroux's sense
it is enough to know the characteristic foliation of the surface, which we define below,
and it
is easier to examine this foliation than to seek for a transverse
contact line element field.

\begin{defi}
Let $F$ be a surface with corners. \emph{The characteristic foliation}
of $F$ is the singular foliation defined by the line field
$T_pF\cap\xi_+(p)$, where
$T_pF$ denotes the tangent plane to $F$ at $p$.
This foliation will be denoted by $\mathscr F(F)$.
We endow it with the coorientation that is inherited from the coorientation of $\xi_+$.

The foliation $\mathscr F(F)$ will be called \emph{nice} if it has only finitely many
singularities, i.e.\ points where $F$ is tangent to $\xi_+$, and each
of them is either an elliptic singularity (also called \emph{a node}) near which the foliation is radial, or a (simple) saddle.
Additionally, we require that if $p$ is a singularity of $\partial F$, then it is a node of $\mathscr F(F)$.

The foliation $\mathscr F(F)$ will be called \emph{very nice} if it is nice and
has neither a closed leaf nor a cooriented saddle connection cycle.
\end{defi}

\begin{prop}\label{veryniceprop}
Let $F$ be a convex surface with corners. Then
there exist convex surfaces with corners~$F'$ and $F''$ equivalent
to $F$ such that $\partial F'=\partial F''=\partial F$, the foliations $\mathscr F(F')$
and $\mathscr F(F'')$ are nice and very nice, respectively, and,
moreover, the surface $F'$ can be chosen arbitrarily $C^1$-close to $F$.
The surface $F''$ can be chosen $C^0$-close to $F$.
\end{prop}

\begin{proof}These facts are classical for $C^\infty$ surfaces \cite{Gi}. For completeness
of the exposition we provide here a proof with some details omitted.

By a $C^1$-small perturbation of $F$, which can be made within
the class of convex surfaces, we can make the set
of points where $F$ is tangent to $\xi_+$, which will be the singularities
of $\mathscr F(F)$, finite.

We can also make the singularities of $\mathscr F(F)$ be of `Morse type',
i.e.\ nodes and saddles. However, nodes require a little more care as typically
the foliation is not radial around them, see Fig.~\ref{spirals}.
\begin{figure}[ht]
\includegraphics{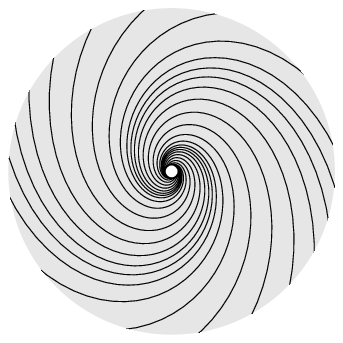}\hskip 2cm\includegraphics{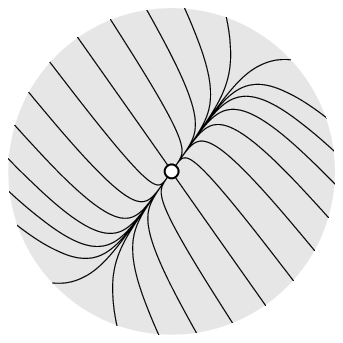}
\caption{Typical nodes of $\mathscr F(F)$}\label{spirals}
\end{figure}
It can be made radial as follows.

Let $p$ be an elliptic singularity in the interior of $F$ and $x,y,z$ local coordinates near $p$
such that $dz+x\,dy-y\,dx$ is a standard contact form,
and $(0,0,0)$ corresponds to $p$. The surface $F$ is locally
a graph of a function $z=f(x,y)$ with a critical point at $p$,
and we have $df_p=0$.

We can modify $f$ in a small neighborhood of $p$ so as
to make it locally constant near $p$.
This can be done so that the new graph $z=f(x,y)$
will have no new tangencies with the contact structure.
For instance, take a small $\varepsilon>0$
and a $C^\infty$-function $h:[0,\infty)\rightarrow[0,1]$ such
that $h(\rho)=0$ if $\rho\in[0,1/2]$, $h(\rho)=1$ if
$\rho\in[1,\infty)$,
and $h'(\rho)>0$ if $\rho\in(1/2,1)$,
and replace $f(x,y)$ with
$\widetilde f(x,y)=h\bigl(\rho/\varepsilon\bigr)f(x,y)$, where $\rho=\sqrt{x^2+y^2}$.

The foliation near $p$ will become radial and no new singularity will be introduced,
which can be seen as follows. New singularities occur at points $(x,y)\ne p$ where the $1$-form $d\widetilde f(x,y)+x\,dy-y\,dx$
vanishes, and this can only happen in the domain $\rho\in[\varepsilon/2,\varepsilon]$. Equivalently,
new singularities correspond to the zeros of the form $\zeta_\varepsilon=
d\bigl(\frac1{\varepsilon^2}\widetilde f(\varepsilon x,\varepsilon y)\bigr)+x\,dy-y\,dx$
in the domain $\rho\in[1/2,1]$. Since this domain is compact, it suffices for us to show that $\zeta_\varepsilon$
does not vanish in it in the limit $\varepsilon\rightarrow0$. The function $\lim_{\varepsilon\rightarrow0}f(\varepsilon x,\varepsilon y)/\varepsilon^2$
is a quadratic form, which can be assumed, without loss of generality, to have the form $(ax^2+by^2)/2$. The elliptic type
of the singularity at $p$ implies $ab>-1$. Now it is a direct check, which is easier if one switches to polar coordinates $\rho,\phi$,
that the $1$-form $\lim_{\varepsilon\rightarrow0}\zeta_\varepsilon=d\bigl(h(\rho)(ax^2+by^2)/2\bigr)+\rho^2\,d\phi$
does not vanish in the domain $\rho\in[1/2,1]$ if $ab>-1$.

If a node $p$ occurs at $\partial F$ we proceed similarly,
but choose the coordinate system more carefully so that
the two arcs of $\partial F$ emanating from $p$ locally have
the form $y/x=\mathrm{const}$, $z=0$, which
is always possible.

If $p$ is a singularity of $\partial F$ and is not a node of $\mathscr F(F)$
we proceed the same way, but this may create
new singularities inside $F$, which should be taken care of as before.

This explains how the surface $F'$ can be produced. In order to produce
$F''$ we need to kill circular leaves and cooriented saddle connection cycles
if there are any. This is done
by using Giroux's elimination lemma in the backward direction.

Namely, let $p$ be a point where $F$ is transverse to the contact structure, and $l$
a contact line element field in a tubular neighborhood of $F$ that is transverse to $F$.
Denote by $s$ the leaf of $\mathscr F(F)$ that comes through $p$.
We may assume that $l$ is transverse to $\xi_+$ at $p$ as this can be achieved by
a small perturbation of~$l$ (if $s$ is a closed leaf or a part of a cooriented saddle connection cycle,
then actually $l$ \emph{must} be transverse to $\xi_+$ along $s$ as we will see in a moment for the case of
a closed leaf).

We can rotate the tangent plane $T_pF$ around $s$ so that it remains
transverse to $l$ and at some moment becomes tangent to $\xi_+$.
We then rotate it a little further, see Fig.~\ref{elim}.
\begin{figure}[ht]
\raisebox{14pt}{\includegraphics{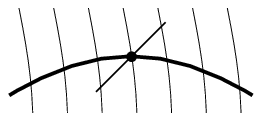}\put(-5,15){$F$}\put(-70,30){$l$}%
\put(-38,15){$p$}\put(-50,30){\tiny$\xi_+(p)$}\hskip0.5cm\raisebox{18pt}{$\longrightarrow$}\hskip0.5cm\includegraphics{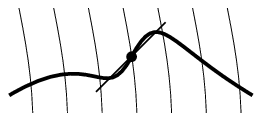}}
\hskip1cm
\includegraphics{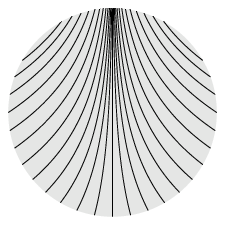}\hskip0.5cm\raisebox{32pt}{$\longrightarrow$}\hskip0.5cm\includegraphics{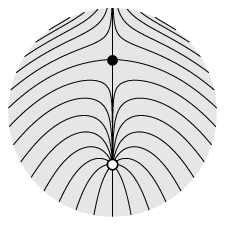}
\caption{Creating a node--saddle pair of singularities}\label{elim}
\end{figure}
This rotation can be realized by a $C^0$-small deformation of $F$ in the class of $C^1$-surfaces
transverse to $l$ so that two singularities, one node and one saddle,
are created. They break up a closed leaf or a saddle connection cycle
if one comes through $p$.

If $s$ is a closed regular leaf, then this operation can be done so that no new closed leaves `parallel' to $s$ are created.
Indeed, in a tubular neighborhood $U$ of $s$, we can introduce a positively oriented
local coordinate system $(x,y,t)\in(-\varepsilon,\varepsilon)\times(-1,1)\times\mathbb S^1$, in which the surface $F$
is defined by the equation $x=0$, the leaf $s$ by the system $x=0$, $y=0$,
the contact line element $l$ is $\pm\partial_x$, and a standard contact form $\alpha$ can
is written as $a(y,t)\,dx+b(y,t)\,dy+c(y,t)\,dt$. We will have $c(0,t)=0$, $b(0,t)\ne0$ for all $t\in\mathbb S^1$. Without
loss of generality we may assume $b(0,t)>0$.

The contact condition $\alpha\wedge d\alpha>0$ reads
\begin{equation}\label{abc}a_t'b-ab_t'+ac_y'-a_y'c>0.\end{equation}
Whenever $a(0,t)=0$ we must have $a_t'(0,t)>0$, which means that the sign of $a$
can change from `$-$' to `$+$' but not back when $t$ runs over the circle, hence, it does not change at all
and we have $a(0,t)\ne0$ for all $t$. This means that $l$ is transverse to $\xi_+$ along $s$.
By choosing a smaller neighborhood if necessary we can ensure that $l$ is transverse to $\xi_+$
everywhere in $U$.

Now we claim that there can be no other closed leaf in $U$.
Indeed, \eqref{abc} means that the $2$-form $d(\alpha/a)$, which is well defined if $a\ne0$, does not
vanish in $U\cap F$, so by Stokes' theorem there can be no closed curve $s'\subset U\cap F$
such $s$ and $s'$ cobound an annulus and $\alpha|_{s'}=0$.

This also implies that the neighborhood $U$ can be chosen so that $\partial U$ is transverse to
the leaves of~$\mathscr F(F)$. Now if we create a node-saddle pair on $s$ as described above so
that $F\setminus U$ is untouched, then no new closed leaf can be created. Indeed,
otherwise we can cancel the node-saddle pair back keeping the newly created leaf unchanged
thus producing an annulus bounded by two closed leaves and such that $l$
is tranverse to~$\xi_+$ everywhere in the annulus, which is impossible.
\end{proof}

Let $F$ be a surface with corners such that $\mathscr F(F)$
is very nice. Suppose that $\partial F$ is Legendrian.
Then every regular leaf of $\mathscr F(F)$ is an arc
approaching singularities of the foliation at the ends.
Indeed, we have forbidden closed leaves and saddle connection cycles
that could have served as a limit cycle for a leaf.

Let $\gamma$ be a regular leaf. Its closure $\overline\gamma$ is a closed Legendrian $C^1$-smooth
arc such that $F$ is tangent to~$\xi_+$ at the endpoints. When
$p$ traverses $\gamma$ the contact plane $\xi_+(p)$ rotates
relative to the tangent plane~$T_pF$ by $-\pi$, $0$, or~$\pi$. We express this by saying
that $\gamma$ is a $-1$-arc, $0$-arc, or $1$-arc, respectively.

All but finitely many leaves are those that connect two nodes.
Each such leaf is a $-1$-arc for the (co)orientation reason.
For a similar reason
a leaf connecting a node to a saddle can be only a $-1$-arc
or a $0$-arc. A saddle connection may be of any type, $-1$, $0$, or $1$,
see Fig.~\ref{arcs}.
\begin{figure}[ht]
\includegraphics{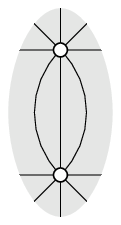}\put(-29,32){$\scriptscriptstyle-1$}\hskip1cm
\includegraphics{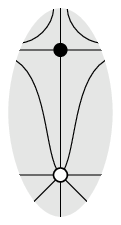}\put(-29,32){$\scriptscriptstyle-1$}\hskip1cm
\includegraphics{arc2.eps}\put(-24.5,32){$\scriptscriptstyle0$}\hskip1cm
\includegraphics{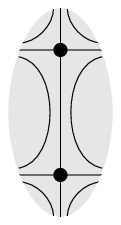}\put(-29,32){$\scriptscriptstyle-1$}\hskip1cm
\includegraphics{arc3.eps}\put(-23.5,32){$\scriptscriptstyle0$}\hskip1cm
\includegraphics{arc3.eps}\put(-23.5,32){$\scriptscriptstyle1$}
\caption{Possible types of regular leaves of a very nice foliation}\label{arcs}
\end{figure}

The following statement gives a simple criteria for a surface to be convex if its
characteristic foliation is very nice. It generalizes the similar
criteria of~\cite{Gi} for oriented $C^\infty$ surfaces. The criteria
for weaker restrictions given in~\cite{Gi}
can also be extended to surfaces with corners (not necessarily oriented),
but we don't need it here.

\begin{prop}\label{no1-arcs}
Let $F$ be a surface with corners such that the characteristic
foliation $\mathscr F(F)$ is very nice. Then $F$ is convex if and only
if $\mathscr F(F)$ has no leaf that is a $1$-arc.
The claim remains true if we allow $\mathscr F(F)$ to have multiple saddles.
\end{prop}

\begin{proof}
The extension of Giroux's proof to surfaces with corners (and not necessarily oriented)
is quite straightforward, so we give again only a sketch.

First, we show that the presence of a $1$-arc among the leaves of $\mathscr F(F)$
means that the surface $F$ is not convex. Let $\gamma$
be such an arc endowed with a regular parametrization. Assume that there exists a contact vector
field $X$ transverse to $F$ in a small neighborhood of $\gamma$.
Denote by $\alpha$ a standard contact form.

After a small perturbation if necessary we may ensure that $0$
is not a critical value of the function $f=\alpha(X)$ restricted to $\gamma$.
The rotation of the contact plane by $\pi$ relative to the tangent plane
means that at some point of $\gamma$ the function $f$ vanishes,
since it is positive at one end of $\gamma$ and negative at the other.
This occurs if and only if the plane $\langle\dot\gamma,X\rangle$ spanned
by $\dot\gamma$ and $X$
becomes contact.

Since $\gamma$ is a Legendrian arc, we get from~\eqref{linear}:
$d\alpha(X,\dot\gamma)+df(\dot\gamma)=0$.
Opposite signs of the summands in the left hand side
at the moment when $f=0$
correspond to rotation of the contact plane relative to the plane
$\langle\dot\gamma,X\rangle$ in the negative direction.
The total rotation with respect to the plane $\langle\dot\gamma,X\rangle$, therefore, must be smaller than $\pi$,
which is incompatible with rotation by $\pi$ relative to the tangent
plane, which remains transverse to $\langle\dot\gamma,X\rangle$.
A~contradiction.

Now we proceed with the proof that the absence of $1$-arcs is sufficient
for the surface $F$ to be convex if the foliation $\mathscr F(F)$ is very nice.

Denote by $\mathscr G=\mathscr G(F)$ the union of all the singularities
and $0$-arcs of $\mathscr F(F)$. Let $\alpha$ be a standard contact form.

Denote by $F'$ a small tubular neighborhood of $\mathscr G(F)$ in $F$.
The coorientations induced by $\xi_+$ at singularities of $\mathscr F$ lying in
the same connected component of $F'$ agree along any path
connecting them, which is by construction,
hence $F'$ is orientable even if $F$ is not.
We choose the orientation of $F'$ so that the form $d\alpha$ is positive
at the singularities of $\mathscr F$. (This is possible since the form $\alpha\wedge
d\alpha$ never vanishes, so~$d\alpha$ defines an orientation
of contact planes that agrees with their coorientation.) Let $\omega$
be an orientation form on $F'$.

We orient each $0$-arc so that if $u$ is a positive tangent vector
to the arc then the coorientation of the arc defined by $\omega(u,\cdot)$
coincides with the coorientation of this arc as a leaf of $\mathscr F$.

Now we choose a smooth function $g:\mathbb S^3\rightarrow[1,\infty)$ whose
restriction to each $0$-arc is a regular monotonic function increasing
in the positive direction. This is possible because
the oriented $0$-arcs form no oriented cycle. Indeed,
an oriented saddle connection cycle cannot occur by the definition
of a very nice foliation. If a cycle contains
at least one node, then the two arcs in it emanating from the
node have opposite orientations. This is because
a $0$-arc connecting an elliptic singularity to a saddle is always
oriented toward the saddle.

One can now show that the restriction
of the $2$-form $d\bigl(\exp(\mu g)\alpha\bigr)$ does not vanish at $\mathscr G$
if $\mu>0$ is a large enough constant.
We replace $\alpha$ by $\exp(\mu g)\alpha$ and $F'$ by a possibly smaller
neighborhood of $\mathscr G$ in $F$ such that $d\alpha$ (for the new $\alpha$) does not vanish in $\overline{F'}$.

Since there are no $1$-arcs, the surface $F\setminus\mathscr G(F)$ is foliated
by $-1$-arcs, and the closure of each has both endpoints in $\mathscr G(F)$.
This implies that the surface $F'$ can be chosen so as to satisfy an additional
restriction that $\partial F'$ is transverse to $\mathscr F(F)$.

Recall that for any smooth function $f:\mathbb S^3\rightarrow\mathbb R$
there exists a unique contact vector field $X$ such that $\alpha(X)=f$,
see~\eqref{linear}. We denote this vector field by $X_f$.

It is possible to find a smooth function $\mathbb S^3\rightarrow\mathbb R$ such that:
\begin{enumerate}
\item
$f(p)=1$ if $p\in F'$;
\item
for any $-1$-arc $\gamma$ the restriction of $f$ to $\gamma\setminus\overline{F'}$
has a single critical point, which is a non-degenerate local minimum, and the critical value is $0$.
\end{enumerate}

The contact line element field $X_{\pm\sinh(\nu\sqrt{f})}$ for large enough $\nu$ is then transverse to $F$.
\end{proof}

Now we have a look at the characteristic foliation of the surface $\widehat\Pi$, where $\Pi$ is a rectangular
diagram of a surface. This foliation is never nice as $\widehat\Pi$
is tangent to $\xi_+$ along a $1$-dimensional subcomplex. We can fix this by a small
perturbation of the surface that can be specified explicitly.

Define $\widehat r^\kappa$ and $\widehat\Pi^\kappa$ in the same way as $\widehat r$ and $\widehat\Pi$ in Definition~\ref{tile-def}
with the function $\widetilde h_r$ replaced by
$$\widetilde h_r(v)=(2/\pi)\arctan\bigl(\tan(\pi h_r(v)/2)\bigr)^{1/(2+\kappa)}.$$

\begin{rema}
The surface $\widehat\Pi=\widehat\Pi^0$ is typically of class $C^1$ only,
whereas $\widehat\Pi^\kappa$ is $C^2$ for $\kappa>0$.
\end{rema}

\begin{prop}
For any rectangular diagram of a surface $\Pi$ and any $\kappa>0$ the characteristic foliation
of $\widehat\Pi^\kappa$ is very nice and has no $1$-arcs.
If $\kappa=0$ it has no $1$-arcs either.
\end{prop}

\begin{proof}
Let $r=[\theta_1,\theta_2]\times[\varphi_1,\varphi_2]$ be a rectangle from $\Pi$.
Denote $a=(\theta_1,\varphi_1)$, $b=(\theta_2,\varphi_1)$, $c=(\theta_2,\varphi_2)$,
$d=(\theta_1,\varphi_2)$. For any $\kappa\geqslant0$ the disc $\widehat r^\kappa$
is transverse to all arcs $\widehat v$ with $v\in\interior(r)$. So, there are no singularities
of $\mathscr F(\widehat r^\kappa)$ in the interior of $\widehat r^\kappa$.

Along the arcs $\widehat a,\widehat c\subset\partial\widehat r^\kappa$ the disc
$\widehat r^\kappa$ is tangent to the plane field
\begin{equation}\label{xi-kappa}
\xi_-^\kappa=\ker\bigl(\cos^{2+\kappa}(\pi\tau/2)\,d\varphi-\sin^{2+\kappa}(\pi\tau/2)\,d\theta\bigr),
\end{equation}
which is transverse to $\xi_+$ outside of $\mathbb S^1_{\tau=0}\cup\mathbb S^1_{\tau=1}$
and make a $\pi$-twist relative to $\xi_+$ along arcs of the form~$\widehat v$, $v\in\mathbb T^2$.
Thus, $\widehat a$ and $\widehat c$ are $-1$-arcs.

Along the arcs $\widehat b$ and $\widehat d$ the disc $\widehat r^\kappa$ is tangent to the plane field
\begin{equation}\label{xi+kappa}
\xi_+^\kappa=\ker\bigl(\cos^{2+\kappa}(\pi\tau/2)\,d\varphi+\sin^{2+\kappa}(\pi\tau/2)\,d\theta\bigr).
\end{equation}
If $\kappa>0$, then this field coincides with with $\xi_+$ only at the endpoints and in the middle of
any arc of the form $\widehat v$. The rotation of the field relative to $\xi_+$ along each
half of $\widehat v$ is zero,
so, each of $\widehat b$ and $\widehat d$ consists of the closures of two $0$-arcs.

To see the structure of $\mathscr F(\widehat r^\kappa)$ in the interior of
$\widehat r^\kappa$ it is useful to apply the torus projection, see Fig.~\ref{foliatedtile}.
The obtained foliation in $\interior(r)$ is tangent to the line
field in $r\subset\mathbb T^2$:
$$\ker\Bigl(\cos^{\frac2{2+\kappa}}\bigl((\pi/2)h_r(\theta,\varphi)\bigr)\,d\varphi+
\sin^{\frac2{2+\kappa}}\bigl((\pi/2)h_r(\theta,\varphi)\bigr)\,d\theta\Bigr).$$
All the leaves have negative slope that tends to $0$ when
the leaf approaches a horizontal side of $r$ (away of the endpoints
of the side) and to $\infty$
when the leaf approaches a vertical side.
\begin{figure}[ht]
$$\begin{array}{ccc}
\includegraphics[scale=0.7]{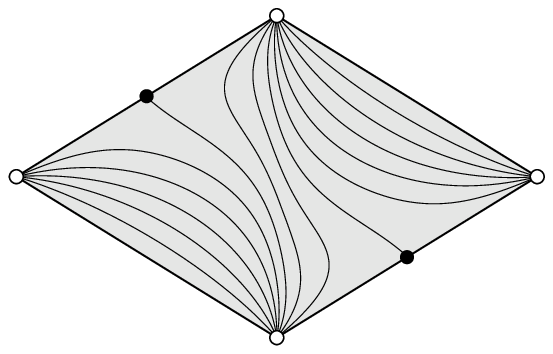}\put(-90,10){$\widehat a$}\put(-90,53){$\widehat d$}\put(-26,10){$\widehat b$}\put(-26,53){$\widehat c$}\hskip0.2cm
\raisebox{10pt}{\includegraphics[scale=0.5]{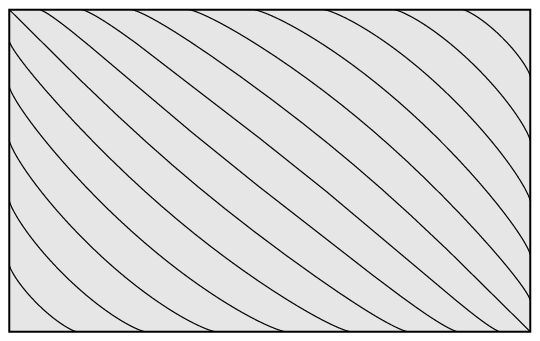}\put(-85,-2){$a$}\put(-85,50){$d$}\put(0,-2){$b$}\put(0,50){$c$}}&
\hskip.5cm&
\includegraphics[scale=0.7]{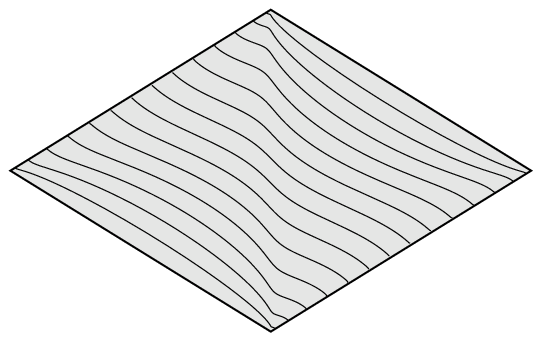}\put(-90,10){$\widehat a$}\put(-90,53){$\widehat d$}\put(-26,10){$\widehat b$}\put(-26,53){$\widehat c$}\hskip0.2cm
\raisebox{10pt}{\includegraphics[scale=0.5]{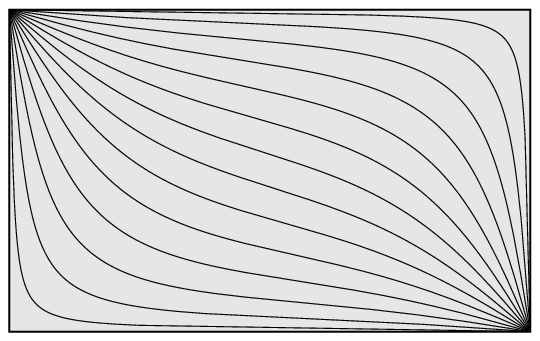}\put(-85,-2){$a$}\put(-85,50){$d$}\put(0,-2){$b$}\put(0,50){$c$}}\\
\kappa>0&&\kappa=0
\end{array}$$
\caption{Foliation $\mathscr F(\widehat r^\kappa)$ and its torus projection}\label{foliatedtile}
\end{figure}
If $\kappa>0$ all leaves, except two or one, end up on internal points
of the sides of $r$, and the exceptional leaves approach
the corners~$b$ or~$d$ and have the slope $-1$ in the limit. If $\kappa=0$ all the leaves
approach corners $b$ and $d$ at different angles.

For the foliation $\mathscr F(\widehat r^\kappa)$, in the case $\kappa>0$, this means that
it has nodes at the corners
of $\widehat r^\kappa$ and saddles in the middle of the arcs $\widehat b$ and $\widehat d$.
All the arcs in the interior of $\widehat r^\kappa$, except two or one separatrices
connect an endpoint of $\widehat b$ with an endpoint of $\widehat d$.

If $\kappa=0$ the sides $\widehat b$ and $\widehat d$ consist of singularities
by whole, and the rest of the tile $\widehat r$ is foliated by
arcs with one endpoint at $\widehat b$ and the other at $\widehat d$.

In both cases all the leaves in the interior of $\widehat r^\kappa$
are $-1$ arcs. This is because the closure of such a leaf
is isotopic to $\widehat a$ in the class of Legendrian curves in $\widehat r^\kappa$
with endpoints at singularities or at $0$-arcs.
\end{proof}

\subsection{The Giroux graph and dividing curves}\label{gg-sec}
Let $F$ be a convex surface with corners
such that the characteristic foliation $\mathscr F(F)$ is very nice.

\begin{defi}
The union $\mathscr G(F)$ of singularities and $0$-arcs of $\mathscr F(F)$
introduced in the proof of Proposition~\ref{no1-arcs}
is called \emph{the Giroux graph of $F$}.

A union $\widetilde{\mathscr G}$ of the Giroux graph with
a finite collection of $-1$-arcs such that $F\setminus\widetilde{\mathscr G}$
is a union of pairwise disjoint open discs
will be called \emph{an extended Giroux graph of $F$}.
\end{defi}

Clearly, both the Giroux graph and any extended Giroux graph are
Legendrian graphs.
The fact that the complement to the Giroux graph is foliated by $-1$-arcs implies
that $F\setminus\mathscr G(F)$ is a union of pairwise disjoint open annuli,
open M\"obius bands, and strips homeomorphic to $(0,1)\times[0,1]$.
Each of them can be cut into discs, so, we have the following.

\begin{prop}
If the foliation $\mathscr F(F)$ is very nice, then
there exists an extended Giroux graph.
Any extended Giroux graph contains $\partial F$.
\end{prop}

Another important thing that can be easily seen is that there is a
collection $\beta_1,\ldots,\beta_k$ of simple closed curves
and proper arcs (meaning: $\beta_i\cap\partial F=\partial\beta$) in $F$
such that:
\begin{enumerate}
\item
for each $i=1,\ldots,k$, the curve $\beta_i$ is transverse to the standard contact
structure;
\item
every $-1$-arc of $\mathscr F(F)$ intersects $\cup_{i=1}^k\beta_i$
exactly once.
\end{enumerate}
Such curves $\beta_i$ are called \emph{the dividing curves of $\mathscr F(F)$}.

\begin{rema}
Dividing curves can be defined for any convex surface, even if the characteristic
foliation is not nice. A full collection of dividing curves can be obtained
as follows. Take a generic contact line element field transverse to the surface,
and the points where the field is tangent to the contact structure will cut
the surface along a full family of dividing curves. See details in~\cite{Gi}.
\end{rema}

While the equivalence class of the Giroux graph viewed as a Legendrian graph
can easily change (which will be explored below) when a surface is isotoped
in the class of convex surfaces, the isotopy class of dividing curves
is an invariant of such an isotopy~\cite{Gi}.

\begin{rema}
Legendrian graphs viewed up to transformations that can
occur to the Giroux graph are precisely the subject
of the paper~\cite{pra}, though the relation to Giroux
graphs is not noticed there. We will address this relation
in a subsequent paper.
\end{rema}

We endow the dividing curves with \emph{a canonical orientation} induced by the coorientation
of the contact structure.

\begin{rema}
Typically, one assumes the surface $F$ to be oriented, so the Giroux
graph falls naturally into two disjoint subgraphs, positive and negative,
depending on wether the orientation of the surface coincides
with the orientation of the contact structure at the singularities
of the foliation. Each dividing curve
separates a positive component of $\mathscr G(F)$
from a negative one and thus is endowed with a \emph{co}orientation.

If we consider non-oriented surfaces a dividing curve may be a core of a
M\"obius band, but the canonic orientation of all dividing curves is always
well defined.
\end{rema}

To see an example we look again at the surface $\widehat\Pi^\kappa$,
where $\Pi$ is a rectangular diagram of a surface and $\kappa>0$. An
extended Giroux graph of $\widehat\Pi^\kappa$ can be obtained
by taking $\widehat G$, where $G$ is the set of all vertices of
rectangles in $\Pi$. If $r=[\theta_1,\theta_2]\times[\varphi_1,\varphi_2]$
is a rectangle from $\Pi$, then $\widehat{(\theta_1,\varphi_2)}$
and $\widehat{(\theta_2,\varphi_1)}$ are two edges of the Giroux graph.
The torus projection of a dividing curve passing through $\widehat r^\kappa$
must go from $(\theta_1,\varphi_1)$ to $(\theta_2,\varphi_2)$,
see Fig.~\ref{dividing}.
\begin{figure}[ht]
\includegraphics{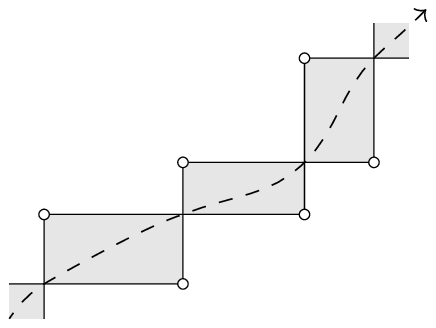}\put(-72,8){$\mathscr G$}\put(-37,28){$\mathscr G$}\put(-17,43){$\mathscr G$}%
\put(-125,33){$\mathscr G$}\put(-85,48){$\mathscr G$}\put(-50,78){$\mathscr G$}\put(0,93){\tiny\parbox{2cm}{dividing\\curve}}
\hskip1cm
\includegraphics{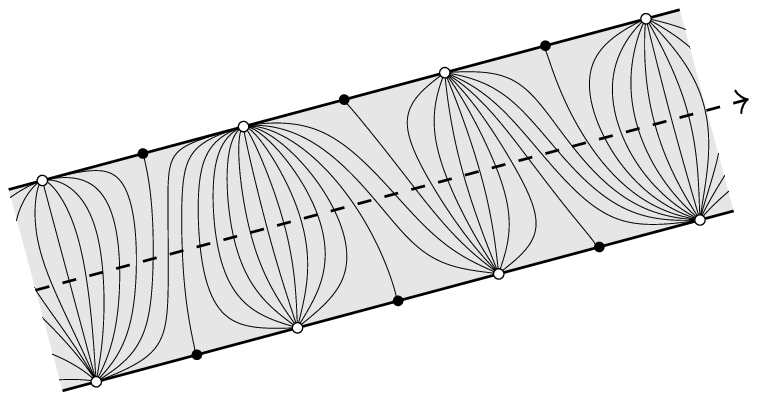}\put(-100,25){$\mathscr G$}\put(-120,95){$\mathscr G$}\put(0,90){\tiny\parbox{2cm}{dividing\\curve}}
\caption{The Giroux graph and a dividing curve on a torus projection (left) and the surface $\widehat\Pi^\kappa$ itself (right)}\label{dividing}
\end{figure}

\subsection{Proof of part~(ii) of Theorem~\ref{main2}}
We prove a more general statement of which Theorem~\ref{main2}
is a partial case. Instead of a link we will consider a general
Legendrian graph to which the surface is attached in a `nice' way,
which we now define.

\begin{defi}\label{properly}
Let $\Gamma$ be a Legendrian graph and $F$ a surface with corners.
We say that \emph{$F$ is properly attached to $\Gamma$} if $F\cap\Gamma$
is a sublink of $\partial F$ and for any $p\in\partial F$
the projection of $F\cup\Gamma$ to the contact plane $\xi_+(p)$
along a vector field transverse to $F$ is locally an injection.
\end{defi}

\begin{defi}
Let $\Gamma_0$, $\Gamma_1$ be Legendrian graphs and $F_0$, $F_1$
convex surfaces with corners that are properly attached to $\Gamma_0$ and $\Gamma_1$,
respectively. We say that the pairs $(F_0,\Gamma_0)$ and $(F_1,\Gamma_1)$
\emph{are equivalent} if there is an isotopy from $F_0\cup\Gamma_0$
to $F_1\cup\Gamma_1$ whose restriction to $\Gamma_0$
satisfies conditions~(a)--(d) from Proposition~\ref{equivalent-equivalence},
and to $F_0$ conditions from Definition~\ref{convex-equiv}.
\end{defi}

Let $F$ be a surface with corners properly attached to a Legendrian graph $\Gamma$.
Let $\gamma\subset\partial F$ be a simple arc such that $\partial\gamma\in\overline{\Gamma\setminus\partial F}$.
Then $F$ is tangent to $\xi_+$ at the endpoints of $\gamma$. So, when $p$ traverses $\gamma$
the angle of rotation of $\xi_+(p)$ relative to $T_pF$ will be a multiple of $\pi$.
We call this angle divided by $2\pi$ \emph{the Thurston--Bennequin number of $\gamma$
relative to $F$ and $\Gamma$} and denote by $\tb_+(\gamma;F,\Gamma)$.

If $\gamma\subset\mathbb S^3$ is an arc or simple closed curve of the form $\widehat X$
with $X$ a finite subset of $\mathbb T^2$, then $|X|$ will be referred to as \emph{the length of $\gamma$}.

\begin{prop}\label{pigamma}
Let $F$ be a convex surface properly attached to a Legendrian graph $\Gamma$.
Let $G$ be a rectangular diagram of a graph such that $\widehat G$ and $\Gamma$
are equivalent. Suppose that, for any arc $\gamma\subset\partial F\cap\Gamma$
such that $\partial\gamma\in\overline{\Gamma\setminus\partial F}$,
the respective arc in $\widehat G$ has length not smaller than
$-2\tb_+(\gamma;F,\Gamma)$. Suppose the same holds for
any connected component $\gamma$ of the link $\Gamma\cap\partial F$
{\rm(}one should use $\tb_+(\gamma;F)$ instead of $\tb_+(\gamma;F,\Gamma)${\rm)}.

Then there exists a rectangular diagram of a surface $\Pi$
such that the pairs $(\widehat\Pi,\widehat G)$ and $(F,\Gamma)$ are equivalent.
\end{prop}

\begin{proof}

We will find an isotopy from the given surface $F$ to a `rectangular' surface.
The isotopy will be defined simultaneously for $\Gamma$ so as to
realize the equivalence between $(F,\Gamma)$ and $(\widehat\Pi,\widehat G)$.

\smallskip\noindent\emph{Step 1.}
Deform $F$ slightly keeping $\Gamma$ and $\partial F$ fixed so that $\mathscr F(F)$ becomes very nice
and $F$ remain convex during the deformation. We can do this by Proposition~\ref{veryniceprop}.

\smallskip\noindent\emph{Step 2.} Eliminate all saddle connections and also $-1$-arcs in $\partial F$
connecting a saddle to a node.

If a saddle connection occurs in the interior of $F$, then it is eliminated by a generic $C^1$-small
perturbation of $F$, so there is no problem to keep $F$ within the class of convex surfaces
during the deformation.

If a saddle connection occurs at $\partial F$, and it is a $0$-arc,
then it can be eliminated by twisting $F$ around this arc. By Proposition~\ref{no1-arcs}, to see that the
surface remains convex during the deformation, which is only $C^0$-small,
it suffices to show that: (i) the foliation remains very nice all the time except
a single moment when a multiple saddle is created; (ii) no $1$-arc is produced.
The former is easy to ensure.
The latter is true if the deformation is supported in a small
enough neighborhood of the $0$-arc because all leaves that come close to the scene approach a
node at the other end, so, they cannot eventually turn into a $1$-arc (which
must connect two saddles) or a closed cycle. The evolution of
the foliation is shown in Fig.~\ref{elimsaddlecon} (a).
\begin{figure}[ht]
\begin{tabular}{cp{2cm}c}
\includegraphics{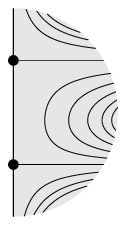}\put(-40,33){$\scriptscriptstyle0$}\hskip0.5cm\raisebox{30pt}{$\longrightarrow$}\hskip0.5cm
\includegraphics{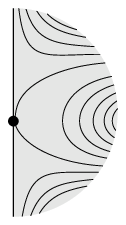}\hskip0.5cm\raisebox{30pt}{$\longrightarrow$}\hskip0.5cm
\includegraphics{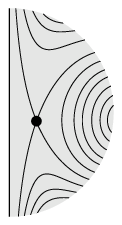}
&\hskip2cm&
\includegraphics{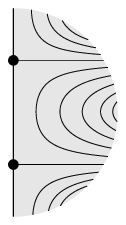}\put(-46,33){$\scriptscriptstyle-1$}\put(-30,33){$\scriptscriptstyle-1$}%
\put(-30,22){$\scriptscriptstyle0$}\put(-30,52){$\scriptscriptstyle0$}
\hskip0.5cm\raisebox{30pt}{$\longrightarrow$}\hskip0.5cm
\includegraphics{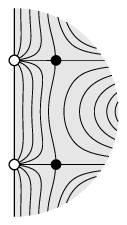}\\
(a)&&(b)
\end{tabular}
\caption{Eliminating saddle connections at $\partial F$}\label{elimsaddlecon}
\end{figure}

Now assume that we have a saddle connection in $\partial F$ which
is a $-1$-arc. Denote it by $\gamma$. Let $s_1$ and~$s_2$ be the saddles connected by $\gamma$.
Leaves of $\mathscr F(F)$ passing by close enough to $\gamma$ must proceed to nodes
at both ends. So, they are also $-1$-arcs. This implies that two separatrices
from $s_1$ and~$s_2$ lying in the interior of $F$ are $0$-arcs. A small twist
of the surface around a portion of these separatrices near $s_1$ and~$s_2$
will shift the saddles off the boundary creating two boundary nodes instead,
see Fig.~\ref{elimsaddlecon}~(b).

$-1$-arcs in $\partial F$ connecting a saddle to a node are eliminated
similarly.

Now the Giroux graph of $\mathscr F(F)$ has a very simple structure. Leaves
passing by a saddle that are not separatrices must be $-1$-arcs. It can be seen from this that
two separatrices emanating from it in opposite directions must be $0$-arcs
and the other two separatrices (or just one if the saddle lies at the boundary)
$-1$-arcs, see Fig.~\ref{saddles2}.
\begin{figure}[ht]
\includegraphics[scale=1.3]{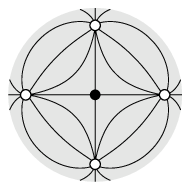}\put(-27,20){\tiny$-1$}\put(-27,55){\tiny$-1$}\put(-62,20){\tiny$-1$}\put(-62,55){\tiny$-1$}%
\put(-52,40){\tiny$-1$}\put(-37,40){\tiny$-1$}\put(-43,45){\tiny$0$}\put(-43,30){\tiny$0$}
\hskip1cm\includegraphics[scale=1.3]{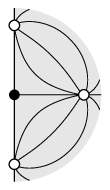}\put(-27,20){\tiny$-1$}\put(-27,55){\tiny$-1$}%
\put(-37,40){\tiny$-1$}\put(-43,45){\tiny$0$}\put(-43,30){\tiny$0$}\\
\caption{Foliation around saddles after simplification}\label{saddles2}
\end{figure}

Thus, the Giroux graph consists of nodes, which are vertices, connected
by smooth arcs each composed of two $0$-arcs and
a saddle in the middle, which are edged, see Fig.~\ref{simplegraph}.
\begin{figure}[ht]
\includegraphics{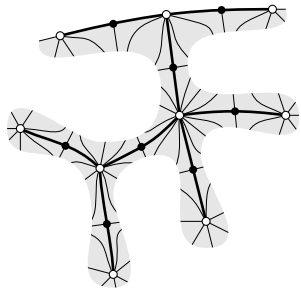}
\caption{Simple Giroux graph}\label{simplegraph}
\end{figure}
We will express this by saying that the Giroux graph is \emph{simple}.
From now on we will
keep the Giroux graph simple except at finitely many moments when an elliptic-saddle
pair is created or eliminated by an application of Giroux's elimination lemma.

\smallskip\noindent\emph{Step 3.}
Now we've got some freedom, which we use to take $\Gamma$ to $\widehat G$.

By varying $F$
in the class of convex surfaces and using Giroux' elimination lemma we can: subdivide edges of the Giroux graph;
subdivide a $-1$-arc
into a smaller $-1$-arc and two $0$-arcs whose closure become a new edge of the Giroux graph;
eliminate two-valent vertices of the Giroux graph;
eliminate one-valent vertices together with the adjacent edges;
move vertices of $\mathscr G$ along the boundary unless
they belong to $\overline{\Gamma\setminus\partial F}\cap\partial F$.
All this can be done without altering $\partial F\cup\Gamma$.
Here is one more thing we can do:

\begin{lemm}\label{giroux-adjustment}
Let $\Gamma'$ and $\mathscr G'$ be Legendrian graphs such that
$\Gamma'\cup\mathscr G'$ is obtained by an angle adjustment
from
$\Gamma\cup\mathscr G$ so that $\Gamma$ is taken to $\Gamma'$
and $\mathscr G$ to $\mathscr G'$.
Then there exists a convex surface with corners~$F'$ such
that the pair $(F,\Gamma)$ is equivalent to $(F',\Gamma')$, and $\mathscr G'$
is the Giroux graph of $F'$.
\end{lemm}

\begin{proof}
The idea of the proof is pretty much the same as that of Proposition~\ref{adjustment}.
Consider again a small cylindrical neighborhood of an elliptic singularity $p$ of $\mathscr F(F)$
with a local coordinate system $x,y,z$ such that $p=(0,0,0)$ and $dz+x\,dy-y\,dx$ is a standard
contact form. The intersection of $F\cup\Gamma$ with a small neighborhood of $p$ can be recovered uniquely from the projection of $\Gamma$
and of the foliation~$\mathscr F(F)$ to the $xy$-plane and the knowledge of
the rest of $F$ and $\Gamma$.

The changes of the $xy$-projection of $\mathscr F(F)$ and $\Gamma$ that can occur as the result
of a small allowed deformation of $F\cup\Gamma$ include those that
can be achieved by applying an area preserving homeomorphism~$\phi$
of the $xy$-plane such that:
\begin{enumerate}
\item
$\phi$ is identity outside of a small neighborhood
of $p$ and $C^0$-close to identity in this neighborhood;
\item
$\phi$ is smooth outside of the origin and has definite
directional derivatives at the origin.
\end{enumerate}
Note that if $p$ lies in the interior of $F$ non-smoothness of $\phi$ does not mean
that the change of $F$ cannot be realized by a smooth isotopy.
We don't need the isotopy to be leafwise, so, each point of
$F$ can simply move along the $z$-axis while angles
between the leaves will be changing.

Such homeomorphisms $\phi$ have enough freedom to realize
an angle adjustment between finitely many leaves of $\mathscr F(F)$
and edges of $\Gamma$ at $p$.
\end{proof}

By a $C^1$-small deformation of $F$ near the points
from $\overline{\Gamma\setminus F}\cap\partial F$
we can make these points nodes of~$\mathscr F(F)$. Then
we use Lemma~\ref{giroux-adjustment} to smoothen the boundary of $F$.
The graphs $\Gamma$ and $\widehat G$ are still equivalent, so,
we can obtain $\widehat G$ from $\Gamma$ by an angle adjustment followed
by a contactomorphisms. Angle adjustment is necessary only at the points of
$\Gamma$ that will be mapped to vertices of $\widehat G$
(lying at $\mathbb S^1_{\tau=0}\cup\mathbb S^1_{\tau=1}$).

Let $q_1,\ldots,q_N\in\Gamma$ be these points.
The bound on the lengths of certain arcs
and simple closed curves in~$\widehat G$
included in the assumption of Proposition~\ref{pigamma} guarantees that there
are enough points $q_i$ on respective arcs and simple
closed curves in~$\Gamma$
to allow us to move the boundary nodes of $\mathscr F(F)$ along $\partial F$
so that any connected component of $\partial F\setminus\{q_1,\ldots,q_N\}$
has a non-empty intersection with at most one $-1$-arc.
We can also make all $q_i$'s nodes either by moving
the existing nodes along $\partial F$ or creating new ones.

Now we use Lemma~\ref{giroux-adjustment} again to do
an angle adjustments in $q_i$'s after which that the graph $\Gamma$
can be taken to $\widehat G$ by a contactomorphism. By Eliashberg's
theorem~\cite{Eli} the group of contactomorphisms of $\mathbb S^3$
is connected.
So, $(F,\Gamma)$ can be deformed in the allowed manner so
that we get $\Gamma=\widehat G$, which is assumed in the sequel.

\smallskip\noindent\emph{Step 4.}
Eliminate unnecessary nodes at $\partial F$.

Each arc of the form $\widehat v$, $v\in\mathbb T^2$,
in $\Gamma$ contains at most one $-1$-arc and
some number of $0$-arcs. All nodes in $\interior(\widehat v)$
should be eliminated, thus making $\interior(\widehat v)$
a single $0$- or $-1$-arc. Note that this does not require
$F$ to be altered near the endpoints of $\widehat v$.

\smallskip\noindent\emph{Step 5.} Choose an extended
Giroux graph $\widetilde{\mathscr G}$ and make it `rectangular'.

We choose $\widetilde{\mathscr G}$ by adding to $\widetilde{\mathscr G}$
some $-1$-arcs connecting nodes. Those nodes may appear
only at $\mathbb S^1_{\tau=0}\cup\mathbb S^1_{\tau=1}$.
Then we find a rectangular diagram of a graph $G_1$ such that
the graph $\overline{\widetilde{\mathscr G}\setminus\Gamma}$ is equivalent
to $\widehat{G_1}$ relative to $\Gamma=\widehat G$.
It exists by Proposition~\ref{leg-appr-pcinciple}.
In the same manner as in Step~3
we can deform $\overline{\widetilde{\mathscr G}\setminus\Gamma}$ to $\widehat{G_1}$
simultaneously with $F$ so that $\Gamma$ is preserved.
To do so we may need to subdivide some $-1$-arcs that
we included in $\widetilde{\mathscr G}$. This will create new $0$-arcs,
which will then be included in $\mathscr G$.

Thus, from now on we assume that $\Gamma=\widehat G$ and
that the extended Giroux graph $\widetilde{\mathscr G}$ has the form $\widehat{G'}$
for some $G'\subset\mathbb T^2$.
We eliminate all unnecessary nodes in $\widetilde{\mathscr G}$,
so that each arc of the form $\widehat v$, $v\in G'$, becomes a $0$-arc or a $-1$-arc.

\smallskip\noindent\emph{Step 6.} Deform $F$ to a surface of the form $\widehat\Pi^\kappa$.

Pick a $\kappa>0$.
We twist $F$ around the boundary so as to make it tangent
to $\xi_+^\kappa$ along $0$-arcs in $\widetilde{\mathscr G}$
and to~$\xi_-^\kappa$ along $-1$-arcs (the plane fields $\xi_\pm^\kappa$
are defined by~\eqref{xi-kappa} and \eqref{xi+kappa}). This can be done
without creating or eliminating any singularity because
the tangent plane to $F$ and the corresponding field $\xi_\pm^\kappa$
already have similar `qualitative behavior' along each edge of $\widetilde{\mathscr G}$.

Now by Proposition~\ref{isotopy-to-rectangular},
in which we can seamlessly replace $\widehat\Pi$ by $\widehat\Pi^\kappa$,
there exists a rectangular diagram of a surface~$\Pi$ and
an isotopy from $F$ to $\widehat\Pi^\kappa$ that
is fixed at $\Gamma\cup\widetilde{\mathscr G}$
and preserves the tangent plane to the surface at any point $p\in\widetilde{\mathscr G}$.

We recall that $F\setminus\widetilde{\mathscr G}$ is a union of open discs,
hence so is $\widehat\Pi^\kappa\setminus\widetilde{\mathscr G}$. Moreover,
both surfaces $F$ and $\widehat\Pi^\kappa$ are convex,
and all the discs in $F\setminus\widetilde{\mathscr G}$ and
$\widehat\Pi^\kappa\setminus\widetilde{\mathscr G}$ are foliated by $-1$-arcs.

Since the tangent planes to $F$ at $\widetilde{\mathscr G}$ are kept fixed
during the isotopy $F\leadsto\widehat\Pi^\kappa$ the latter
can be decomposed into a sequence of isotopies
\begin{equation}\label{Fseq}F=F_0\leadsto F_1\leadsto\ldots\leadsto F_n=\widehat\Pi^\kappa
\end{equation}
such that each of them alters only one of the discs in $F_i\setminus\widetilde{\mathscr G}$.
Let $D\subset F$
be the closure of one of the discs $F\setminus\widetilde{\mathscr G}$,
and $D'$ the respective disc in $\widehat\Pi^\kappa$.
(Strictly speaking, the boundary of $D$ may not be a simple
curve, hence $D$ may not be a disc, but the reasoning below
can be generalized straightforwardly, so we keep assuming it to be a disc.)
So far we have a sequence of isotopies
\begin{equation}\label{Dseq1}
D=D_0\leadsto D_1\leadsto\ldots\leadsto D_m=D'
\end{equation}
each of which realizes a single isotopy in the sequence~\eqref{Fseq}.
Let $D_{i-1}\leadsto D_i$ correspond to $F_{j_i-1}\leadsto F_{j_i}$.

By decomposing further isotopies in~\eqref{Dseq1} to `smaller' ones if necessary we can
assume, for each $i=1,\ldots,m$, the existence of an open ball $B_i$ such that
\begin{enumerate}
\item
$\overline{B_i}\cap(F_{j_i-1}\cup\Gamma)=D_{i-1}$,
$\overline{B_i}\cap(F_{j_i}\cup\Gamma)=D_i$,
$D_{i-1}\cap\partial B_i=
D_i\cap\partial B_i=\partial D$;
\item
the isotopy $D_{i-1}\leadsto D_i$ takes place inside $B_i$ (i.e.\ no point of the interior of the disc
escapes $B_i$).
\end{enumerate}
Having fixed such balls we are going to modify the isotopies so that these properties
still hold. The new isotopies will then compile into an isotopy $F\leadsto\widehat\Pi^\kappa$.
All isotopies of the discs will be assumed to be fixed at $\widetilde{\mathscr G}$
and to keep unchanged the tangent planes to $F$ at $\widetilde{\mathscr G}$.

We can replace each isotopy $D_{i-1}\leadsto D_i$ by composition of two successive isotopies
$D_{i-1}\leadsto D_i'\leadsto D_i$ so that $D_{i-1}\cap D_i'=D_i\cap D_i'=\partial D$.
This is done by `a back and forth trick': the first isotopy pushes the disc toward $\partial B_i$
(either way of the two) and the second moves it back and then to $D_i$.
Thus, without loss of generality we can assume from the beginning that $D_{i-1}\cap D_i=\partial D$.

Now alter each of $D_i$ if necessary so as to make all these discs convex and
the foliation $\mathscr F(D_i)$ very nice.
This can be done without violating the condition $D_{i-1}\cap D_i=\partial D$ and
so that the alteration of $D_i$ takes place in $B_i\cap B_{i+1}$.

By another similar alteration we can eliminate all singularities of $\mathscr F(D_i)$
from $\interior(D_i)$. There is no obstruction to elimination a node-saddle pair
in $\interior(D_i)$ because
the corresponding deformation occurs in a small neighborhood of a $0$-arc
connecting the pair in $\interior(D_i)$. Any node in $\interior(D_i)$
must be connected by a $0$-arc to a saddle in $\interior(D_i)$, so, we can eliminate all nodes
in $\interior(D_i)$. We claim that this will also eliminate all saddles.

Indeed, $\partial D$ consists of four consecutive arcs $\beta_1,\beta_2,\beta_3,\beta_4$,
say, such that $\beta_1,\beta_3\subset\mathscr G(D)$ and $\interior(\beta_2),\interior(\beta_4)$
are $-1$-arcs. Since the tangent planes at $\partial D$ are untouched during the isotopies
the same is true for $D_i$. If a saddle remains in $\interior(D_i)$ after eliminating the vertices
there must be an edge $\gamma$ of $\mathscr G(D_i)$ connecting two
points in $\beta_1\cup\beta_3$. If~$\gamma$ connects a point in $\beta_1$
to a point in $\beta_3$, then it cuts $D_i$ into halves such that the
boundary of each of them contains exactly one $-1$-arc and some number
of $0$-arcs, which is impossible. If $\gamma$ connects two
points in~$\beta_1$ (or in $\beta_3$), then it cuts off an
overtwisted disc (a disc whose boundary has zero relative Thurston--Bennequin number),
which is also a contradiction by a well known result of Bennequin~\cite{ben}.

Thus, $\mathscr F(D_i)$ has no singularities in $\interior(D_i)$.
It also has no closed leaf since such a leaf would also cut off an overtwisted
disc. Application of Lemma~\ref{discs-lem} proven below completes the construction
of an isotopy~$F\leadsto\widehat\Pi^\kappa$ that
realized the equivalence of $(F,\widehat G)$ and $(\widehat\Pi^\kappa,\widehat G)$.

\smallskip\noindent\emph{Step 7.} Now we isotop $\widehat\Pi^\kappa$ to $\widehat\Pi$
simply by moving $\kappa$ to $0$.
\end{proof}

\begin{lemm}\label{discs-lem}
Let $D_0$, $D_1$ be two discs with corners such that
\begin{enumerate}
\item
the discs $D_0$ and $D_1$ have common boundary and the same tangent planes
along it;
\item
the interiors of $D_0$ and $D_1$ are disjoint;
\item
$D_0$ and $D_1$ approach $\partial D_0=\partial D_1$ from the same side;
\item
the characteristic foliations of both are very nice and have no singularities in the interior.
\end{enumerate}
Then there exists an isotopy $\psi:[0,1]\times D_0\rightarrow\mathbb S^3$ such
that $\psi_0=\mathrm{id}|_{D_0}$, $\psi_1(D_0)=D_1$,
and, for any $t\in(0,1)$, the disc $D_t=\psi_t(D_0)$
is convex, and we have $D_t\cap D_0=D_t\cap D_1=
\partial D_t$.
\end{lemm}

The last condition implies that $\partial D_t$ remains unchanges, and $D_t$
remains `in between' the discs $D_0$ and~$D_1$ and have the same
tangent planes at $\partial D_t$.

\begin{proof}
The proof is based on classical arguments \cite{Eli,gi1,Ge,mas1}
though we could not find exactly this statement (even
for discs with smooth boundary). We will again be sketchy.

First, we pick an isotopy $\psi$ from $D_0$ to $D_1$ such that $D_{t'}\cap D_{t''}=\partial D_0$
for any $t'\ne t''$, where $D_t=\psi_t(D_0)$. In other words, the discs are deformed in the transverse
direction when $t$ changes. By choosing a generic isotopy we can
ensure that all but finitely many discs $D_t$, $t\in(0,1)$ have a nice characteristic foliation,
and the exceptional moments correspond to a birth or death of a node-saddle pair. The
singularity at such a moment moment can be viewed as a saddle of multiplicity~$0$,
so it cannot be a reason for the corresponding disc to be non-convex.

The only possible reasons for $D_t$ not to be convex by Proposition~\ref{no1-arcs} are:
\begin{enumerate}
\item
$D_t$ contains a closed leaf or a cooriented saddle connection cycle, and hence, its
characteristic foliation is not very nice;
\item
$\mathscr F(D_t)$ contains a $1$-arc.
\end{enumerate}
A closed leaf or a cooriented saddle connection cycle cannot appear in a disc because
otherwise it cuts off an overtwisted disc, which is impossible.

So, we are left with a $1$-arc as the only origin of non-convexity. A small transverse
deformation always resolves a $1$-arc, so there can be only finitely
many moments when a $1$-arc occurs. Moreover, to get a $1$-arc
resolution it is enough to make a transverse deformation
in a small neighborhood of any point of the arc.

Let $t=t_0$ be such a moment, and $\gamma$ be a $1$-arc in $D_{t_0}$.
We can assume that it is the only $1$-arc in $D_{t_0}$ as this is achieved by a small
perturbation of the isotopy.
Let $B$ be a ball intersecting $\gamma$ such that $\overline B$
is disjoint from any $0$-arc in $D_{t_0}$ and from $\partial\gamma$.

We can modify our isotopy slightly for the moments $t\in[t_0-2\delta,t_0+2\delta]$,
where $\delta>0$ is small, so that $D_t\setminus B=D_{t_0}\setminus B$
when $t\in[t_0-\delta,t_0+\delta]$ and the deformation of $D_t\cap B$
proceeds in a transverse direction.

Now the point is that we can isotop $D_{t_0-\delta}$ to $D_{t_0+\delta}$
another way through convex discs. To do so, we first eliminate all singularities
in $\interior(D_{t_0-\delta})$ leaving the intersection $D_{t_0-\delta}\cap B$
untouched. This is possible because we need only to modify the disc near $0$-arcs,
which are far from $B$. They are also far from the disc boundary,
so the disc may be kept `in between' the discs $D_0$, $D_1$.

All singularities in the interior of the disc can be eliminated for
the following reason. There are no singularities in $\interior(D_0)$,
which implies that $\mathscr F(D_0)$ has a single dividing curve.
This is equivalent to $\tb_+(\partial D_0;D_0)=-1$, and this
number is preserved during an isotopy. So,
$\mathscr F(D_{t_0-\delta})$ also has a single dividing curve,
hence all singularities in the interior can be eliminated.

When all singularities are eliminated from $\interior(D_{t_0-\delta})$
the $1$-arc is no longer present, so the desired
isotopy inside $B$ can be performed keeping
the disc convex. Then the elimination of singularities is
reversed, and we obtain $D_{t_0+\delta}$.
Doing so for all passages through the $1$-arc obstruction
gives an isotopy from $D_0$ to $D_1$ through convex discs only.
\end{proof}

\subsection{Annuli with Legendrian boundary}\label{annuli}
Convex annuli with Legendrian boundary and
a single dividing curve are building blocks for more general
convex surfaces. Any closed convex surface in $\mathbb S^3$ as well as any
orientable convex surface $F$ with boundary such that $\tb_+(K;F)=0$ for every component
$K$ of $\partial F$ can be composed of such blocks.

Let $A$ be such an annulus. It is bounded by two Legendrian knots that
are obviously isotopic. But are they alway equivalent as Legendrian knots?
If the answer is positive, then a more general statement is true:
if $A$ is an annulus with Legendrian boundary whose components
have zero Thurston--Bennequin numbers relative to $A$,
then they are equivalent as Legendrian knots.

This question was addressed in~\cite{Gos} and the statement
was claimed to be true. However, the argument in~\cite{Gos} is incomplete.
Up to this writing the problem appears to be open.

Theorem~\ref{main2} have the following corollary for convex annuli, which may
be useful to solve the above mentioned problem.

\begin{prop}
Let $A$ be a convex annuli with a single dividing curve. Then it is equivalent
to the convex surface $\widehat\Pi$, where $\Pi=\{r_1,\ldots,r_k=r_0\}$ is a rectangular diagram
such that, for each $i=1,\ldots,k$, the rectangle $r_{i-1}$ shares with $r_i$ a single vertex
that is top right for $r_{i-1}$ and bottom left for $r_i$.
\end{prop}

\begin{proof}
Let $\Pi$ be an arbitrary rectangular diagram of a surface such that
$\widehat\Pi$ is equivalent to $A$. Such a diagram exists by Theorem~\ref{main2}.
The surface $\widehat\Pi$ has a single dividing curve, which means (see the end of Subsection~\ref{gg-sec})
that $\Pi$ has the form $\{r_1,\ldots,r_k=r_0\}$, where the top right vertex of $r_{i-1}$
coincides with the bottom left vertex of $r_i$, $i=1,\ldots,k$. If they share no more
vertices we are done.

If some rectangles in $\Pi$ share two vertices,
we can find an edge $e$ of the diagram $\partial\Pi$ endowed
with the boundary framing induced by $\Pi$ such that
some point $v$ in $e$ is the top left vertex of one rectangle from $\Pi$
and the bottom right for another. Then
we can `cut' the diagram at the `wrong' vertex and get a diagram
with the same number of rectangles but longer boundary, see Fig.~\ref{cut}.
\begin{figure}[ht]
\begin{tabular}{ccccccc}
\hbox to 0pt{\hss\includegraphics{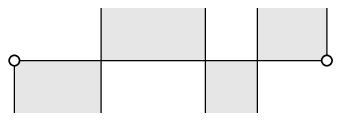}\hss}&&&\hbox to 0pt{\hss\raisebox{18pt}{$\longrightarrow$}\hss}&&&
\hbox to 0pt{\hss\includegraphics{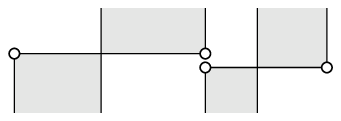}\hss}\\
\includegraphics{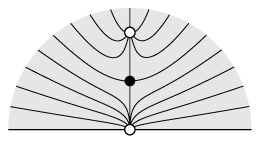}&\raisebox{18pt}{$\rightarrow$}&
\includegraphics{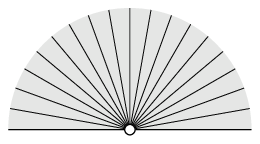}&\raisebox{18pt}{$\rightarrow$}&
\includegraphics{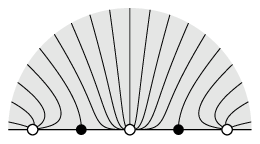}&\raisebox{18pt}{$\rightarrow$}&
\includegraphics{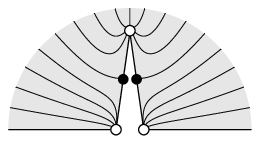}
\end{tabular}
\caption{Simplifying an annulus}\label{cut}
\end{figure}
After finitely
many such cuts no two rectangles will share more than one vertex.

We claim that such cuts produce an equivalent convex surface. To see this
we switch from $\widehat\Pi$ to $\widehat\Pi^\kappa$ with $\kappa>0$.
For $\widehat\Pi^\kappa$ the discussed operation means cutting along
an edge of the Giroux graph such that exactly one endpoint of this edge
lies at the boundary of the annulus. The same result
can be obtained by a continuous deformation consisting of
four stages: elimination of a node-saddle pair in the interior of
the annulus, creation of two new node-saddle pairs at the boundary,
angle adjustment, contact isotopy. This is sketched in Fig.~\ref{cut}.
\end{proof}

\begin{conj}
The two boundary components of the annulus represented
by the rectangular diagram in Fig.~\ref{example} are not
equivalent as Legendrian knots.
\begin{figure}[ht]
\includegraphics[width=250pt]{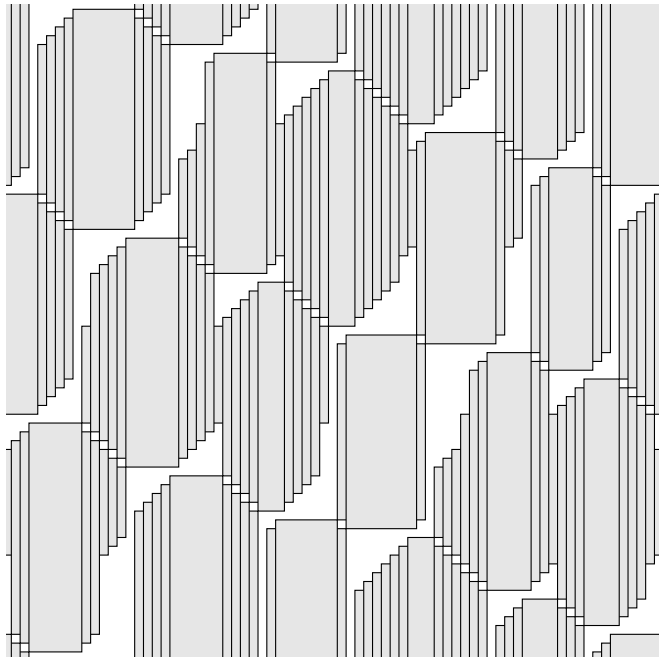}
\caption{A rectangular diagram of an annulus}\label{example}
\end{figure}
\end{conj}

To construct the example we used a generalization of rectangular diagrams of surfaces that corresponds to branched
surfaces (in the sense of~\cite{fo}), which in our particular case have the form of an $I$-bundle over a train track. Such presentation is
more compact and allows us to work with infinite series of surfaces each series presented by a  single `branched' diagram.

Author's experience shows that to find simpler examples for which the equivalence of the two boundary components would not be
obvious is pretty hard. It happens very often that the knots can be transformed into each other by exchange moves.
The example in Fig.~\ref{example} is not such, both components of the boundary~$\partial\Pi$
are rigid in a sense that they don't admit any exchange move. This fact may be used in the future
to prove that they are not equivalent, since establishing that some kind of `rectangular' surfaces
does not exist for the given knot becomes much easier if the diagram is rigid. (For instance,
it requires very little work to show that the knots in our example are hyperbolic.)

The known algebraic
invariants do not seem to help to distinguish our knots as most of them are
just too hard to compute, whereas the Pushkar--Chekanov invariant~\cite{PC} is computable
at some effort but vanishes for both.

\end{document}